\documentclass[reqno,12pt, centertags]{amsart}
\usepackage{amsmath,amsthm,amscd,amssymb}
\usepackage{upref}
\usepackage{latexsym}
\usepackage{lscape}
\usepackage{graphicx}
\usepackage{color}

\makeatletter
\def\theequation{\@arabic\c@equation}

\newcommand{\mas}{\operatorname{Mas}}

\newcommand{\gaD}{\gamma_{{}_D}}

\newcommand{\Mas}{\operatorname{Mas}}
\newcommand{\Mor}{\operatorname{Mor}}

\newcommand{\beq}{\begin{equation}}
\newcommand{\enq}{\end{equation}}
\newcommand{\mor}{\operatorname{Mor}}

\newcommand{\bbR}{{\mathbb{R}}}
\newcommand{\R}{{\mathbb{R}}}

\newcommand{\bbC}{{\mathbb{C}}}

\newcommand{\cB}{{\mathcal B}}

\newcommand{\cX}{{\mathcal X}}

\newcommand{\no}{\nonumber}
\newcommand{\lb}{\label}

\newcommand{\rank}{\text{\rm{rank}}}
\newcommand{\ran}{\text{\rm{ran}}}

\newcommand{\bi}{\bibitem}

\newcommand{\dist}{\operatorname{dist}}

\newcommand{\mi}{\operatorname{Mas}}
\newcommand{\mo}{\operatorname{Mor}}

\numberwithin{equation}{section}

\newcommand{\dom}{\operatorname{dom}}

\newcommand{\Sp}{\operatorname{Spec}}

\renewcommand{\ker}{\operatorname{ker}}

\theoremstyle{plain}
\newtheorem{theorem}{Theorem}[section]

\newtheorem{lemma}[theorem]{Lemma}

\newtheorem{proposition}[theorem]{Proposition}

\theoremstyle{definition}
\newtheorem{definition}[theorem]{Definition}

\newtheorem{remark}[theorem]{Remark}

\newtheorem{claim}[theorem]{Claim}

\allowdisplaybreaks

\title{The Maslov and Morse indices for Schr\"odinger operators on $[0, 1]$}

\author[P.\ Howard and A.\ Sukhtayev]{P.\ Howard and A.\ Sukhtayev}
\address{Mathematics Department,
Texas A\&M University, College Station, TX 77843, USA}

\email{phoward@math.tamu.edu}
\email{alim@math.tamu.edu}
\date{\today}

\begin{document}

\begin{abstract} 
Assuming a symmetric potential and separated self-adjoint boundary conditions, 
we relate the Maslov and Morse indices for Schr\"odinger operators 
on $[0, 1]$. We find that the Morse index can be computed in terms of 
the Maslov index and two associated matrix eigenvalue problems. This provides
an efficient way to compute the Morse index for such operators.  
\end{abstract}

\maketitle

\section{Introduction} \lb{s1}

We consider eigenvalue problems 
\begin{equation} \label{eq:hill}
\begin{aligned}
Hy := -y'' + V(x)y &= \lambda y \\
\alpha_1y(0)+\alpha_2y'(0) &= 0 \\
\beta_1y(1)+\beta_2y'(1)&=0,
\end{aligned}
\end{equation}
where $y \in \mathbb{R}^n$, $V \in C([0, 1])$ is a symmetric  
matrix in $\mathbb{R}^{n \times n}$, and $\alpha_1$, $\alpha_2$,
$\beta_1$, and $\beta_2$ are real-valued $n \times n$ matrices 
such that 
\begin{alignat}{2} 
\rank\begin{bmatrix} \alpha_1 & \alpha_2\end{bmatrix}&= n; 
&\qquad \rank\begin{bmatrix} \beta_1 & \beta_2\end{bmatrix}&=n, \label{RN} \\
\alpha_1\alpha^t_2-\alpha_2\alpha^t_1 &= 0_{n\times n};
&\qquad \beta_1\beta^t_2-\beta_2\beta^t_1&=0_{n\times n}, \label{4.4}
\end{alignat}
where we use superscript $t$ to denote matrix transpose, anticipating 
the use of superscript $T$ to denote transpose in a complex Hilbert
space described below. If \eqref{RN}--\eqref{4.4} hold then without
loss of generality we can take
\begin{equation}\lb{4.4new}
\begin{aligned}
 \alpha_1\alpha^t_1+\alpha_2\alpha^t_2&=I, \\
\beta_1\beta^t_1+\beta_2\beta^t_2&=I
\end{aligned}
\end{equation}
(see, for example, \cite[page 108]{K}).

In particular, we are interested in counting the number of negative
eigenvalues for $H$ (i.e., the Morse index). We proceed by relating 
the Morse index to the Maslov index, which is described in 
Section \ref{maslov_section}. In essence, we'll find that the 
Morse index can be computed in terms of the Maslov index, and that
while the Maslov index is less elementary than the Morse index, 
it's relatively straightforward to compute in the current setting.

The Maslov index has its origins in the work of V. P. Maslov 
\cite{Maslov1965a} and subsequent development by V. I. Arnol'd
\cite{arnold67}. It has now been studied extensively, both 
as a fundamental geometric quantity \cite{BF98, CLM, F, P96, rs93}
and as a tool for counting the number of eigenvalues on specified
intervals \cite{BJ1995, BM2015, CDB09, CDB11, Chardard2009, 
CJLS2014, DJ11, FJN03, J88, J88a, JM2012}. In this latter context, 
there has been a strong resurgence of interest following the 
analysis by Deng and Jones (i.e., \cite{DJ11}) for multidimensional
domains. Our aim in the current analysis is to rigorously develop
a relationship between the Maslov index and the Morse index in 
the relatively simple setting of (\ref{eq:hill}), and to take
advantage of this setting to compute the Maslov index directly 
for example cases so that these properties can be illustrated 
and illuminated. Our approach is adapted from \cite{CJLS2014, DJ11},

As a starting point, we define what we will mean by a {\it Lagrangian
subspace}.

\begin{definition} \label{lagrangian_subspace}
We say $\ell \subset \mathbb{R}^{2n}$ is a Lagrangian subspace
if $\ell$ has dimension $n$ and
\begin{equation*}
(Jx, y)_{\mathbb{R}^{2n}} = 0, 
\end{equation*} 
for all $x, y \in \ell$. Here, $(\cdot, \cdot)_{\mathbb{R}^{2n}}$ denotes
Euclidean inner product on $\mathbb{R}^{2n}$, and  
\begin{equation*}
J = 
\begin{pmatrix}
0 & -I_n \\
I_n & 0
\end{pmatrix},
\end{equation*}
with $I_n$ the $n \times n$ identity matrix. We sometimes adopt standard
notation for symplectic forms, $\omega (x,y) = (Jx, y)_{\mathbb{R}^{2n}}$.
\end{definition}

A simple example, important for intuition, is the case $n = 1$, for which 
$(Jx, y)_{\mathbb{R}^{2}} = 0$ if and only if $x$ and $y$ are linearly 
dependent. In this case, we see that any line through the origin is a 
Lagrangian subspace of $\mathbb{R}^2$. As a foreshadowing of further 
discussion, we note that each such Lagrangian subspace can be identified
with precisely two points on the unit circle $S^1$. 

More generally, any Lagrangian subspace of $\mathbb{R}^{2n}$ can be
spanned by a choice of $n$ linearly independent vectors in 
$\mathbb{R}^{2n}$. We will generally find it convenient to collect
these $n$ vectors as the columns of a $2n \times n$ matrix $\mathbf{X}$, 
which we will refer to as a {\it frame} for $\ell$. 

Lagrangian subspaces arise naturally in the current setting if we consider
the shooting problem in which we evolve forward the family of solutions 
of (\ref{eq:hill}) that satisfy only the left boundary condition 
(i.e., the condition at $0$). In this setting, it will be natural to 
view (\ref{eq:hill}) as a first order system with $p = y$, $q = y'$,
and $\mathbf{p} = {p \choose q}$. We obtain 
\begin{equation} \label{first_order}
\frac{d \mathbf{p}}{dx} = \mathbb{A} (x; \lambda) \mathbf{p},
\end{equation}   
where 
\begin{equation*}
\mathbb{A} (x; \lambda)
=
\begin{pmatrix}
0 & I_n \\
-\lambda I_n + V & 0
\end{pmatrix}.
\end{equation*}

Let $\{\mathbf{p}_j (x)\}_{j=1}^n = \{{p_j (x) \choose q_j (x)}\}_{j=1}^n$ 
denote any collection of $n$ linearly 
independent vectors in $\mathbb{R}^{2n}$ satisfying the left boundary 
conditions  
\begin{equation*} 
\alpha_1 p_j (0) + \alpha_2 q_j (0) = 0 \quad \forall j \in \{1, 2, \dots, n\},
\end{equation*}
and evolving according to (\ref{first_order}). For example, using 
(\ref{4.4}) we can take the vectors $\{p_j (0)\}_{j=1}^n$ to be the 
columns of $\alpha_2^t$, and likewise the vectors $\{q_j (0)\}_{j=1}^n$
to be the columns of $-\alpha_1^t$. We denote by $X (x)$ the 
$n \times n$ matrix obtained by taking each ${p}_j (x)$ as a column, 
and we denote by $Z(x)$ the $n \times n$ matrix 
obtained by taking each $q_j (x)$ as a column. 
We will verify in Theorem \ref{th:lagrange} that the $2n \times n$ 
matrix $\mathbf{X} := {X \choose Z}$ is the frame for a Lagrangian 
subspace that we will denote $\ell (x, \lambda)$. Notice that 
$\ell (x, \lambda)$ varies as $x$ and $\lambda$ vary, and in particular
if we choose any path $\Gamma$ in the $x$-$\lambda$ plane we can 
consider the evolution of $\ell$ along this path. 

Continuing to view this process as a shooting argument, we can take as our 
{\it target} the Lagrangian subspace associated with the boundary condition 
at $x = 1$. It's clear that if $\ell (1, \lambda)$ intersects this 
Lagrangian subspace then $\lambda$ is an eigenvalue of $H$, and also 
that the geometric multiplicity of $\lambda$ corresponds precisely with
the dimension of intersection. In order to clarify the nature of this 
target space, we let    
$\{\mathbf{p}_j^{(1)}\}_{j=1}^n = \{{{p}_j^{(1)} \choose {q}_j^{(1)}}\}_{j=1}^n$
denote any collection of $n$ linearly independent (constant) vectors satisfying 
the right boundary conditions 
\begin{equation*} 
\beta_1 p_j^{(1)} + \beta_2 q_j^{(1)} = 0 \quad \forall j \in \{1, 2, \dots, n\}.
\end{equation*}
For example, we see from (\ref{4.4}) that we can take the vectors 
$\{p_j^{(1)} \}_{j=1}^n$ to be the columns of $\beta_2^t$, and likewise the 
vectors $\{q_j^{(1)}\}_{j=1}^n$ to be the columns of $-\beta_1^t$.
Let $X_1$ denote the $n \times n$ matrix comprising $\{p_j^{(1)}\}_{j=1}^n$
as its columns, and let $Z_1$ denote the $n \times n$ matrix comprising 
$\{q_j^{(1)}\}_{j=1}^n$ as its columns. We see that 
$\mathbf{X}_1 := {X_1 \choose Z_1}$ is a frame
for the Lagrangian subspace $\ell_1$ that can be viewed as our target.

We can now ask the following questions: (1) as $\ell (x, \lambda)$ 
evolves, for what values of $x$ and $\lambda$ does it intersect 
$\ell_1$?; (2) what is the dimension of these intersections?; and 
(3) what is the direction of these intersections? Geometrically, the 
Maslov index is precisely a count of these intersections, including 
both multiplicity and direction. 

We will find it productive to fix $s_0 > 0$ (taken sufficiently small
during the analysis) and $\lambda_{\infty} > 0$ (taken sufficiently
large during the analysis), and to consider the rectangular path 
\begin{equation*}
\Gamma = \Gamma_1 \cup \Gamma_2 \cup \Gamma_3 \cup \Gamma_4,
\end{equation*}
where the paths $\{\Gamma_i\}_{i=1}^4$ are depicted in 
Figure \ref{F1}.

\begin{figure}[h]
 \scalebox{1.25}{
\begin{picture}(100,100)(-20,0)
\put(-75,-2){$-\lambda_{\infty}$}
\put(-65,8){\line(0,1){4}}
\put(80,5){\vector(0,1){95}}
\put(-65,20){\line(0,1){60}}
\put(-65,80){\vector(0,-1){40}}
\put(-74,10){\vector(1,0){190}}
\put(70.5,40){\text{\tiny $\Gamma_2$}}
\put(-63,60){\text{\tiny $\Gamma_4$}}
\put(-82,24){\rotatebox{90}{\text{\tiny no conjugate}}}
\put(-75,34){\rotatebox{90}{\text{\tiny points}}}
\put(84,34){\rotatebox{90}{\text{\tiny conjugate}}}
\put(91,40){\rotatebox{90}{\text{\tiny points}}}
\put(45,73){\text{\tiny $\Gamma_3$}}
\put(43,13){\text{\tiny $\Gamma_1$}}
\put(100,12){$\lambda$}
\put(83,98){$s$}
\put(80,20){\vector(0,1){30}}
\put(83,-5){$0$}
\put(-65,20){\line(1,0){145}}
\put(-10,20){\vector(1,0){50}}
\put(-65,80){\line(1,0){145}}
\put(80,80){\vector(-1,0){55}}
\put(82,78){$1$}
\put(82,18){$s_0 $}
\put(70,20){\circle*{4}}
\put(80,60){\circle*{4}}
\put(80,70){\circle*{4}}
\put(20,80){\circle*{4}}
\put(40,80){\circle*{4}}
\put(60,80){\circle*{4}}
\put(-5,87){{\tiny \text{$H$-eigenvalues}}}
\put(-60,24){{\tiny \text{$V(0) - (P_{R_0} \Lambda_0 P_{R_0})^2,B$-eigenvalues}}}
\end{picture}}
\caption{Schematic of the path $\Gamma = \Gamma_1 \cup \Gamma_2 \cup \Gamma_3 \cup \Gamma_4$}.\label{F1}
\end{figure}
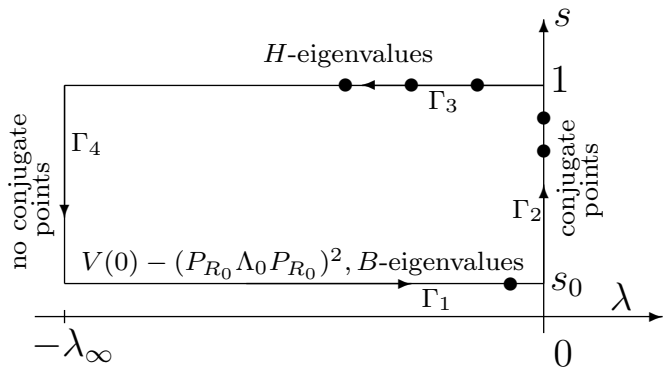 

As discussed, for example, in \cite{CLM}, the Maslov index enjoys path 
additivity so that 
\begin{equation*}
\text{Mas} (\ell, \ell_1; \Gamma) = 
\text{Mas} (\ell, \ell_1; \Gamma_1) 
+ \text{Mas} (\ell, \ell_1; \Gamma_2)
+ \text{Mas} (\ell, \ell_1; \Gamma_3)
+ \text{Mas} (\ell, \ell_1; \Gamma_4).
\end{equation*}  
In addition, the Maslov index is homotopy invariant,
and it follows immediately that the Maslov index around
any closed path will be 0, so that 
\begin{equation*}
\text{Mas} (\ell, \ell_1; \Gamma) = 0.
\end{equation*}

Our analysis is primarily concerned with understanding 
each of the four quantities 
\begin{equation*}
\{\text{Mas} (\ell, \ell_1; \Gamma_i)\}_{i=1}^4.
\end{equation*} 
As a start, we note that in the setting of eigenvalue 
problems such as (\ref{eq:hill}) it's natural to view
the Maslov index along $\Gamma_2$ as a distinguished 
value, and we will designate it the {\it Principal Maslov
Index}. In our setting, this is a readily computable 
quantity, and we will develop a framework for computing 
it, and compute values of it in particular cases. 

We will show that $\text{Mas} (\ell, \ell_1; \Gamma_3)$
is precisely the Morse index of $H$ that we're trying 
to compute, and that given any $0< s_0 < 1$, $\lambda_{\infty} > 0$ 
can be chosen sufficiently large so that  
$\text{Mas} (\ell, \ell_1; \Gamma_4) = 0$. In the case
of Dirichlet boundary conditions we'll find that $s_0$ can
be chosen sufficiently small so that 
$\text{Mas} (\ell, \ell_1; \Gamma_1) = 0$, in which case
we get the very simple relationship 
\begin{equation*}
\text{Mor} (H) = - \text{Mas} (\ell, \ell_1; \Gamma_2).
\quad \quad \tag{\text{Dirichlet case}}
\end{equation*}
More generally, we can have crossings along the bottom 
shelf (i.e., $\Gamma_1$), and in order to efficiently 
characterize these we'll adapt an elegant theorem
from \cite{BK} (see also an earlier version in 
\cite{Kuchment2004}).  

\begin{theorem}[Adapted from \cite{BK}]
Let $\alpha_1$ and $\alpha_2$ be as described in 
(\ref{RN})-(\ref{4.4}). Then there exist three orthogonal
(and mutually orthogonal) projection matrices 
$P_D$ (the Dirichlet projection), $P_N$ (the Neumann 
projection), and $P_R = I - P_D - P_N$ (the Robin
projection), and an invertible self-adjoint
operator $\Lambda$ acting on the space $P_R \mathbb{R}^n$
such that the boundary condition 
\begin{equation*}
\alpha_1 y(0) + \alpha_2 y'(0) = 0
\end{equation*} 
can be expressed as 
\begin{equation*}
\begin{aligned}
P_D y(0) &= 0 \\
P_N y'(0) &= 0 \\
P_R y'(0) &= \Lambda P_R y(0).
\end{aligned}
\end{equation*}
Moreover, $P_D$ can be constructed as the projection 
onto the kernel of $\alpha_2$ and $P_N$ can be 
constructed as the projection onto the kernel of 
$\alpha_1$. Construction of the operator $\Lambda$
will be discussed in the following remark. Precisely
the same statement holds for $\beta_1$ and $\beta_2$
for the boundary condition at $x = 1$.
\end{theorem}

\begin{remark}[Construction of $\Lambda$] \label{lambda_remark}
Let $\mathcal{U}$ denote the unitary matrix 
\begin{equation*}
\mathcal{U} = - (\alpha_1 - i \alpha_2)^{-1} (\alpha_1 + i \alpha_2), 
\end{equation*}
where the inverse is guaranteed to exist by our assumptions 
(see Lemma 1.4.7 of \cite{BK}). Let $(\mathcal{U} + I)_R$ denote 
the restriction of $(\mathcal{U} + I)$ to the space 
$P_R \mathbb{R}^n$, so that $(\mathcal{U} + I)_R$ is invertible.
Then 
\begin{equation*}
\Lambda = -i (\mathcal{U} + I)_R^{-1} (\mathcal{U} - I).
\end{equation*}     
It follows that $\alpha_2$ is invertible on the range of 
$\alpha_1 P_R$, and $\Lambda = \alpha_2^{-1} \alpha_1 P_R$.
\end{remark}

\begin{definition}
Let $(P_{D_0}, P_{N_0}, P_{R_0}, \Lambda_0)$ denote the 
projection quadruplet associated with our boundary conditions at $x = 0$,
and let $(P_{D_1}, P_{N_1}, P_{R_1}, \Lambda_1)$ denote the projection 
quadruplet associated with our boundary conditions at $x = 1$.
We denote by $B$ the self-adjoint operator obtained by restricting 
$(P_{R_0} \Lambda_0 P_{R_0} - P_{R_1} \Lambda_1 P_{R_1})$ to 
the space $(\ker P_{D_0}) \cap (\ker P_{D_1})$.
\end{definition}

In Section \ref{schrodinger_section}, we will verify the general
relationship
\begin{equation*}
\text{Mas} (\ell, \ell_1; \Gamma_1) = - \text{Mor} (B)
- \text{Mor} (Q(V(0) - (P_{R_0} \Lambda_0 P_{R_0})^2)Q),
\end{equation*}
where $Q$ denotes the projection matrix onto the null space  of 
$B$. 

We see immediately that if $\alpha_2, \beta_2 = 0$ so that 
$\alpha_1, \beta_1$ have full rank, we obtain 
\begin{equation*}
(P_{D_0}, P_{N_0}, P_{R_0}, \Lambda_0) = (I, 0, 0, 0),
\end{equation*}
and 
\begin{equation*}
(P_{D_1}, P_{N_1}, P_{R_1}, \Lambda_1) = (I, 0, 0, 0).
\end{equation*}
In this case, $B = 0$, and is restricted to the domain 
$(\ker P_{D_0}) \cap (\ker P_{D_1}) = \{0\}$. 
This corresponds with the Dirichlet case mentioned 
above, for which 
$\text{Mas} (\ell, \ell_1; \Gamma_1) = 0$. In particular, 
we have observed that if 
$(\ker P_{D_0}) \cap (\ker P_{D_1}) = \{0\}$ then $Q \equiv 0$.

On the other extreme, suppose $\alpha_2, \beta_2$ both have
full rank (the {\it Neumann-based case}), so that $P_{D_0} = 0$ 
and $P_{D_1} = 0$, and
consequently $(\ker P_{D_0}) \cap (\ker P_{D_1}) = \mathbb{R}^n$.
Focusing on the condition at $x = 0$, we notice that this 
implies $P_{R_0} = I - P_{N_0}$. In this way, $\mathbb{R}^n$
can be decomposed as 
\begin{equation*}
\mathbb{R}^n = P_{N_0} (\mathbb{R}^n) \oplus
P_{R_0} (\mathbb{R}^n),
\end{equation*}
and since $P_{N_0}$ corresponds with projection onto the kernel
of $\alpha_1$ we see that $P_{R_0}$ corresponds with projection
onto the range of $\alpha_1^t$. We conclude that 
$P_{R_0} \alpha_1^t = \alpha_1^t$. Likewise, since $\alpha_1$
annihilates $P_{N_0} (\mathbb{R}^n)$ we see that 
$\alpha_1 P_{R_0} = \alpha_1$. We have, then, 
using Remark \ref{lambda_remark},
\begin{equation*}
P_{R_0} \Lambda_0 P_{R_0} = - P_{R_0} \alpha_2^{-1} \alpha_1 P_{R_0} 
= - P_{R_0} \alpha_2^{-1} \alpha_1.
\end{equation*}
But according to our condition 
$\alpha_1 \alpha_2^t = \alpha_2 \alpha_1^t$, we have 
$\alpha_2^{-1} \alpha_1 = \alpha_1^t (\alpha_2^t)^{-1}$ 
so that 
\begin{equation*}
- P_{R_0} \alpha_2^{-1} \alpha_1 = - \alpha_2^{-1} \alpha_1;
\quad \text{i.e., } P_{R_0} \Lambda_0 P_{R_0} = - \alpha_2^{-1} \alpha_1.
\end{equation*}
We conclude that in this case (where $\alpha_2, \beta_2$ both have
full rank) we have   
\begin{equation*}
\Mas (\ell, \ell_1; \Gamma_1) = 
- \Mor (\beta_2^{-1} \beta_1 - \alpha_2^{-1} \alpha_1)
- \Mor (Q(V(0) - (\alpha_2^{-1} \alpha_1)^2)Q),
\end{equation*}
where in this case $Q$ is a projection onto the null space of 
$B = \beta_2^{-1} \beta_1 - \alpha_2^{-1} \alpha_1$. 

We are now prepared to state the main result of our analysis.

\begin{theorem} \label{main}
For system (\ref{eq:hill}), let $V \in C([0, 1])$ be a symmetric matrix 
in $\mathbb{R}^{n \times n}$, and let $\alpha_1$, $\alpha_2$,
$\beta_1$, and $\beta_2$ be as in (\ref{RN})-(\ref{4.4}). In addition, 
let $Q$ denote projection onto the kernel of $B$, and make the 
non-degeneracy assumption 
$0 \notin \sigma (Q (V(0)-(P_{R_0} \Lambda_0 P_{R_0})^2) Q)$. Then we have 
\begin{equation*}
\Mor (H) = - \Mas (\ell, \ell_1; \Gamma_2)
+ \Mor (B) + \Mor (Q(V(0) - (P_{R_0} \Lambda_0 P_{R_0})^2)Q).
\end{equation*}  
\end{theorem}  

\begin{remark} In the event that 
$0 \in \sigma (Q (V(0)-(P_{R_0} \Lambda_0 P_{R_0})^2) Q)$, our method
still applies, but the resulting expression for $\Mor (H)$ has additional
terms that arise from a higher order perturbation expansion. 
\end{remark}

\begin{remark} As noted in the lead-in to Theorem \ref{main}, we have
an especially straightforward relation for the Dirichlet case, 
\begin{equation*}
\Mor (H) = - \Mas (\ell, \ell_1; \Gamma_2).
\end{equation*}
In particular, since $B$ is restricted to the space 
$(\ker P_{D_0}) \cap (\ker P_{D_1})$,
we see that this relation holds if the boundary condition on either 
side is Dirichlet. 
\end{remark}

\begin{remark} Our emphasis on the negative eigenvalues of $H$ (the Morse index) 
is simply a convention, and we could similarly develop a theorem counting the number 
of eigenvalues of $H$ below any other fixed real value $\lambda_0 \in \mathbb{R}$. 
In this case, the number of eigenvalues less than $\lambda_0$ would be related
to the Maslov index of a path $\Gamma_2^0$ with $\lambda = \lambda_0$ fixed,
and $s$ going from $s_0$ to $1$ (along with appropriate perturbation terms). This,
of course, would allow us to determine the number of eigenvalues of $H$ on any 
interval $[\lambda_1, \lambda_2] \subset \mathbb{R}$.   
\end{remark}

\begin{remark} As we will briefly discuss in Section \ref{schrodinger_section} 
(see Remark \ref{sturm_liouville_remark}),
the standard Sturm-Liouville oscillation theorem for $n=1$ (relating the zeros of 
an eigenfunction to the position of its associated eigenvalue in the sequence
of all eigenvalues; e.g, Theorem XIII.7.50 in \cite{DS} or Theorem 8.4.5 in \cite{At}) 
follows in a straightforward manner from Theorem \ref{main}.
In this way, Theorem \ref{main} can reasonably be viewed as a generalization 
of this theory to the current $n$-dimensional setting. The nature of this 
generalization is especially elegant in the case that the boundary conditions
at $x=1$ are Dirichlet (see Remark \ref{dirichlet_remark} in Section 
\ref{monotonicity_in_s}).

We note that there is a long history of such generalizations, including 
Arnol'd's seminal work with the Maslov index in the 1960's \cite{arnold67}.
For a related approach that does not directly refer to the Maslov index, 
see Chapter 10 in \cite{At}. To the best of our knowledge Theorem \ref{main} 
is the most complete such theorem in the current setting.     
\end{remark}

The paper is organized as follows. In Section \ref{maslov_section} we 
give a precise definition of the Maslov index, suitable for the 
current analysis, and summarize some of its properties. In Section 
\ref{schrodinger_section} we analyze the Maslov index in 
the setting of (\ref{eq:hill}), proving Theorem \ref{main},
and in Section \ref{applications_section} we discuss several 
applications intended to illustrate our results.

\section{The Maslov index} \label{maslov_section}

In this section, we review a definition of the Maslov index
appropriate for the current analysis, and outline some of its
salient properties. We note that several alternative 
definitions are available (see, for example, \cite{CLM}),
all with generally the same properties.
 
Recalling Definition \ref{lagrangian_subspace}, we consider 
the collection of all Lagrangian subspaces of $\mathbb{R}^{2n}$,
which we designate the {\it Lagrangian Grassmannian} and 
denote $\Lambda (n)$. Let $\Sigma \subset \mathbb{R}$ denote
an index interval, and consider any continuous path of Lagrangian
subspaces $\Upsilon: \Sigma \to \Lambda (n)$. Given a fixed 
Lagrangian subspace $\ell_1$ (the {\it target} space, which for 
us will be associated with data at $x = 1$), we will define the 
Maslov index $\text{Mas} (\Upsilon, \ell_1; \Sigma)$
associated with intersections of $(\Upsilon)_{t \in \Sigma}$
with $\ell_1$.   

As a starting point for our construction, which follows particularly
\cite{BF98, F}, we introduce a complex Hilbert space, 
which we will denote $\mathbb{R}_J^{2n}$. The elements of this 
space will continue to be real-valued vectors of length $2n$, 
but we will define multiplication by complex scalars as 
\[
(\alpha + i\beta) u := \alpha u + \beta Ju, \quad u \in \mathbb{R}^{2n}, 
\alpha + i\beta \in \bbC,
\]
and we will define a complex scalar product 
\[
(u,v)_{\mathbb{R}^{2n}_J}
:= 
(u,v)_{\mathbb{R}^{2n}}-i\omega(u,v),\quad u,v\in\mathbb{R}^{2n}
\]
(recalling $\omega (u,v) = (Ju, v)_{\mathbb{R}^n}$).
It is important to note that, considered as a real vector space, $\mathbb{R}^{2n}_J$ 
is identical to $\mathbb{R}^{2n}$, and not its complexification 
$\mathbb{R}^{2n} \otimes_{\bbR} \bbC$. (In fact,  $\mathbb{R}^{2n}_J \cong \bbC^n$ 
while $\mathbb{R}^{2n} \otimes_{\bbR} \bbC \cong \bbC^{2n}$.) 
However, it is easy to see that $\mathbb{R}^{2n}_J \cong \ell \otimes_{\bbR} \bbC$ 
for any  Lagrangian subspace $\ell \in \Lambda(n)$, and we'll take advantage of 
this correspondence.

For a matrix $U$ acting on $\mathbb{R}^{2n}_J$, we denote the adjoint by $U^{J*}$ so 
that 
\begin{equation*}
(U u, v )_{\mathbb{R}^{2n}_J} =
 (u, U^{J*} v )_{\mathbb{R}^{2n}_J},
\end{equation*}
for all $u,v \in \mathbb{R}^{2n}_J$. We denote by $\mathfrak{U}_J$ the space of 
unitary matrices acting on $\mathbb{R}^{2n}_J$ (i.e., the matrices so that 
$U U^{J*} = U^{J*} U = I$). In order to clarify the nature of $\mathfrak{U}_J$, 
we note that we have the identity 
\begin{equation*}
(U u, U v )_{\mathbb{R}^{2n}_J}
= (u, v)_{\mathbb{R}^{2n}_J},
\end{equation*} 
from which 
\begin{equation*}
(U u, U v)_{\mathbb{R}^{2n}}
- i (J U u, U v)_{\mathbb{R}^{2n}}
=
(u, v)_{\mathbb{R}^{2n}}
- i (J u, v)_{\mathbb{R}^{2n}}.
\end{equation*}
Equating real parts, we see that $U$ must be unitary as a matrix on $\mathbb{R}^{2n}$, 
while by equating imaginary parts we see that $UJ = JU$. We have, then, 
\begin{equation*}
\mathfrak{U}_J=\{U\in\mathbb{R}^{2n \times 2n}\,|\,U^tU=UU^t=I_{2n},\, UJ=JU\}.
\end{equation*}

In addition, it will be useful to define a matrix $U^T$ satisfying 
$U^T z:=\overline{U^t\overline{z}}$, or $U^T=\tau_{1}\circ U^t\circ\tau_{1}$, 
where $\tau_{1}$ is the conjugate operation; that is, if $z=x+Jy$, $x,y\in\ell_1$, 
then $\tau_{1}(z)=\overline{z}:=x-Jy$. It is also clear that 
$\tau_{1}=2\Pi_{1}-I_{2n}$, where $\Pi_{1}$ 
is the orthogonal projection onto $\ell_1$.

Given our target space $\ell_1$, we denote by $\ell_1^{\perp}$ the Lagrangian 
subspace perpendicular to $\ell_1$ in $\mathbb{R}^{2n}$. I.e., $\ell_1^{\perp}$
is a Lagrangian subspace, and 
\begin{equation*}
(u, v)_{\mathbb{R}^{2n}} = 0,\quad 
\forall u \in \ell_1, v \in \ell_1^{\perp}. 
\end{equation*}
If $\mathbf{X}_{\ell_1}$ is a frame for $\ell_1$, then 
$J \mathbf{X}_{\ell_1}$ is a frame for $\ell_1^{\perp}$. We 
can express this as $\ell_1^{\perp} = J (\ell_1)$, indicating
that $\ell_1^{\perp}$ is the space obtained by mapping all 
elements of $\ell_1$ with $J$. 

For each $s\in\Sigma$ we choose a unitary operator $U_s$ 
acting on the complex Hilbert space $\mathbb{R}^{2n}_J$ such that 
$\Upsilon(s)=U_s(\ell_1^{\perp})$. This choice is possible by 
\cite[Proposition 1.1]{BF98}. Indeed, in the current setting,
we can associate a canonical frame $\mathbf{X}_{\Upsilon (s)}$
with each $\Upsilon (s)$, as well as a frame $\mathbf{X}_{\ell_1^{\perp}}$,
and find a family of unitary matices satisfying 
$\mathbf{X}_{\Upsilon (s)} = U_s \mathbf{X}_{\ell_1^{\perp}}$.
(The matrices $U_s$ are not uniquely defined, and in fact we'll find 
that different choices of $U_s$ can be useful in different settings.)

This relationship provides a natural and productive connection between 
the elements $\ell$ of the Lagrangian Grassmannian and elements 
$U \in \mathfrak{U}_J$. However, the associated unitary matrices 
are not uniquely specified, and consequently the spectrum of $U$
contains redundant information. For example, in the simple case 
of $\mathbb{R}^2$ this redundant information corresponds with our
previous observation that each element $\ell \in \Lambda (1)$
corresponds with two points on $S^1$. We overcome this difficulty
by defining a new (uniquely specified) unitary matrix 
$W_s$ in $\mathbb{R}^{2n}_J$ by $W_s=U_s U_s^T$. 

We observe that the unitary condition $UJ = JU$ implies $U$
must have the form 
\begin{equation*}
U = 
\begin{pmatrix}
U_{11} & - U_{21} \\
U_{21} & U_{11}
\end{pmatrix}
=
\begin{pmatrix}
U_{11} & 0 \\
0 & U_{11}
\end{pmatrix}
+ J
\begin{pmatrix}
U_{21} & 0 \\
0 & U_{21}
\end{pmatrix}.
\end{equation*}
In addition, we have the scaling condition 
\begin{equation} \label{unitary_scaling}
\begin{aligned}
U_{11}^t U_{11} + U_{21}^t U_{21} &= I \\
U_{11} U_{11}^t + U_{21} U_{21}^t &= I \\
U_{11}^t U_{21} - U_{21}^t U_{11} &= 0 \\
U_{11} U_{21}^t - U_{21} U_{11}^t &= 0.  
\end{aligned}
\end{equation}
In this way, there is a natural one-to-one correspondence between 
matrices $U \in \mathfrak{U}_J$ and the $n \times n$ complex
unitary matrices $\tilde{U} = U_{11} + i U_{21}$ (i.e., the 
$\tilde{U} \in \mathbb{C}^{n \times n}$ so that 
$\tilde{U}^* \tilde{U} = \tilde{U} \tilde{U}^* = I$). 
 
In this way, the matrix $W_s = U_s U_s^T$ has a natural corresponding
matrix $\tilde{W}_s = \tilde{U}_s \tilde{U}_s^T$, where 
$\tilde{U}^T z = \overline{\tilde{U}^* \bar{z}} = \tilde{U} z$.
Ultimately, we will define the Maslov index in terms of 
$\tilde{W}_s$.

The following properties of the matrices $W_s$ and $\tilde{W}_s$ can 
be found in \cite[Lemma 1.3]{BF98} or \cite[Proposition 2.44]{F}.

\begin{lemma}\label{propWs} 
If $\ell_1$ is a real Lagrangian subspace in $\mathbb{R}^{2n}$, 
$\Upsilon\colon\Sigma=[a,b]\to \Lambda (n)$ is a continuous path, 
$\Pi_s$ and $\Pi_{\ell_1}$ are the orthogonal projections onto $\Upsilon(s)$ 
and $\ell_1$ respectively, and $U_s$ is the unitary operator on $\mathbb{R}^{2n}_J$ 
such that $\Upsilon(s)=U_s(\ell_1^\perp)$, then 
\begin{itemize}\item[(i)] $W_s=(I_{\mathbb{R}^{2n}_J}-2\Pi_s)(2\Pi_{\ell_1}-I_{\mathbb{R}^{2n}_J})$;
\item[(ii)] $\ker(W_s+I_{\mathbb{R}^{2n}_J})$ is isomorphic to 
$(\Upsilon(s)\cap\ell_1)\oplus J(\Upsilon(s)\cap\ell_1)\cong(\Upsilon(s)\cap\ell_1)\otimes_{\bbR}\bbC$;
\item[(iii)] $\dim_\bbR(\Upsilon(s)\cap\ell_1)=\dim \ker(\tilde{W}_s+I)$.
\end{itemize}
\end{lemma}

Following \cite{BF98, F, P96}, we define the Maslov index of $\{\Upsilon(s)\}_{s\in\Sigma}$, with 
target $\ell_1$,  
as the spectral flow of the operator family $\{\tilde{W}_s\}_{s\in\Sigma}$ through $-1$;
that is, as the net count (including multiplicity) of the eigenvalues of $\tilde{W}_s$ crossing the point $-1$ counterclockwise 
on the unit circle minus the number of eigenvalues crossing $-1$ clockwise as the parameter $s$ changes. 
Specifically, let us choose a partition $a=s_0<s_1<\dots<s_n=b$ of $\Sigma=[a,b]$ and numbers 
$\epsilon_j\in(0,\pi)$ so that $\ker\big(\tilde{W}_s-e^{i(\pi\pm\epsilon_j)}I\big)=\{0\}$, 
that is, $e^{i(\pi\pm\epsilon_j)}\in\bbC\setminus \sigma(\tilde{W}_s)$, for $s_{j-1}<s<s_j$ and $j=1,\dots,n$. 
For each $j=1,\dots,n$ and any $s\in[s_{j-1},s_j]$ there are only finitely many values $\theta\in[0,\epsilon_j]$ 
for which $e^{i(\pi+\theta)}\in \sigma(\tilde{W}_s)$.

Fix some $j \in \{1, 2, \dots, n\}$ and consider the value
\begin{equation} \label{kdefined}
k (s,\epsilon_j) := 
\sum_{0 \leq \theta < \epsilon_j}
\dim \ker \big(\tilde{W}_s-e^{i(\pi+\theta)}I \big).
\end{equation} 
for $s_{j-1}\leq s\leq s_j$. This is precisely the sum, along with geometric multiplicity,
of the number of eigenvalues of $\tilde{W}_s$ that lie on the arc 
\begin{equation*}
A_j := \{e^{is}: s \in [\pi, \pi+\epsilon_j)\}.
\end{equation*}
The stipulation that 
$e^{i(\pi\pm\epsilon_j)}\in\bbC\setminus \sigma(\tilde{W}_s)$, for $s_{j-1}<s<s_j$
asserts that no eigenvalue can enter $A_j$ in the clockwise direction 
or exit in the counterclockwise direction during the interval $s_{j-1}<s<s_j$.
In this way, we see that $k(s_j, \epsilon_j) - k (s_{j-1}, \epsilon_j)$ is a count of the number of 
eigenvalues that entered $A_j$ in the counterclockwise direction minus the number that left
in the clockwise direction during the interval $(s_{j-1}, s_j)$.

In dealing with the concatenation of paths, it's particularly important to 
understand this quantity if an eigenvalue resides at $-1$ at either $s = s_{j-1}$
or $s = s_j$. If an eigenvalue moving in the counterclockwise direction 
arrives at $-1$ at $s = s_j$, then we increment the difference foward. On
the other hand, suppose an eigenvalue resides at -1 at $s = s_{j-1}$ and moves
in the counterclockwise direction. There is no change, and so we do not increment
the difference.    

We are ready to define the Maslov index.

\begin{definition}\label{dfnDef3.6}  
Let $\ell_1$ be a fixed Lagrangian subspace in a real Hilbert space $\mathbb{R}^{2n}$ and let 
$\Upsilon\colon\Sigma=[a,b]\to \Lambda(n)$ be a continuous path in the Lagrangian--Grassmannian. 
The Maslov index $\mas(\Upsilon,\ell_1;\Sigma)$ is defined by
\begin{equation}
\mas(\Upsilon,\ell_1;\Sigma)=\sum_{j=1}^n(k(s_j,\epsilon_j)-k(s_{j-1},\epsilon_j)).
\end{equation}
\end{definition}

We refer to \cite[Theorem 3.6]{F} for a list of basic properties of the Maslov index; 
in particular, as mentioned in our introduction, the Maslov index is a homotopy invariant 
and is additive under catenation of paths. 

It will be useful to anticipate some later developments and briefly discuss how 
the Maslov index applies to the contour $\Gamma$ described in Figure \ref{F1}. 
For this, we'll find it notationally convenient to use the notation 
$\ell (s,\lambda) = \ell (x, \lambda)|_{x=s}$ (effectively, distinguishing
between the independent variable $x$ and the variable endpoint $s$). For 
$(s,\lambda) \in \Gamma$, let $\tilde{W}_{s,\lambda}$ denote the unitary 
complex matrix associated with $\ell (s,\lambda)$ and target $\ell_1$.
For this discussion, we will use the important fact, 
verified below, that we have monotonicity in $\lambda$ in the following 
sense: {\it as $\lambda$ increases (with $s$ fixed), the eigenvalues of 
$\tilde{W}_{s, \lambda}$ move clockwise around $S^1$}. 

Focusing first on $\Gamma_1$ (for which $s = s_0$): as our contour proceeds in the counterclockwise
direction the eigenvalues of $\tilde{W}_{s_0, \lambda}$ move clockwise around $S^1$.
In this way, crossings necessarily correspond with eigenvalues of $\tilde{W}_{s_0, \lambda}$
rotating out of some $A_j$, thus reducing the Maslov index. Each of these crossings
corresponds with a solution to the eigenvalue problem 
\begin{equation} \label{ev_equation_s}
\begin{aligned}
H_s y := - y'' + V(x) y &= \lambda y \\
\alpha_1 y(0) + \alpha_2 y'(0) &= 0 \\
\beta_1 y(s) + \beta_2 y'(s) &= 0 
\end{aligned}
\end{equation} 
(with $s = s_0$.) It's convenient to set 
$\xi = x/s$ and $u(\xi) = y(x)$ so that $u$ solves the 
eigenvalue problem 
\begin{equation} \label{ev_equation_u}
\begin{aligned}
H(s) u := - u'' + s^2 V(s \xi) u &= s^2 \lambda u \\
\alpha_1 u(0) + \frac{1}{s} \alpha_2 u'(0) &= 0 \\
\beta_1 u(1) + \frac{1}{s} \beta_2 u'(1) &= 0. 
\end{aligned}
\end{equation} 
It's clear that crossings along $\Gamma_1$ correspond with the 
existence of eigenvalues of the operator $H(s_0)$. More precisely,
a crossing will occur along $\Gamma_1$ at $\lambda$, provided
${s_0}^2 \lambda$ is an eigenvalue of $H(s_0)$. The number of negative
eigenvalues of $H(s_0)$, including multiplicity, is its Morse index,
and since each such eigenvalue decreases the Maslov index by 
its multiplicity we obtain the relation 
\begin{equation*}
\mas (\ell, \ell_1; \Gamma_1) = - \mor (H(s_0)).
\end{equation*}   

\begin{remark} \label{HversusH}
We note for future reference that $H_s$ and $H(s)$ refer to different
operators with different domains. To be precise, 
\begin{equation*}
\begin{aligned}
\dom(H_s) &= \Big\{y \in H^2 (0,s): \alpha_1 y(0) + \alpha_2 y'(0) = 0, 
\beta_1 y(s) + \beta_2 y'(s) = 0 \Big\} \\
\dom(H (s)) &= \Big\{
u \in H^2 (0,1): \alpha_1 u(0)+\frac{1}{s}\alpha_2u'(0)=0; \, \beta_1 u(1)+\frac{1}{s}\beta_2u'(1)=0\Big\}
\end{aligned}
\end{equation*}
Of particular importance, $\lambda (s)$
is an eigenvalue of $H(s)$ if and only if $\lambda_s = \lambda (s)/s^2$
is an eigenvalue of $H_s$. 
\end{remark}

Suppose we have an intersection at the corner point $(s_0,0)$,
where $\Gamma_1$ meets $\Gamma_2$. Since the eigenvalues of 
$\tilde{W}_{s_0, \lambda}$ are moving clockwise around $S^1$,
this must correspond with an eigenvalue of $\tilde{W}_{s_0, \lambda}$
stopping at $-1$ from the clockwise direction. This eigenvalue 
does not leave $A_j$, and so the Maslov index does not increment.

On the other hand, let's consider what happens on $\Gamma_3$.
In this case, $\lambda$ will be decreasing (for counterclockwise
movement along $\Gamma$), so eigenvalues of 
$\tilde{W}_{s_0, \lambda}$ will move in the counterclockwise 
direction along $S^1$.
Accordingly, crossings will correspond with eigenvalues moving into 
some $A_j$, and so the Maslov index will increase. These crossings 
correspond with eigenvalues of $H$ (i.e., $H(1)$), and so 
\begin{equation*}
\mas (\ell, \ell_1; \Gamma_3) = \mor (H).
\end{equation*}   

Suppose we have an intersection at the corner point $(1,0)$. By 
monotonicity in $\lambda$, as $\lambda$ decreases from $0$ the
the eigenvalues of $\tilde{W}_{s_0, \lambda}$ will move in the
counterclockwise direction {\it into} some $A_j$. Since these 
eigenvalues are already in $A_j$ at the start of the time interval,
the Maslov index does not change. 

Finally, let's consider the contour $\Gamma_2$. Aside from the 
Dirichlet case, we don't necessarily have monotonicity (with 
respect to $s$) along $\Gamma_2$, but we can still say something
about the Maslov index based on eigenvalue curves $E_{s_*, \lambda_*}$,
which we'll define as continuous paths in the $s$-$\lambda$ plane
crossing through $(s_*, \lambda_*)$ and along which $\lambda$ is an 
eigenvalue of $H(s)$. Suppose such a 
curve crosses $\Gamma_2$ at some point $(s_*, 0)$. If it bends 
upward, we can consider a small box local to the intersection, so 
that the path exits this box through its top shelf. As with our 
discussion of $\Gamma_3$ this will correspond with an increase
in the Maslov index, and so by homotopy invariance the crossing 
at $(s_*,0)$ will correspond with a decrease in the Maslov index. 
Likewise, if the path crossing $(s_*, 0)$ bends downward the crossing
will correspond with an increase in the Maslov index.

\section{Application to the Schr\"odinger Equation} \label{schrodinger_section}

We now focus on the eigenvalue problem (\ref{eq:hill}), and especially
the first-order form (\ref{first_order}). Throughout our 
analysis, we will make use of the following remark concerning the 
matrices used in defining our boundary conditions. 

\begin{remark}
Note that \eqref{4.4}, \eqref{4.4new} imply that 
\begin{align}
\begin{bmatrix} \beta_1 & -\beta_2\\ 
\beta_2 & \beta_1\end{bmatrix}
\begin{bmatrix} \beta^t_1 & \beta^t_2\\ 
-\beta^t_2 & \beta^t_1\end{bmatrix}=I_{2n},
\end{align}
which, in turn, implies that 
\begin{align}
\begin{bmatrix} \beta^t_1 & \beta^t_2\\ 
-\beta^t_2 & \beta^t_1\end{bmatrix}
\begin{bmatrix} \beta_1 & -\beta_2\\ 
\beta_2 & \beta_1\end{bmatrix}=I_{2n}.
\end{align}
Or,
\begin{align}
 \beta_1^t\beta_2-\beta^t_2\beta_1&=0_{n\times n}\,\\
 \beta^t_1\beta_1+\beta^t_2\beta_2&=I.
\end{align}
Similar equalities hold for matrices $\alpha_1, \alpha_2$.
\end{remark}

Following \cite{DJ11}, for each $\lambda\in\R$ and $s\in(0,1]$ we define the following set of vector valued functions on $[0,s]$:
\begin{equation} \label{defY}
Y_{\lambda} =\Big{\{} \mathbf{p} \in H^1 (0,s): \mathbf{p} \text{ solves (\ref{first_order}) and } \alpha_1 p(0)+\alpha_2 q(0) = 0 \Big{\}}.   
\end{equation}
That is, we consider the ($n$ dimensional) solution space to the equation \eqref{first_order}, defined on $[0,s]$, 
consisting of the solutions that satisfy the boundary condition at $0$. 

We define the trace map $\Phi^{\lambda}_s: Y_{\lambda} \to \R^{2n}$ by the following formula:
\begin{equation}\label{eq:trace}
 \Phi^{\lambda}_s:  \mathbf{p} \mapsto \mathbf{p} (s).
\end{equation}
I.e., for the path of Lagrangian spaces $\ell (s,\lambda)$, we have $\ell (s,\lambda) = \Phi^{\lambda}_s (Y_{\lambda})$.

In what follows, we will use the observation that if 
$\mathbf{X} = {X \choose Z}$ is the frame for a Lagrangian subspace, then 
\begin{equation*}
X^t Z - Z^t X = 0.
\end{equation*}
To see this, we observe that 
since $\mathbf{X}$ is the frame of a Lagrangian subspace, each
of its columns ${x \choose z} \in \mathbb{R}^{2n}$ must satisfy 
\begin{equation*}
(J {x \choose y}, {x \choose y})_{\mathbb{R}^{2n}} = 0, \quad 
\Rightarrow ({-y \choose x}, {x \choose y})_{\mathbb{R}^{2n}} = 0,
\end{equation*}
from which the identity $X^t Z - Z^t X = 0$ is apparent.

\begin{theorem}\label{th:lagrange}
For all  $s \in (0,1]$ and $\lambda \in \R$ the plane $\Phi^{\lambda}_s(Y_{\lambda})$ belongs to the space $\Lambda(n)$ of 
Lagrangian $n$-planes in $\R^{2n}$, with the Lagrangian structure $\omega(v_1,v_2)=(Jv_1, v_2)_{\R^{2n}}$.
\end{theorem}

\begin{proof}
Our target space $\ell_1$ can be represented by a $2n\times n$ matrix ${-\beta^t_2 \choose \beta^t_1}$. 
Since $-\beta_2\beta_1^t=-\beta_1\beta^t_2$ by \eqref{4.4}, the symplectic form $\omega$ vanishes on $\ell_1$. 
Also, $\ell_1$ is $n$-dimensional \eqref{RN}. Hence, $\ell_1$ is Lagrangian.  

Next, we represent $\Phi_s^\lambda(Y_{\lambda})$ as a $2n\times n$ matrix 
${X(s,\lambda) \choose Z(s,\lambda)}$. 
Then 
\begin{align}
(X^tZ-Z^tX)'&=(X^t)'Z+X^tZ'-(Z^t)'X-Z^tX' \\
& =Z^tZ+X^t(V-\lambda I)X-X^t(V-\lambda I)X-Z^tZ=0. 
\end{align}
But since 
\begin{equation*}
X^t(0,\lambda)Z(0,\lambda)-Z^t(0,\lambda)X(0,\lambda)=-\alpha_2\alpha_1^t+\alpha_1\alpha^t_2=0,
\end{equation*} 
we see that $X^tZ-Z^tX=0$. 
Therefore,  the symplectic form $\omega$ vanishes on $\Phi_s^\lambda(Y_{\lambda})$. 
And since $Y_{\lambda}$ is $n$-dimensional, $\Phi_s^\lambda(Y_{\lambda})$ is also $n$-dimensional, 
and, therefore, it is Lagrangian.
\end{proof}

At this point, we would like to relate the crossings of the path $\big\{\Phi_\lambda^s(Y_{\lambda})\big\}$ 
to eigenvalues of differential operators $H_s$ introduced in (\ref{ev_equation_u}).
We remark that $y \in \ker\big(H_{s}-\lambda I\big)$ if and only if the vector valued function 
$\mathbf{p}$ is a solution of \eqref{first_order} on $[0,s]$ that satisfies the 
boundary conditions $\alpha_1 p(0) + \alpha_2 q(0) = 0$ and 
$\beta_1 p(s) + \beta_2 q(s) = 0$.
In addition, let $H_s^D$ denote the operator $H_s$ 
with Dirichlet boundary conditions (i.e., 
$\begin{bmatrix} \beta_1 & \beta_2\end{bmatrix}=\begin{bmatrix} I_n & 0\end{bmatrix}$).

As discussed in Section \ref{maslov_section}, we proceed by associating 
each Lagrangian subspace 
$\ell (s, \lambda)$ with a matrix $U_{s, \lambda} \in \mathfrak{U}_J$. In 
particular, $U_{s, \lambda}$ should map $\ell_1^{\perp}$ to $\ell (s, \lambda)$.
In terms of frames, this asserts that 
\begin{equation*}
\mathbf{X} (s, \lambda) = U(s, \lambda) 
\begin{bmatrix}
\beta_1^t \\
\beta_2^t
\end{bmatrix},
\end{equation*}
where we will need to scale $\mathbf{X}$ to ensure that 
$U_{s, \lambda}$ is unitary (see below). According to our
condition $UJ = JU$, we know that $U$ must have the form 
\begin{equation*}
U = 
\begin{pmatrix}
U_{11} &  - U_{21} \\
U_{21} & U_{11}
\end{pmatrix},
\end{equation*}
allowing us to express the relationship for $U$ as 
\begin{equation*}
\begin{bmatrix}
X^t \\
Z^t
\end{bmatrix}
=
\begin{pmatrix}
\beta_1 & -\beta_2 \\
\beta_2 & \beta_1
\end{pmatrix}
\begin{bmatrix}
U_{11}^t \\
U_{21}^t
\end{bmatrix}.
\end{equation*}
In order to ensure the unitary normalization 
$U_{11}^t U_{11} + U_{21}^t U_{21} = I$, we note
that we can choose the frame $\mathbf{X}$
to be $XM \choose ZM$ for any $n \times n$ 
invertible matrix $M$. With this choice, we find 
that $U$ has the form 
\begin{equation*}
U = 
\begin{pmatrix}
XM & -ZM \\
ZM & XM
\end{pmatrix}
\mathcal{B},
\end{equation*}
where 
\begin{equation*}
\mathcal{B} := 
\begin{bmatrix} \beta_1 & \beta_2\\ 
-\beta_2 & \beta_1\end{bmatrix},
\end{equation*}
and we must have 
\begin{equation*}
\begin{aligned}
M^t X^t X M + M^t Z^t Z M &= I \\
M^t X^t Z M - M^t Z^t X M &= 0.
\end{aligned}
\end{equation*}
We will check below that the choices 
$M = (X^t X + Z^t Z)^{-1/2}$ and
$M = X^{-1} (I+M_D^2)^{-1/2}$, where 
$M_D = ZX^{-1}$ can both be effective. 
(As discussed in \cite{LS1}
$M_D$ is the Weyl-Titchmarsh function associated with $H_s^D$.) 

For the following calculations we will find it convenient to define two matrices
\begin{equation*}
\begin{aligned}
\mathbb{M} (s, \lambda) &:= I + M_D^2 (s, \lambda), \\
\mathbb{X} (s, \lambda) &:= X^t (s, \lambda) X (s, \lambda) + Z^t (s, \lambda) Z (s, \lambda).
\end{aligned}
\end{equation*}

\begin{lemma} \label{m_lemma} 
The matrix $M_D = Z X^{-1}$ is symmetric whenever $X$ is invertible. Moreover, 
we have the relations
\begin{equation*}
\begin{aligned}
X \mathbb{X}^{-1} X^t + Z \mathbb{X}^{-1} Z^t &= I_n, \\
Z \mathbb{X}^{-1} X^t - X \mathbb{X}^{-1} Z^t &= 0_n,
\end{aligned}
\end{equation*}
as well as the commutation 
\begin{equation*}
\mathbb{M}^{-1/2} M_D = M_D \mathbb{M}^{-1/2}.
\end{equation*} 
\end{lemma}

\begin{proof} 
For symmetry, we observe that if $X$ is invertible, we can 
write 
\begin{equation*}
Z X^{-1} X = (X^t)^{-1} X^t Z X^{-1} X.
\end{equation*} 
Recalling the relation $X^t Z - Z^t X = 0$, and interchanging transpose 
with inverse, we find 
\begin{equation*}
(X^t)^{-1} X^t Z X^{-1} X = (X^{-1})^t Z^t X X^{-1} X 
= (X^{-1})^t Z^t X.
\end{equation*}
We see that $Z X^{-1} = (X^{-1})^t Z^t$; i.e., 
$M_D = M_D^t$.

For the last claim, we first assume $X(s,\lambda)$ 
and $Z(s,\lambda)$ are invertible. Then,
\begin{align}\lb{iden}
\begin{split}
X \mathbb{X}^{-1} X^t + Z \mathbb{X}^{-1} Z^t &= X(X^tX+Z^tZ)^{-1}X^t+Z(X^tX+Z^tZ)^{-1}Z^t \\
&= ((X^t)^{-1}(X^tX+Z^tZ)X^{-1})^{-1}+((Z^t)^{-1}(X^tX+Z^tZ)Z^{-1})^{-1} \\
&= (I+(X^t)^{-1}Z^tZX^{-1})^{-1}+((Z^t)^{-1}X^tXZ^{-1}+I)^{-1} \\
&= (I+M_D^2)^{-1}+(M_D^{-2}+I)^{-1} \\
&= (I+M_D^2)^{-1}+M^2_D(I+M^2_D)^{-1} = I_n.
 \end{split}
\end{align}
Since $X\mathbb{X}^{-1}X^t+Z\mathbb{X}^{-1}Z^t$ is continuous with respect to 
$(s,\lambda)$, \eqref{iden} holds for any $s$ and $\lambda$ (i.e., even for
pairs with $X(s,\lambda)$ not invertible).
Similarly, one can check that 
\begin{equation*}
Z\mathbb{X}^{-1}X^t-X\mathbb{X}^{-1}Z^t=0.
\end{equation*}

In order to see the commutation relation, we note that the relation 
\begin{equation*}
M_D \mathbb{M} = \mathbb{M} M_D
\end{equation*}
is trivial and leads immediately to 
\begin{equation*}
\mathbb{M}^{-1} M_D = M_D \mathbb{M}^{-1}.
\end{equation*}
The claim now follows from the general observation that if $A$ is 
positive definite and $AB = B A$ then $A^{1/2} B = B A^{1/2}$
and $B A^{-1/2} = A^{-1/2} B$.
\end{proof}

We will identify two choices of unitary matrix $U_{s,\lambda}$, which will be specified in 
terms of the matrices 
\begin{equation} \label{matrices_defined}
\begin{aligned}
\mathcal{M}_D &:= 
\begin{bmatrix} \mathbb{M}^{-1/2} & -\mathbb{M}^{-1/2}M_D\\ 
\mathbb{M}^{-1/2}M_D & \mathbb{M}^{-1/2}
\end{bmatrix} \\
\mathcal{X}_D &:=
\begin{bmatrix} X\mathbb{X}^{-1/2} & -Z\mathbb{X}^{-1/2}\\ 
Z\mathbb{X}^{-1/2} & X\mathbb{X}^{-1/2}\end{bmatrix}. \\
\end{aligned}
\end{equation}

\begin{lemma} \label{us_lemma}
Suppose $\mathbf{X} (s, \lambda) = {X (s,\lambda) \choose Z (s,\lambda)}$ is 
any frame for the Lagrangian subspace $\Phi_s^{\lambda} (Y_{\lambda})$. Then
\begin{equation*}
U_{s, \lambda} = \mathcal{M}_D \mathcal{B}
\end{equation*}
is unitary in $\mathbb{R}_J^{2n}$ and 
satisfies $\Phi_s^{\lambda} (Y_{\lambda}) = U_{s, \lambda} (\ell_1^{\perp})$
for all $\lambda \in \mathbb{R} \backslash \sigma (H_s^D)$, and 
\begin{equation*}
Q_{s, \lambda} := \mathcal{X}_D \mathcal{B}
\end{equation*} 
is unitary in $\mathbb{R}_J^{2n}$ and satisifies the same relation 
for all $\lambda \in \mathbb{R}$.
\end{lemma}

\begin{proof}
First, using \eqref{4.4} and \eqref{4.4new}, we see that 
\begin{equation*}
\mathcal{B} \mathcal{B}^t = 
\begin{bmatrix}
\beta_1 \beta_1^t + \beta_2 \beta_2^t & - \beta_1 \beta_2^t + \beta_2 \beta_1^t \\
- \beta_2 \beta_1^t + \beta_1 \beta_2^t & \beta_2 \beta_2^t + \beta_1 \beta_1^t
\end{bmatrix} 
= 
\begin{bmatrix}
I_n & 0_n \\
0_n & I_n
\end{bmatrix}. 
\end{equation*}

We can now readily check that $U_{s, \lambda}$ is unitary on $\mathbb{R}^{2n}$. We compute
\begin{align*}
U_{s,\lambda}U^t_{s,\lambda} &= \mathcal{M}_D \mathcal{B} \mathcal{B}^t \mathcal{M_D}^t
= \mathcal{M}_D \mathcal{M_D}^t \\
&=\begin{bmatrix} (I+M_D^2)^{-1}+(I+M_D^2)^{-1}M_D^2 & M_D(I+M_D^2)^{-1}-M_D(I+M_D^2)^{-1}\\ 
M_D(I+M_D^2)^{-1}-M_D(I+M_D^2)^{-1} &(I+M_D^2)^{-1}M_D^2+(I+M_D^2)^{-1}\end{bmatrix}=I_{2n}.
\end{align*}
Note that we used the fact that $M_D$ is symmetric. Similarly, $U_{s,\lambda}^tU_{s,\lambda}=I_{2n}$,
and it is also easy to check that $U_{s,\lambda}J=JU_{s,\lambda}$.

For $Q_{s, \lambda}$, we proceed as with $U_{s, \lambda}$ to find
\begin{align*}
Q_{s, \lambda} Q_{s, \lambda}^t &= \mathcal{X}_D \mathcal{X}_D^t \\
&= 
\begin{bmatrix}  
X \mathbb{X}^{-1} X^t + Z \mathbb{X}^{-1} Z^t & X \mathbb{X}^{-1} Z^t-Z \mathbb{X}^{-1} X^t\\ 
Z \mathbb{X}^{-1} X^t-X \mathbb{X}^{-1} Z^t & Z \mathbb{X}^{-1} Z^t+X \mathbb{X}^{-1}X^t
\end{bmatrix}
= I_{2n}.
\end{align*} 
Proceeding similarly, we can show that 
$Q^t_{s,\lambda} Q_{s,\lambda}=I_{2n}$. 

In order to check the relation 
$\Phi_s^{\lambda} (Y_{\lambda}) = U_{s, \lambda} (\ell_1^{\perp})$,
let $u\in \ell_1^{\perp}$, so that $u=(\beta_1^tx,\beta_2^tx)^\top$ for
some $x\in\mathbb{R}^{n}$. 
Therefore,
\begin{equation*} 
U_{s,\lambda}u=((I+M^2(s,\lambda))^{-1/2}x,M(s,\lambda)(I+M^2(s,\lambda))^{-1/2}x)^\top
=(Xy,Zy)^\top
\in\Phi_s^\lambda(Y_{\lambda}),
\end{equation*}
where $y=X^{-1}(I+M^2(s,\lambda))^{-1/2}x$. Similarly, 
\begin{equation*}
Q_{s,\lambda}u=(X(X^tX+Z^tZ)^{-1/2}x,Z(X^tX+Z^tZ)^{-1/2}x)^\top
=(Xy,Zy)^\top
\in\Phi_s^\lambda(Y_{\lambda}).
\end{equation*}
where $y=(X^tX+Z^tZ)^{-1/2}x$
\end{proof}

\begin{remark} \label{complex_u}
The matrices $\mathcal{M}_D$ and $\mathcal{X}_D$ are in $\mathfrak{U}_J$, and
as discussed in Section \ref{maslov_section}, can be associated with $n \times n$
complex-valued unitary matrices. To be precise, notice that we can express the matrix $\mathcal{M}_D$ as 
\begin{equation*}
\mathcal{M}_D = 
\begin{bmatrix}
\mathbb{M}^{-1/2} & 0 \\
0 & \mathbb{M}^{-1/2} 
\end{bmatrix}
+ 
J 
\begin{bmatrix}
\mathbb{M}^{-1/2} M_D & 0 \\
0 & \mathbb{M}^{-1/2} M_D 
\end{bmatrix},
\end{equation*}
which can be associated with the complex-valued $n \times n$ matrix
\begin{equation*}
\tilde{\mathcal{M}}_D = \mathbb{M}^{-1/2} + i \mathbb{M}^{-1/2} M_D
= \mathbb{M}^{-1/2}(I + i M_D).
\end{equation*}
Likewise, for $\mathcal{X}_D$ we can write 
\begin{equation*}
\mathcal{X}_D = 
\begin{bmatrix}
X \mathbb{X}^{-1/2} & 0 \\
0 & X \mathbb{X}^{-1/2} 
\end{bmatrix}
+ 
J 
\begin{bmatrix}
Z \mathbb{X}^{-1/2} & 0 \\
0 & Z \mathbb{X}^{-1/2} 
\end{bmatrix},
\end{equation*}
and we associate with this the complex-valued $n \times n$ matrix
\begin{equation*}
\tilde{\mathcal{X}}_D = X \mathbb{X}^{-1/2} + i Z \mathbb{X}^{-1/2}
= (X + i Z) \mathbb{X}^{-1/2}.
\end{equation*}
\end{remark}

We are now prepared to derive an expression for the matrix 
$W_{s, \lambda} = U_{s, \lambda} U_{s, \lambda}^T$ described
in our definition of the Maslov index. We note at the 
outset that we can write 
\begin{equation*}
W_{s,\lambda}=U_{s,\lambda} U^T_{s,\lambda}
= U_{s,\lambda} \tau_{1} U_{s,\lambda}^t \tau_{1}
= \mathcal{M}_D \mathcal{B} \tau_{1} \mathcal{B}^t \mathcal{M}_D^t \tau_{1}.
\end{equation*}

\begin{lemma} \label{ws_lemma}
Under the assumptions of Lemma \ref{us_lemma},
\begin{equation*}
W_{s, \lambda} = \mathcal{M}_D (s,\lambda)^2 \mathfrak{B},
\end{equation*}
where 
\begin{equation*}
\mathfrak{B} = 
\begin{bmatrix} 
\beta_1^t\beta_1-\beta_2^t\beta_2 & 2\beta_2^t\beta_1\\ 
-2\beta_2^t\beta_1 & \beta_1^t\beta_1-\beta_2^t\beta_2
\end{bmatrix}.
\end{equation*}
\end{lemma}

\begin{proof}
First, we would like to find $\tau_{1}=2\Pi_{1}-I_{2n}$. It is clear that 
\begin{align}
 \Pi_{1}=\begin{bmatrix} \beta_2^t\beta_2 & -\beta_2^t\beta_1\\ 
-\beta_1^t\beta_2 & \beta_1^t\beta_1\end{bmatrix},
\end{align}
and we can check directly that $\Pi^2_{1}=\Pi_{1}$. 
Moreover, using \eqref{4.4} and \eqref{4.4new}, we obtain that
\begin{align*}
 \Pi^2_{1}&=
\begin{bmatrix} 
\beta_2^t\beta_2 & -\beta_2^t\beta_1\\ 
-\beta_1^t\beta_2 & \beta_1^t\beta_1
\end{bmatrix}
\begin{bmatrix} 
\beta_2^t\beta_2 & -\beta_2^t\beta_1\\ 
-\beta_1^t\beta_2 & \beta_1^t\beta_1
\end{bmatrix} \\
&= \begin{bmatrix} \beta_2^t\beta_2\beta_2^t\beta_2+\beta_2^t\beta_1\beta_1^t\beta_2 & 
-\beta_2^t\beta_2\beta_2^t\beta_1-\beta_2^t\beta_1\beta_1^t\beta_1\\ 
-\beta_1^t\beta_2\beta_2^t\beta_2-\beta_1^t\beta_1\beta_1^t\beta_2 & 
\beta_1^t\beta_2\beta_2^t\beta_1+\beta_1^t\beta_1\beta_1^t\beta_1
\end{bmatrix} \\
&= \begin{bmatrix} \beta_2^t\beta_2 & -\beta_2^t\beta_1\\ 
-\beta_1^t\beta_2 & \beta_1^t\beta_1\end{bmatrix}=\Pi_{1}.
\end{align*}
Also,
\begin{align*}
 \Pi_{1}  \begin{bmatrix} -\beta_2^tx \\ \beta_1^tx \end{bmatrix}
=\begin{bmatrix} \beta_2^t\beta_2 & -\beta_2^t\beta_1\\ 
-\beta_1^t\beta_2 & \beta_1^t\beta_1\end{bmatrix}\begin{bmatrix} -\beta_2^tx\\ 
\beta_1^tx\end{bmatrix}=\begin{bmatrix} -\beta_2^t\beta_2\beta_2^tx-\beta_2^t\beta_1\beta_1^tx\\ 
\beta_1^t\beta_2\beta_2^tx+\beta_1^t\beta_1\beta_1^tx\end{bmatrix}=\begin{bmatrix} -\beta_2^tx\\ 
\beta_1^tx\end{bmatrix},
\end{align*}
and
\begin{align*}
 \Pi_{1} \begin{bmatrix} \beta_1^tx \\ \beta_2^tx\end{bmatrix}
=\begin{bmatrix} \beta_2^t\beta_2 & -\beta_2^t\beta_1\\ 
-\beta_1^t\beta_2 & \beta_1^t\beta_1\end{bmatrix}\begin{bmatrix} \beta_1^tx\\ 
\beta_2^tx\end{bmatrix}=\begin{bmatrix} \beta_2^t\beta_2\beta_1^tx-\beta_2^t\beta_1\beta_2^tx\\ 
-\beta_1^t\beta_2\beta_1^tx+\beta_1^t\beta_1\beta_2^tx\end{bmatrix}=0.
\end{align*}
Therefore,
\begin{align}
 \tau_{1}=2\Pi_{1}-I_{2n}=\begin{bmatrix} 2\beta_2^t\beta_2-I & -2\beta_2^t\beta_1\\ 
-2\beta_1^t\beta_2 & 2\beta_1^t\beta_1-I\end{bmatrix}.
\end{align}

Computing directly, we find 
\begin{equation*}
\begin{aligned}
\mathcal{B} \tau_1 \mathcal{B}^t
&= 
\begin{bmatrix} 
\beta_1 & \beta_2\\ 
-\beta_2 & \beta_1
\end{bmatrix} 
\begin{bmatrix} 
2 \beta_2^t \beta_2 - I & -2 \beta_2^t \beta_2 \\ 
- 2 \beta_1^t \beta_2 & 2 \beta_1^t \beta_1 - I
\end{bmatrix} 
\begin{bmatrix} 
\beta_1^t & - \beta_2^t \\ 
\beta_2^t & \beta_1^t
\end{bmatrix} \\
&= 
\begin{bmatrix} 
\beta_1 & \beta_2\\ 
-\beta_2 & \beta_1
\end{bmatrix} 
\begin{bmatrix} 
- \beta_1^t & - \beta_2^t \\ 
-\beta_2^t & \beta_1^t
\end{bmatrix}  
=
\begin{bmatrix} 
-I_n & 0 \\ 
0 & I_n
\end{bmatrix}. 
\end{aligned}
\end{equation*}

We have, then, 
\begin{equation*}
\begin{aligned}
\mathcal{B} \tau_1 \mathcal{B}^t \mathcal{M}_D^t 
&= 
\begin{bmatrix} 
-I_n & 0 \\ 
0 & I_n
\end{bmatrix}
\begin{bmatrix} 
\mathbb{M}^{-1/2} & \mathbb{M}^{-1/2}M_D \\ 
- \mathbb{M}^{-1/2}M_D & \mathbb{M}^{-1/2}
\end{bmatrix} \\
&=
\begin{bmatrix} 
- \mathbb{M}^{-1/2} & -\mathbb{M}^{-1/2}M_D \\ 
- \mathbb{M}^{-1/2}M_D & \mathbb{M}^{-1/2}
\end{bmatrix} \\
&= 
\begin{bmatrix} 
\mathbb{M}^{-1/2} & - \mathbb{M}^{-1/2}M_D \\ 
\mathbb{M}^{-1/2}M_D & \mathbb{M}^{-1/2}
\end{bmatrix}
\begin{bmatrix} 
-I_n & 0 \\ 
0 & I_n
\end{bmatrix}.
\end{aligned}
\end{equation*}

In this way, we see that 
\begin{equation*}
\begin{aligned}
\mathcal{M}_D \mathcal{B} \tau_1 \mathcal{B}^t \mathcal{M}_D^t \tau_1
&= 
\mathcal{M}_D^2 \begin{bmatrix} 
-I_n & 0 \\ 
0 & I_n
\end{bmatrix}
\begin{bmatrix} 
2 \beta_2^t \beta_2 - I & -2 \beta_2^t \beta_2 \\ 
- 2 \beta_1^t \beta_2 & 2 \beta_1^t \beta_1 - I
\end{bmatrix} \\
&= \mathcal{M}_D^2 
\begin{bmatrix} 
\beta_1^t \beta_1 - \beta_2^t \beta_2 & 2 \beta_2^t \beta_1 \\ 
- 2 \beta_2^t \beta_1 & \beta_1^t \beta_1 - \beta_2^t \beta_2
\end{bmatrix}.   
\end{aligned}
\end{equation*}
\end{proof} 

We observe that $\mathfrak{B}$ is a unitary matrix in the 
form 
\begin{equation*}
\mathfrak{B} = 
\begin{bmatrix} 
\beta_1^t \beta_1 - \beta_2^t \beta_2 & 2 \beta_2^t \beta_1 \\ 
- 2 \beta_2^t \beta_1 & \beta_1^t \beta_1 - \beta_2^t \beta_2
\end{bmatrix}
= \begin{bmatrix} 
\beta_1^t \beta_1 - \beta_2^t \beta_2 & 0 \\ 
0 & \beta_1^t \beta_1 - \beta_2^t \beta_2
\end{bmatrix}
+ J
\begin{bmatrix} 
- 2 \beta_2^t \beta_1 & 0 \\ 
0 & - 2 \beta_2^t \beta_1
\end{bmatrix},
\end{equation*}
and can be associated with the $n \times n$ complex unitary matrix
\begin{equation*}
\tilde{\mathfrak{B}} = (\beta_1^t \beta_1 - \beta_2^t \beta_2) - i 2 \beta_2^t \beta_1.
\end{equation*}

In this way, $W_{s, \lambda}$ corresponds with the complex $n \times n$ matrix
\begin{equation*}
\tilde{W}_{s, \lambda} = \tilde{\mathcal{M}}_D^2 
\tilde{\mathfrak{B}}.
\end{equation*}
Here, 
\begin{equation*}
\begin{aligned}
\tilde{\mathcal{M}}_D^2 &= (\mathbb{M}^{-1/2} + i \mathbb{M}^{-1/2} M_D)^2 \\
&= (I+M_D^2)^{-1} (I+iM_D)^2 = ((I+iM_D)(I-iM_D))^{-1} (I+iM_D)^2 \\
&= (I - iM_D)^{-1} (I + iM_D),
\end{aligned}
\end{equation*}
which is the standard Cayley transform of $i M_D$. 

In the event that $X$ is invertible, we find (using the definition of 
$M_D$) that 
\begin{equation*}
\tilde{\mathcal{M}}_D^2 = (X+iZ)(X-iZ)^{-1},
\end{equation*}
and more generally we can arrive at this form by repeating 
our calculations using $Q$ in place of $U$. We conclude with the matrix
we'll use for our Maslov index calculations, 
\begin{equation*}
\tilde{W}_{s, \lambda} = (X+iZ)(X-iZ)^{-1} 
((\beta_1^t \beta_1 - \beta_2^t \beta_2) - i 2 \beta_2^t \beta_1).
\end{equation*}

\begin{remark} \label{sturm_liouville_remark} We are now in a position to 
indicate how the Sturm-Liouville oscillation theorem for $n=1$ follows 
from Theorem \ref{main}. In this case (i.e., for $n=1$) we have 
\begin{equation*}
\tilde{W}_{s, \lambda} 
= \frac{y (s; \lambda)+iy' (s; \lambda)}{y (s; \lambda)-iy' (s; \lambda)} (\beta_1^2 - \beta_2^2 - i 2 \beta_1 \beta_2),
\end{equation*}
and for simplicity let's focus on the case in which we have Dirichet 
boundary conditions at both $x = 0$ and $x=1$ (so that $\alpha_1, \beta_1 = 1$ 
and $\alpha_2, \beta_2 = 0$). In this case, we have a crossing at $s^*$ 
(so that $\tilde{W}_{s^*, \lambda} = -1$) if and only if $y(s^*; \lambda) = 0$. 
We'll see in Section \ref{monotonicity_in_s} that in this case crossings on $S^1$ must 
occur in the clockwise direction, and since $\tilde{W}_{0, \lambda} = -1$
(due to the Dirichlet condition at $x = 0$) we will have 
$\tilde{W}_{s_0, \lambda} = e^{i (\pi - \epsilon)}$ for $s_0$ sufficiently
small (and some $\epsilon > 0$). The Principal Maslov Index will now be the
negative of a count of the number of times $\tilde{W}_{s, 0}$ 
crosses $-1$ as $s$ goes from $s_0$ to $1$. Moreover, each of these crossings 
will correspond with a zero of $y (s;0)$ (as noted above), and so we can conclude
from Theorem \ref{main} that the number of negative eigenvalues of $H$ is 
precisely the number of zeros of $y$. (The standard Sturm-Liouville oscillation 
theorem for $n=1$ requires $\lambda = 0$ to be an eigenvalue, but we clearly 
do not need that.) Other cases follow similarly.
\end{remark}

Our final preliminary lemma addresses continuity of the path of Lagrangian 
subspaces $\{\ell (s,\lambda)\}_{(s,\lambda) \in \Gamma}$.

\begin{lemma} \label{continuity_lemma}
For system (\ref{eq:hill}), let $V \in C([0, 1])$ be a symmetric matrix 
in $\mathbb{R}^{n \times n}$, and let $\alpha_1$, $\alpha_2$,
$\beta_1$, and $\beta_2$ be as in (\ref{RN})-(\ref{4.4}). Then the 
path of Lagrangian subspaces $\{\ell (s,\lambda)\}_{(s,\lambda) \in \Gamma}$
is continuous.
\end{lemma}

\begin{proof} Following \cite{F} (p. 274), we specify our metric on the 
Lagrangian Grassmannian $\Lambda (n)$ in terms of orthogonal projections 
onto elements $\ell \in \Lambda (n)$. Precisely, let $\mathcal{P}_i$ 
denote the orthogonal projection matrix onto $\ell_i \in \Lambda (n)$
for $i = 1,2$. We take our metric $d$ on $\Lambda (n)$ to be defined 
by 
\begin{equation*}
d (\ell_1, \ell_2) := \|\mathcal{P}_1 - \mathcal{P}_2 \|,
\end{equation*} 
where $\| \cdot \|$ can denote any matrix norm.

For $\ell (s,\lambda)$, we have a frame $\mathbf{X}$, and it follows from 
elementary matrix theory that the associated orthogonal projection matrix 
$\mathcal{P}_{s, \lambda}$ satisfies 
$\mathcal{P}_{s, \lambda} = \mathbf{X} (\mathbf{X}^t \mathbf{X})^{-1} \mathbf{X}^t$.
Computing directly, we find 
\begin{equation*}
\mathcal{P}_{s, \lambda} = 
\begin{pmatrix}
X \mathbb{X}^{-1} X^t & X \mathbb{X}^{-1} X^t \\
Z \mathbb{X}^{-1} X^t & Z \mathbb{X}^{-1} Z^t.
\end{pmatrix}
\end{equation*}

We see, then, that continuity of $\ell (s,\lambda)$ follows immediately 
from the continuity of $\mathcal{P}_{s, \lambda}$, which in turn follows
from the continuity of solutions of (\ref{first_order}) in $x$ and $\lambda$.
\end{proof}

\subsection{Crossings on $\Gamma_3$ ($s=1, \lambda \in [0,-\lambda_{\infty}$])}

In this section, we verify our claim in the introduction that along the top shelf 
$\Gamma_3$ the Maslov index is precisely the Morse index of $H$. The inverted 
interval $[0, -\lambda_{\infty}]$ indicates the direction of the path $\Gamma_3$.

\begin{lemma} \label{Gamma3Lemma} 
Under the assumptions of Lemma \ref{continuity_lemma} we have 
\begin{equation}
 \mo(H)=\mi(\ell, \ell_1; \Gamma_3).
 \end{equation} 
\end{lemma}

\begin{proof}
From Lemma \ref{propWs}, we know that 
$\dim (\Phi^{\lambda}_s(Y_{\lambda}) \cap \ell_1) = \dim \ker(\tilde{W}_{s,\lambda}+I)$ for $s=1, \lambda\in[0, -\lambda_{\infty}]$.
Assume that $\lambda^*\in[0,-\lambda_{\infty}]$ is a crossing, that is, $\Phi^{\lambda^*}_1 (Y_{\lambda^*}) \cap \ell_1 \neq\{0\}$. 
Then there exists a solution of \eqref{eq:hill} such that the boundary conditions are satisfied. Therefore, $\lambda^*$ 
is an eigenvalue of $H$. Moreover, since $\Phi^{\lambda^*}_1 (Y_{\lambda^*})$ are the traces of weak solutions that satisfy the boundary condition
at $0$, 
$\dim(\ker(H-\lambda^*I))=\dim (\Phi^{\lambda^*}_L(Y_{\lambda^*}) \cap \ell_1)=\dim \ker(\tilde{W}_{1,\lambda^*}+I)$.

Next, we would like to compute the Maslov index of the path 
$\big\{\Phi_1^\lambda(Y_{1,\lambda})\big\}_{\lambda=\lambda^*-\varepsilon}^{\lambda^*+\varepsilon}$, i.e., 
the net count of the eigenvalues of $\tilde{W}_{1,\lambda}$ crossing the point $-1$ as $\lambda$ goes from $\lambda^*-\varepsilon$ to 
$\lambda^*+\varepsilon$. As a starting point, we 
differentiate $\tilde{W}_{s,\lambda}$ with respect to $\lambda$:
\begin{align*}
 \frac{\partial}{\partial \lambda}\tilde{W}_{s,\lambda}&=(\dot X+i\dot Z)(X-iZ)^{-1} \tilde{\mathfrak{B}}
 -(X+iZ)(X-iZ)^{-1}(\dot X-i\dot Z)(X-iZ)^{-1} \tilde{\mathfrak{B}} \\
 &= (\dot X+i\dot Z)(X-iZ)^{-1} \tilde{\mathfrak{B}}-
 \tilde{W}_{s, \lambda} \tilde{\mathfrak{B}}^* (\dot X-i\dot Z) (X-iZ)^{-1} \tilde{\mathfrak{B}},
\end{align*}
where $\dot{X}, \dot{Z}$ denote derivatives of $X$ and $Z$ with respect to $\lambda$, and we've
used the fact that $\tilde{\mathfrak{B}}$ is unitary.

Now, we multiply both sides by $\tilde{W}_{s, \lambda}^*$
\begin{align*}
 &\tilde{W}_{s, \lambda}^* \dot{\tilde{W}}_{s, \lambda} 
= \tilde{\mathfrak{B}}^* 
(X^t+iZ^t)^{-1} (X^t-iZ^t) (\dot X+i\dot Z)(X-iZ)^{-1} \tilde{\mathfrak{B}}
- \tilde{\mathfrak{B}}^* (\dot X-i\dot Z)(X-iZ)^{-1} \tilde{\mathfrak{B}} \\
 &= \tilde{\mathfrak{B}}^* (X^t+iZ^t)^{-1} [(X^t-iZ^t)(\dot X+i\dot Z)-(X^t+iZ^t)(\dot X-i\dot Z)](X-iZ)^{-1} \tilde{\mathfrak{B}} \\
 &= ((X-iZ)^{-1}\tilde{\mathfrak{B}})^* [2iX^t\dot Z-2iZ^t\dot X]((X-iZ)^{-1} \tilde{\mathfrak{B}}).
\end{align*}
Multiplying on the left by $\tilde{W}_{s, \lambda}$, and recalling that 
$\tilde{W}_{s, \lambda}$ is unitary, we find 
\begin{equation*}
\begin{aligned}
\dot{\tilde{W}}_{s, \lambda} &= i \tilde{W}_{s, \lambda} \tilde{\Omega},
\end{aligned} 
\end{equation*}
where 
\begin{equation*}
\tilde{\Omega} = 2 ((X-iZ)^{-1}\tilde{\mathfrak{B}})^* [X^t\dot Z-Z^t\dot X]((X-iZ)^{-1} \tilde{\mathfrak{B}}).
\end{equation*}

Let's take a close look at $X^t \dot{Z} - Z^t \dot{X}$. Taking an $s$ derivative of this 
quantity, denoted with a prime, and using 
$(\dot X)'=\dot Z$ and $(\dot Z)'=(V-\lambda I)\dot X-X$, we find
\begin{align}
 (X^t\dot Z-Z^t\dot X)' &= Z^t\dot Z+X^t((V-\lambda I)\dot X-X)-X^t(V-\lambda I)\dot X-Z^t\dot Z\\
 &=-X^tX.
\end{align}
After integration, we arrive at
\begin{equation*}
X^t\dot Z-Z^t\dot X=-\int_0^sX^t(t,\lambda)X(t,\lambda)dt+X^t(0,\lambda)\frac{d}{d\lambda}Z(0,\lambda)-
Z^t(0,\lambda)\frac{d}{d\lambda}X(0,\lambda).
\end{equation*}
In the current setting, $X(0,\lambda)$ and $Z (0, \lambda)$ are constant in 
$\lambda$, so that
\begin{equation*}
X^t\dot Z-Z^t\dot X=-\int_0^sX^t(t,\lambda)X(t,\lambda)dt,
\end{equation*} 
and 
\begin{equation*}
\tilde{\Omega} = - 2 ((X-iZ)^{-1}\tilde{\mathfrak{B}})^* \int_0^sX^t(t,\lambda)X(t,\lambda)dt
((X-iZ)^{-1} \tilde{\mathfrak{B}}).
\end{equation*}

It's clear that $\tilde{\Omega}$ is self-adjoint, and we also claim that
it's negative definite. Indeed, if we temporarily set $A = (X-iZ)^{-1} \tilde{\mathfrak{B}}$
and $B = \int_0^sX^t(s,\lambda)X(s,\lambda)dt$ we see that 
$\tilde{\Omega} = -2 A^* B A$, where $B$ is positive definite (when 
$X$ is invertible) and $A$ is invertible. It follows immediately
that $\tilde{\Omega}$ is negative definite.

Finally, we'll show in Lemma \ref{W} below that under these conditions
the eigenvalues of $\tilde{W}_{s,\lambda}$ move clockwise on the unit 
circle as $\lambda$ increases from $\lambda^*-\varepsilon$ to 
$\lambda^*+\varepsilon$, or counterclockwise as $\lambda$ decreases 
from $\lambda^*+\varepsilon$ to 
$\lambda^*-\varepsilon$. Therefore, $\mi(\ell, \ell_1; \Gamma_3)=\dim \ker(\tilde{W}_{1, \lambda^*}+I)$, 
and so $\mi(\ell, \ell_1; \Gamma_3)=\mo(H)$.
\end{proof}

\begin{remark} As discussed in \cite{F}, p. 307, the signature of 
$\tilde{\Omega}$ corresponds precisely with the signature of the
crossing form associated with $(\ell, \ell_1)$ at any intersection. 
\end{remark}

\begin{lemma}\lb{W}
Let $\tilde{W} (\tau)$ be a smooth family of unitary $n \times n$ matrices on some interval $I$, and 
suppose $\tilde{W} (\tau)$ satisfies the differential equation 
$\frac{d}{d\tau} \tilde{W} (\tau) = i \tilde{W} (\tau) \tilde{\Omega} (\tau)$, 
where $\tilde{\Omega} (\tau)$ is continuous, self-adjoint and negative-definite. 
Then the eigenvalues of $\tilde{W} (\tau)$ move clockwise on the unit circle as $\tau$ increases.  
\end{lemma}

\begin{proof}
As a start, fix some $\tau_0 \in [0, 1]$, and denote the eigenvalues of $W(\tau_0)$
by $\{\lambda_k (\tau_0)\}_{k=1}^n$. We claim that for $\tau$ near $\tau_0$ we can express
$W(\tau)$ as 
\begin{equation*}
W(\tau) = W(\tau_0) e^{i R(\tau)}, 
\end{equation*} 
for some appropriate matrix $R(\tau)$. Indeed, we know $R(\tau)$ exists, because 
$W(\tau_0)^{-1} W (\tau)$ is invertible, and so has a logarithm. It's 
convenient to notice here that $R(\tau_0) = 0$. 

Next, we compute $W'(\tau_0)$. For this, we write 
\begin{equation*}
W (\tau) = W (\tau_0) \sum_{j=1}^{\infty} \frac{i^j}{j!} R(\tau)^j,
\end{equation*}
so that 
\begin{equation*}
W' (\tau) = W (\tau_0) \sum_{j=1}^{\infty} \frac{i^j}{j!} \frac{d}{d\tau} R(\tau)^j.
\end{equation*}
Generally, we run into a commutation problem when computing derivatives of
powers of matrices, but since $R(\tau_0) = 0$ we see that 
\begin{equation*}
\frac{d}{d\tau} R(\tau)^j  \Big|_{\tau = \tau_0} = 0,
\end{equation*}
for $j = 2, 3, \dots$. In this way, 
\begin{equation*}
W'(\tau_0) = i W (\tau_0) R' (\tau_0),
\end{equation*}
and we recognize that $\Omega (\tau_0) = R' (\tau_0)$. 

According to Theorem II.5.4 in \cite{Kato}, if $\Omega (\tau_0)$
is negative definite then the eigenvalues of $R(\tau_0)$, which 
we denote $\{r_k (\tau_0)\}_{k=1}^n$, are decreasing as $\tau$
increases at $\tau_0$. By spectral mapping, the eigenvalues of 
$e^{i R(\tau_0)}$ are $\{e^{i r_k (\tau_0)}\}_{k=1}^n$.

At this point, we proceed similarly as in \cite{F}, p. 306. We 
fix any $\theta$ so that 
$e^{i \theta} \notin \{\lambda_k^*\}_{k=1}^n$, and set 
\begin{equation*}
A (\tau) := i (e^{i \theta} I - W(\tau))^{-1} (e^{i \theta} I + W(\tau)),
\end{equation*}
for $\tau$ near $\tau_0$.
Proceeding as in \cite{F}, we claim that  
\begin{equation*} \label{claim}
A'(\tau_0) = \Big((e^{i \theta}I - W(\tau_0))^{-1} \Big)^*
2 R' (\tau_0)
\Big(e^{i \theta}I - W(\tau_0)^{-1} \Big)
\end{equation*}
To see this, we compute 
\begin{equation*}
\begin{aligned}
A' (\tau) &= -i (e^{i \theta}I - W(\tau))^{-1} (-W'(\tau))
(e^{i \theta}I - W(\tau))^{-1} (e^{i \theta}I + W(\tau)) \\
&\quad + i (e^{i \theta}I - W(\tau))^{-1} W' (\tau) \\
&= i (e^{i \theta}I - W(\tau))^{-1} W' (\tau)
\Big{\{}I + (e^{i \theta}I - W(\tau))^{-1} (e^{i \theta}I + W(\tau)) \Big{\}} \\
&= i (e^{i \theta}I - W(\tau))^{-1} W' (\tau) (e^{i \theta}I - W(\tau))^{-1}
\Big{\{}(e^{i \theta}I - W(\tau)) + (e^{i \theta}I + W(\tau)) \Big{\}} \\
&= i (e^{i \theta}I - W(\tau))^{-1} W' (\tau) (e^{i \theta}I - W(\tau))^{-1}
2e^{i \theta}.
\end{aligned}
\end{equation*}
Continuing, we see that 
\begin{equation*}
\begin{aligned}
A' (\tau) &= i (e^{i \theta}I - W(\tau))^{-1} e^{i \theta} 2 W' (\tau) (e^{i \theta}I - W(\tau))^{-1} \\
&= i \Big(e^{-i \theta} (e^{i \theta}I - W(\tau))\Big)^{-1} 2 W' (\tau) (e^{i \theta}I - W(\tau))^{-1} \\
&= i (I - e^{-i \theta} W(\tau))^{-1} 2 W' (\tau) (e^{i \theta}I - W(\tau))^{-1} \\
&= i (I - e^{-i \theta} W(\tau))^{-1} 2 W(\tau) W(\tau)^{-1} W' (\tau) (e^{i \theta}I - W(\tau))^{-1} \\
&= i \Big(W(\tau)^{-1} (I - e^{-i \theta} W(\tau))\Big)^{-1} 2 W(\tau)^{-1} W' (\tau) (e^{i \theta}I - W(\tau))^{-1}.
\end{aligned}
\end{equation*}
Now, we use the fact that $W (\tau)$ is unitary to see that 
\begin{equation*}
\begin{aligned}
A'(\tau) &= i (W(\tau)^* - e^{-i \theta} I)^{-1} 2 W(\tau)^* W' (\tau) (e^{i \theta}I - W(\tau))^{-1} \\
&= i \Big((- W(\tau) + e^{i \theta} I)^{-1}\Big)^* (- 2 W(\tau)^* W' (\tau)) \Big(e^{i \theta}I - W(\tau)\Big)^{-1}.
\end{aligned}
\end{equation*}
Finally, recalling that $W'(\tau_0) = i W(\tau_0) R'(\tau_0)$, we see that 
$R'(\tau_0) = -i W(\tau_0)^* W' (\tau_0)$, giving the claim.

We see from (\ref{claim}) that $A'(\tau_0)$ is negative definite (since $R' (\tau_0)$ is).
We conclude (again, from Theorem II.5.4 in \cite{Kato}) that the eigenvalues of $A (\tau)$
are decreasing as $\tau$ increases at $\tau_0$. 

At this point, we would like to relate the motion of the eigenvalues of $A(\tau)$ (which
we understand) to the motion of the eigenvalues of $W(\tau)$ (which determine the Maslov
index). We denote the eigenvalues of $A(\tau)$ by $\{a_k (\tau)\}_{k=1}^n$ and recall
that we are denoting the eigenvalues of $W(\tau)$ by $\{\lambda_k (\tau)\}_{k=1}^n$.
By spectral mapping, we have (with an appropriate labeling scheme)
\begin{equation*}
a_k (\tau) = i (e^{i \theta} - \lambda_k (\tau))^{-1} (e^{i \theta} + \lambda_k (\tau)),
\end{equation*}
from which we find 
\begin{equation*}
\lambda_k = - e^{i \theta} \frac{1 + ia_k}{1 - i a_k}
= e^{i(\theta + \pi)} \frac{1 + ia_k}{1 - i a_k}. 
\end{equation*}
In order to better understand this relationship, let $b_k$ satisfy
\begin{equation*}
e^{i b_k} = \frac{1 + ia_k}{1 - i a_k},
\end{equation*}
so that 
\begin{equation*}
b_k = \tan^{-1} \frac{2 a_k}{1 - a_k^2}. 
\end{equation*}

As $a_k$ moves from $-\infty$ to $-1$, $b_k$ corresponds with counterclockwise
rotation along $S^1$ from $(-1, 0)$ to $(0, -1)$. Likewise, as $a_k$ moves from
$-1$ to $+1$, $b_k$ corresponds with rotation in the counterclockwise direction
from $(0,-1)$ to $(0,1)$. Finally, as $a_k$ moves from $1$ to $+\infty$, $b_k$
corresponds with rotation from $(0,1)$ to $(-1,0)$, closing a single full loop 
around $S^1$. Summarizing, we see that there is a monotonic relationship between
the motion of $a_k$ on $\mathbb{R}$ and the motion of $e^{i b_k}$ on $S^1$. 
(This is a standard, well-known property of the Cayley Transform on $\mathbb{R}$.)   

We see, then, that at any $\tau^* \in [0, 1]$ $a_k (\tau)$ decreases through 
$\tau^*$, and correspondingly $\lambda_k (\tau)$ rotates in the clockwise
direction. Since $\tau^*$ is arbitrary, we conclude that the eigenvalues of 
$W(\tau)$ rotate monotonically clockwise as $\tau$ increases from $0$ to 
$1$.   
\end{proof}

\subsection{No crossings on $\Gamma_4$} 
Associated with $H(s)$, we introduce the operator family $L(s)$
\begin{equation*}
\begin{aligned}
L(s)u &= (-\frac{d^2}{dx^2}+s^2V(sx)-s^2\lambda)u,\\
\dom(L(s))&=\Big\{
u \in H^2 (0,1): \alpha_1 u(0)+\frac{1}{s}\alpha_2u'(0)=0; \, \beta_1 u(1)+\frac{1}{s}\beta_2u'(1)=0\Big\}.
\end{aligned}
\end{equation*}

We would like to show that there are no crossings on $\Gamma_4$ provided $\lambda_{\infty}=\lambda_{\infty}(s_0)$ is large enough. 

\begin{lemma}\lb{lem:Gamma4} 
Suppose $V \in C([0,1];\mathbb{R}^n \times \mathbb{R}^n)$ is symmetric.
For each $s_0 \in(0,1]$ there exists a positive $\lambda_{\infty}=\lambda_{\infty}(s_0)$ 
such that the path $\Phi_{s}^{\lambda}(Y_{\lambda})$ has no crossings for any fixed 
$\lambda\in(-\infty,-\lambda_\infty]$ as $s$ changes from $s_0$ to $1$. In particular, 
the path $\Phi_{s}^{\lambda}(Y_{\lambda})$ has no crossings on $\Gamma_4$.
\end{lemma}

\begin{proof}
It is enough to show that for each $s_0 \in(0,1]$ there exists a positive 
$\lambda_{\infty}=\lambda_{\infty}(s_0)$ such that $0\not\in\Sp(L(s))$ 
for any $s\in[s_0,1]$ and $\lambda\in(-\infty,-\lambda_{\infty}]$. 
In fact, we will show  that the operator $L(s)$ is positive-definite  
for any $s\in[s_0,1]$ and $\lambda\in(-\infty,-\lambda_{\infty}]$.

Fix $s_0 \in(0,1]$, and let $u\in\dom (L(s))$. We take an inner product 
(in $L^2 (0,1)$) of $L(s) u$ with $u$ and integrate by parts: 
\begin{equation}
\langle L(s)u,u\rangle_{L^2(0,1)} = 
\|u'\|^2_{L^2(0,1)} + s^2\langle (V(s x)-\lambda)u,u\rangle_{L^2(0,1)}\\
-(u(1), u'(1))_{\mathbb{R}^{n}} + (u(0), u'(0))_{\mathbb{R}^{n}}. 
\lb{dfnlGt6.2}
\end{equation}
For the boundary terms, we follow a calculation from p. 21 of 
\cite{BK}, and write 
\begin{equation*}
\begin{aligned}
(u(1), u'(1))_{\mathbb{R}^n} &= ((P_{D_1} + P_{N_1} + P_{R_1})u(1), u'(1))_{\mathbb{R}^n} \\
&= (P_{D_1} u(1), u'(1))_{\mathbb{R}^n} + (P_{N_1} u(1), u'(1))_{\mathbb{R}^n}
+ (P_{R_1} u(1), u'(1))_{\mathbb{R}^n} \\
&= 
(u(1), P_{N_1} u'(1))_{\mathbb{R}^n} + (P_{R_1}^2 u(1), u'(1))_{\mathbb{R}^n} \\
&= (P_{R_1} u(1), P_{R_1} u'(1))_{\mathbb{R}^n} = 
(P_{R_1} u(1), s \Lambda_1 P_{R_1} u(1))_{\mathbb{R}^n} \\
&= 
s (P_{R_1} \Lambda_1 P_{R_1} u(1), u(1))_{\mathbb{R}^n}.
\end{aligned}
\end{equation*}

Proceeding similarly for $(u(0), u'(0))_{\mathbb{R}^n}$ we see that 
\begin{equation*}
\begin{aligned}
- (u(1),& u'(1))_{\mathbb{R}^n} + (u(0), u'(0))_{\mathbb{R}^n}
= -s (P_{R_1} \Lambda_1 P_{R_1} u(1), u(1))_{\mathbb{R}^n}
+ s (P_{R_0} \Lambda_0 P_{R_0} u(0), u(0))_{\mathbb{R}^n} \\
&=
- s (\mathcal{P} \gamma_D u, \gamma_D u)_{\mathbb{R}^{2n}},
\end{aligned}
\end{equation*} 
where 
\begin{equation*}
\mathcal{P} = \begin{pmatrix}
- P_{R_0} \Gamma_0 P_{R_0} & 0 \\
0 & P_{R_1} \Gamma_1 P_{R_1}
\end{pmatrix},
\end{equation*}
and $\gamma_D$ will denote the Dirichlet trace 
$\gamma_D u = {u(0) \choose u(1)}$.

Let $c_B > 0$ be large enough so that 
\begin{equation*}
|(\mathcal{P} \gamma_D u, \gamma_D u)_{\mathbb{R}^{2n}}|
\le c_B \|\gamma_D u\|_{\mathbb{R}^2n}^2,
\end{equation*}
and also notice that given any $\epsilon > 0$ 
there is a corresponding $\beta (\epsilon)$ so that 
\begin{equation*}
\|\gamma_D u\|_{\mathbb{R}^{2n}}^2 \le
\epsilon \|u'\|_{L^2 (0,1)}^2 + \beta (\epsilon) \|u\|_{L^2 (0,1)}^2.
\end{equation*}
(See, e.g., \cite{BK} Lemma 1.3.8.)
In this way, we see that 
\begin{equation*}
\begin{aligned}
-s c_B \|\gamma_D u\|_{\mathbb{R}^n}^2 &\ge 
-s c_B \Big(\beta (\epsilon) \|u\|_{L^2 (0,1)}^2 + \epsilon \|u'\|_{L^2 (0,1)}^2 \Big) \\
&\ge 
- c_B \Big(\beta (\epsilon) \|u\|_{L^2 (0,1)}^2 + \epsilon \|u'\|_{L^2 (0,1)}^2 \Big),
\end{aligned}
\end{equation*}
where the second inequality uses $s \in (0,1]$.

Choose $\epsilon > 0$ small enough so that $c_B \epsilon < 1$ and set 
\begin{equation*}
\lambda_{\infty} := \|V\|_{L^{\infty} (0,1)} + (1 + c_B \beta (\epsilon)) s_0^{-2}. 
\end{equation*}
Then, 
\begin{equation*}
\begin{aligned}
s^2 \langle (V(sx) - \lambda) u, u \rangle_{L^2 (0,1)}
&= s^2 \Big(\langle V(sx) u, u \rangle_{L^2 (0,1)} - \lambda \|u\|_{L^2 (0,1)}^2 \Big) \\
&\ge
s^2 \Big(-\|V\|_{L^{\infty} (0,1)} + \lambda_{\infty} \Big) \|u\|_{L^2 (0,1)}^2.
\end{aligned}
\end{equation*}
Combining these observations, we find 
\begin{equation*}
\begin{aligned}
\langle L(s) u, u \rangle_{L^2 (0,1)} &\ge
\|u'\|_{L^2 (0,1)}^2 + (- s^2 \|V\|_{L^{\infty} (0,1)} + s^2 \lambda_{\infty}) \|u\|_{L^2 (0,1)}^2 \\
&- c_B \epsilon \|u'\|_{L^2 (0,1)}^2 - c_B \beta (\epsilon) \|u\|_{L^2 (0,1)}^2 \\
&= (1 - c_B \epsilon) \|u'\|_{L^2 (0,1)}^2 
+ \Big(s^2 \lambda_{\infty} - c_B \beta (\epsilon) - s^2 \|V\|_{L^{\infty} (0,1)} \Big) \|u\|_{L^2 (0,1)}^2. 
\end{aligned}
\end{equation*}
We have 
\begin{equation*}
\begin{aligned}
s^2 \lambda_{\infty} &- c_B \beta (\epsilon) - s^2 \|V\|_{L^{\infty} (0,1)}
= s^2 \|V\|_{L^{\infty} (0,1)} + \frac{s^2}{s_0^2} (1 + c_B \beta (\epsilon))
- c_B \beta(\epsilon) - s^2 \|V\|_{L^{\infty} (0,1)} \\
&= \frac{s^2}{s_0^2} (1 + c_B \beta (\epsilon)) - c_B \beta(\epsilon) \ge 1,
\end{aligned}
\end{equation*}
where in obtaining the final inequality we've observed $s > s_0 > 0$. 

We conclude that 
\begin{equation*}
\langle L(s) u, u \rangle_{L^2 (0,1)} \ge (1-c_B \epsilon) \|u'\|_{L^2 (0,1)}^2
+ \|u\|_{L^2 (0,1)}^2,
\end{equation*}
from which we see that for $\lambda \le - \lambda_{\infty}$, $L(s)$ is positive
definite.
\end{proof}

\subsection{Crossings on $\Gamma_1$. Asymptotic expansions as $s\to0$}

Our goal in this section is to show that the Maslov index along 
$\Gamma_1$ can be expressed as 
\begin{equation*}
\Mas (\ell, \ell_1; \Gamma_1) = - \Mor (H(s_0)) 
= - \Mor (B) - \Mor (Q (V(0)- (P_{R_0} \Lambda_0 P_{R_0})^2) Q),
\end{equation*}
where $B$ and $Q$ are as in Theorem \ref{main}. For this discussion, 
we work with the operator $H(s)$, defined in \ref{ev_equation_u},
and with the domain 
\begin{equation*}
\dom(H(s)) = 
\big\{u\in H^2(0,1): \alpha_1 u(0)+\frac{1}{s} \alpha_2 u'(0) = 0; \,
\beta_1  u(1)+\frac{1}{s} \beta_2 u'(1) = 0\big\}.
\end{equation*}
Notice that $H(0) = -\frac{d^2}{dx^2}$, with 
\begin{equation*}
\dom(H(0)) = \big\{u\in H^2(0,1): P_{D_i} u(i) = 0, \,
P_{N_i} u'(i) = 0, \, P_{R_i} u'(i) = 0, \, i=0,1 \big\}.
\end{equation*}

If---as in the Dirichlet case---$H(0)$ does not have zero as an eigenvalue, 
then there cannot be any crossings along $\Gamma_1$. On the other hand, 
if zero is an eigenvalue of $H(0)$---as, for example, in the Neumann-based 
cases---there will be an associated family of eigenvalues of $H(s)$ for small 
$s$. Our ultimate goal is an asymptotic formula for the eigenvalues of $H(s)$ 
that bifurcate from a zero eigenvalue of $H(0)$ as $s\to0$. As a start, 
we characterize the eigenspace associated with the zero eigenvalue.

\begin{lemma} \label{zero_eigenspace} For $H(0)$ as defined above, 
zero is an eigenvalue of $H(0)$ if and only if 
$(\ker P_{D_0}) \cap (\ker P_{D_1}) \ne \{0\}$. Moreover, if 
zero is an eigenvalue of $H(0)$ then the eigenspace associated
with zero is precisely the set of constant vectors characterized
by this intersection. I.e.,  
\begin{equation*}
\ker H(0) = (\ker P_{D_0}) \cap (\ker P_{D_1}).
\end{equation*}
\end{lemma}

\begin{proof} It's clear that solutions of $H(0) = 0$ have 
the form 
\begin{equation*}
u(x) = ax + b, \quad a, b \in \mathbb{R}^n.
\end{equation*}
According to our boundary conditions, we have $P_{D_0} b = 0$,
$P_{N_0} a = 0$, $P_{R_0} a = 0$, 
$P_{D_1} (a+b) = 0$, $P_{N_1} a = 0$, and $P_{R_1} a = 0$.
Since $P_{D_0} + P_{N_0} + P_{R_0} = I$ (and similarly 
for the right boundary condition), we see that 
$(I - P_{D_0}) a = 0$ and $(I - P_{D_1}) a = 0$.

We have, then, 
\begin{equation*}
(a,b)_{\mathbb{R}^n} = (a,P_{D_0} b + (I-P_{D_0})b)_{\mathbb{R}^n}
= (a, (I-P_{D_0})b)_{\mathbb{R}^n}
= ((I-P_{D_0}) a, b)_{\mathbb{R}^n} = 0,
\end{equation*}  
and similarly $(a, a+b)_{\mathbb{R}^n} = 0$. It follows immediately
that $|a|^2 = 0$, so that $a = 0$ and $u(x) = b$. Finally, we
see that since $a = 0$ we must have both $P_{D_0} b = 0$ and 
$P_{D_1} b = 0$, and also that if these conditions are satisfied 
for $b \ne 0$ then zero is certainly an eigenvalue of $H(0)$. 
\end{proof}

\begin{remark} \label{intersection}
In what follows, we generally won't introduce any notation to distinguish
between $(\ker P_{D_0}) \cap (\ker P_{D_1})$ as a subspace of $\mathbb{C}^n$
or a subspace of $L^2 (0,1)$. We will denote the dimension of this 
intersection by $d$. I.e., $d = \dim [(\ker P_{D_0}) \cap (\ker P_{D_1})]$.
\end{remark}

Now, Consider the sesquilinear form ${h}(s)$ on $L^2(0,1)$, 
defined for $s\in[0,1]$ by 
\begin{align}
{h}(s)(u,v) &= 
\langle u',v'\rangle_{L^2(0,1)}+s^2\langle V(s x)u,v\rangle_{L^2(0,1)}
- s (\mathcal{P} \gamma_D u, \gamma_D u)_{\mathbb{C}^{2n}},
\no\\
\dom({h}(s)) &= \{(u,v) \in H^1(0,1)\times H^1(0,1): 
P_{D_i} u(i) = 0, P_{D_i} v(i) = 0\}, \lb{dfnlGt}
\end{align}
where $\gamma_D$ is defined in the proof of Lemma \ref{lem:Gamma4}. 
(See Theorem 1.4.11 in \cite{BK} for a discussion of why $h(s)$ with
the domain specified here is that natural quadratic form to associate
with $H(s)$.)

\begin{remark} Notice that at this point we begin working with 
complex inner products in anticipation of employing complex
analytic tools, including especially Riesz projections. We keep
in mind that even though complex-valued functions and vectors 
are now allowed, all inner products will ultimately be evaluated
at real-valued functions and vectors. 
\end{remark}

Following the general discussion of holomorphic families of closed, 
unbounded operators in \cite[Section VII.1.2]{Kato}, we introduce our 
next definition.

\begin{definition}\label{cont}
A family of closed, not necessarily bounded, operators $\{T(s)\}_{s\in \Sigma}$ on a Hilbert space $\cX$ 
is said to be continuous on an interval $\Sigma_0 \subset \Sigma$ if there exists 
a Hilbert space $\cX'$ and continuous families of operators $\{U(s)\}_{s\in \Sigma_0}$ 
and $\{W(s)\}_{s\in \Sigma_0}$ in $\cB(\cX',\cX)$ such that $U(s)$ is a one-to-one 
map of $\cX'$ onto $\dom(T(s))$ and the identity $T(s)U(s)=W(s)$ holds for all $s\in \Sigma_0$.
\end{definition}

Before applying this definition, we recall that the continuity of $V$ implies 
\begin{equation}\label{unif1}
\sup_{x\in[0,1]}\|V(s x)-V(s_0x)\|_{\bbR^{n\times n}}\to 0 \,\,\,\hbox{as}\,\,\, s\to s_0\,\text{ 
for any $s_0\in[0,1]$}.
\end{equation}

\begin{lemma}
Assume $V \in C([0, 1])$ is a symmetric matrix in $\mathbb{R}^{n \times n}$. 
Then the family $\{H(s)\}_{s\in [0,1]}$ is continuous near $0$; 
that is, on some interval $\Sigma_0$ that contains $0$.
\end{lemma}

\begin{proof}
We notice that formally we can write 
\begin{equation} \label{beginning_and_end}
H(s)(H(s)+I)^{-1}=I-(H(s)+I)^{-1}.
\end{equation}
In this way it is sufficient to establish 
that $U(s):=(H(s)+I)^{-1}$ is a continuous family of operators.

First, we note that it's clear from our construction of ${h}$
that we have the identity
\begin{equation*}
{h} (s) (u,v) = \langle H(s) u, v \rangle_{L^2 (0,1)},
\end{equation*}
for all $s \in [0,1]$. In particular, we have 
\begin{equation*}
\langle (H(0) + I) u, u \rangle_{L^2 (0,1)}
= ({h} (0) + 1) (u, u) = 
\|u'\|_{L^2 (0,1)}^2 + \|u\|_{L^2 (0,1)}^2, 
\end{equation*}
where we have used the convenient operator 
notation $1 (u,u) = \|u\|_{L^2 (0,1)}^2$. It 
follows that the operator $(H(0) + I)$ is self-adjoint, invertible
and positive definite, with a well-defined square root,
which we denote 
\begin{equation*}
G := (H(0) + I)^{1/2}; \quad G: \dom (h(s)) \to L^2 (0,1).
\end{equation*} 
We notice that for any $u \in \dom (H(s))$ we have 
\begin{equation} \label{Gu_squared}
\begin{aligned}
\|Gu\|_{L^2 (0,1)}^2 &= \langle Gu, Gu \rangle_{L^2 (0,1)}
= \langle G^2u, u \rangle_{L^2 (0,1)} \\
&= \langle H(0)u + u, u \rangle_{L^2 (0,1)}
= \|u'\|_{L^2 (0,1)}^2 + \|u\|_{L^2 (0,1)}^2 = \|u\|_{H^1 (0,1)}^2,
\end{aligned}
\end{equation}
from which we conclude that $G$ is an (invertible) isometry. 

Now, take any $u,v\in L^2(0,1)$ such that $\|u\|_{L^2(0,1)},\|v\|_{L^2(0,1)}\leq1$,
and compute  
\begin{align}\label{unif}
\big|(\mathcal{P} \gamma_D G^{-1}u,\gamma_DG^{-1}v)_{\mathbb{C}^{2n}}\big|
&\leq
C_1 \|\gamma_D G^{-1}u\|_{\mathbb{C}^{2n}}\|\gamma_D G^{-1}v\|_{\mathbb{C}^{2n}} \\
&\le 
C_2 \|G^{-1} u\|_{H^1 (0,1)} \|G^{-1} v\|_{H^1 (0,1)}
= C_2 \|u\|_{L^2 (0,1)} \|v\|_{L^2 (0,1)} 
\le C_2, 
\end{align}
where we've used the observation from the proof of Lemma \ref{lem:Gamma4}
that $\gamma_D$ is bounded as a map from $H^1 (0,1)$ to $\mathbb{C}^{2n}$.

We introduce a new sesquilinear form
\begin{align}
&\widetilde{{h}}(s)(u,v)
:= {h} (s)\left(G^{-1}u,G^{-1}v \right),\\
&\dom(\widetilde{{h}}(s))=L^2(0,1)\times L^2(0,1).
\end{align}
From \eqref{Gu_squared} and \eqref{unif} it is easy to see that 
$\widetilde{{h}}(s)$ is bounded on  
$L^2(0,1)\times L^2(0,1)$. Let $\widetilde{H}(s)\in\cB(L^2(0,1))$ 
be the self-adjoint operator associated with $\widetilde{{h}}(s)$ 
by the First Representation Theorem \cite[Theorem VI.2.1]{Kato}. Then 
\begin{align}\lb{LLtil}
\langle \widetilde{H}&(s)u,v\rangle_{L^2(0,1)}
=\widetilde{{h}}(s)(u,v)\\
&={h}(s)\left(G^{-1}u,G^{-1}v\right) 
\,\text{ for all} \,u,v\in L^2(0,1).\no
\end{align}
Taking into account \eqref{unif} and \eqref{unif1}, we conclude that
\begin{equation*}
\langle\widetilde{H}(s)u,v\rangle_{L^2(0,1)}
\rightarrow\langle\widetilde{H}(s_0)u,v\rangle_{L^2(0,1)}\,\,\hbox{as}\,\,s\to s_0
\end{equation*}
uniformly with respect to $u$ and $v$ satisfying $\|u\|_{L^2(0,1)},\|v\|_{L^2(0,1)}\leq1$. 
Hence 
\begin{equation}
\|\widetilde{H}(s)-\widetilde{H}(s_0)\|_{\cB(L^2(0,1))}\rightarrow0\,\,\hbox{as}\,\,s\to s_0,
\end{equation}
which implies $\widetilde{H}(s)\in\cB(L^2(0,1))$ is a continuous family on $[0,1]$.

Replacing $u$ in \eqref{LLtil} by $Gu$ (and similarly for $v$), we conclude that
\[
{h}(s)(u,v)=\left<\widetilde{H}(s)Gu,Gv\right>_{L^2(0,1)}
\] 
for any $u,v\in \dom (h(s))$.
Therefore, cf.\ \cite[VII-(4.4), (4.5)]{Kato}, for all $u\in\dom(H(s))$
\beq\lb{simLL}
H(s)u=G\widetilde{H}(s)Gu,
\enq
when $G$ is viewed as an unbounded, self-adjoint operator on $L^2(0,1)$.
Adding $I$ to both sides, we find 
\begin{equation*}
H(s)+I = G\widetilde{H}(s) G + I 
= G \Big(\widetilde{H}(s)+G^{-2} \Big) G. 
\end{equation*}
Now, $H(0) + I = G^2$, so $G^2 = G(\tilde{H} (0) + G^{-2}) G$, 
giving $\widetilde{H} (0) + G^{-2} = I$. 
We've seen that $\tilde{H} (s) \in \mathcal{B} (L^2 (0,1))$ is 
a continuous family, and since $\tilde{H} (0) + G^{-2} = I$ it follows
that $\widetilde{H} (s) + G^{-2}$ is boundedly invertible for $s$ 
near 0. We conclude that near $s = 0$ 
\begin{equation*}
(H(s)+I)^{-1}=G^{-1}\Big(\widetilde{H}(s)+G^{-2}\Big)^{-1}G^{-1}.
\end{equation*}
Thus $U(s)=(H(s)+I)^{-1}$ and $W(s)=I-U(s)$ are both continuous families
near $s = 0$, and it is now clear that 
\begin{equation}\label{UV}
H(s)U(s)=W(s)\,\text{ for $s$ near $0$}.
\end{equation} 
Hence, (\ref{beginning_and_end}) is justified, and 
according to Definition \ref{cont} the family $\{H(s)\}$ is continuous near $0$. 
\end{proof}

For $\zeta \in \mathbb{C} \setminus \sigma(H(s))$, we denote the 
resolvent 
\[
R(\zeta,s)=\big(H(s)-\zeta I\big)^{-1} \in \cB(L^2(0,1)).
\]

\begin{lemma}\label{rescont}
Let $\zeta \in \bbC \setminus \sigma(H(0))$. 
Then $\zeta \in \bbC \setminus \sigma (H(s))$ 
for $s$ near $0$. Moreover, the function 
$s \mapsto R(\zeta,s)$ is continuous for 
$s$ near $0$, uniformly for $\zeta$ in compact 
subsets of $\bbC \setminus \sigma (H(s))$.
\end{lemma}

\begin{proof}
Let $\zeta \in \bbC \setminus \sigma(H(0))$. Since 
$H(s) U(s) = W(s)$, we have (for $s$ near 0)
\begin{equation}\lb{eq1.26}
(H(s)-\zeta I) U(s) 
= W(s)-\zeta U(s). 
\end{equation}
The operator 
\begin{equation*}
W(0)-\zeta U(0)=(H(0)-\zeta I) U(0)
=(H(0)-\zeta I)(H(0)+I)^{-1}
\end{equation*}
is a bijection of $L^2(0,1)$ onto $L^2(0,1)$ (because
$H(0)+I$ and $H(0)-\zeta I$ are both boundedly 
invertible). By continuity, the operator $W(s)-\zeta U(s)$
is boundedly invertible for $s$ near $0$. 
This implies that $(H(s) - \zeta I) U(s)$ is boundedly
invertible with inverse $U(s)^{-1} (H(s) - \zeta I)^{-1}$.
In this way, we see that 
\begin{equation}
(H(s)-\zeta)^{-1}=U(s) (W(s)-\zeta U(s))^{-1},
\end{equation}
the product of two bounded operators. Hence, 
$\zeta \in \mathbb{C} \setminus \sigma (H(s))$,
and the function 
$s \mapsto R(\zeta,s)$ is continuous for $s$ near $0$ 
in the operator norm, uniformly in $\zeta$.
\end{proof}

Our next lemma gives an asymptotic result for the difference 
of the resolvents of the operators $H(s)$ and $H(0)$ as $s\to0$, 
which involves the value $V(0)$ of the potential at zero. We
observe at the outset that since $R(\zeta, 0)$ is a bounded
linear operator, it has a bounded linear adjoint (both on 
$L^2 (0,1)$). Consider the composite map 
$\gamma_D R(\zeta, 0)^*: L^2 (0,1) \to \mathbb{C}^{2n}$,
which for any $u \in L^2 (0,1)$, $z \in \mathbb{C}^{2n}$
satisfies 
\begin{equation} \label{adjoint_relation}
(z,\gamma_D R(\zeta, 0)^* u)_{\mathbb{C}^{2n}}
= \langle (\gamma_D R(\zeta, 0)^*)^* z, u  \rangle_{L^2 (0,1)}.
\end{equation}  

\begin{lemma} \label{RminusR}
If $\zeta \in \bbC \setminus \sigma (H(0))$ and $\|u\|_{L^2(0,1)}\le1$, then 
\begin{equation}\label{resolvent}
\begin{split}
R(\zeta, s)u - R(\zeta, 0)u 
&= s[\gaD   R(\zeta, 0)^*]^* \mathcal{P} \gaD R(\zeta, 0)u
- s^2R(\zeta, 0)V(0)R(\zeta, 0)u \\
&\quad+s^2[\gaD   R(\zeta, 0)^*]^* \mathcal{P} 
\gaD[\gaD   R(\zeta, 0)^*]^* \mathcal{P} \gaD  R(\zeta, 0)u+r(s),
\end{split}
\end{equation}
where $\|r(s)\|_{L^2(0,1)}=\mathrm{o}(s^{2})$ as $s\to0$, uniformly for $\zeta$ 
in compact subsets of $\bbC \setminus \sigma (H(0))$ and $\|u\|_{L^2(0,1)}\le1$.
\end{lemma}

\begin{proof}
We recall that $\zeta \in \bbC \setminus \sigma(H(s))$ for $s$ near $0$ 
by Lemma \ref{rescont}, since $\zeta \in \bbC \setminus \sigma (H(0))$. 
For $\|u\|_{L^2(0,1)}\le1$, we set $w := R(\zeta, s)u - R(\zeta, 0)u$. 
Since $R(\zeta, 0): L^2 (0,1) \to \mathcal{D} (H(0))$, we see that 
for any $u \in L^2 (0,1)$, $v \in H^1 (0,1)$ we have 
\begin{equation*}
{h} (0) (R(\zeta, 0)u, v) = \langle H(0) R(\zeta, 0) u, v \rangle_{L^2 (0,1)}, 
\end{equation*}
and likewise (since $H(0)$ is self-adjoint) 
\begin{equation*}
{h}(0) (v, R(\zeta, 0)^* u) = \langle v, H(0) R(\zeta, 0)^* u \rangle_{L^2 (0,1)}. 
\end{equation*}
In this way, we have 
\begin{equation*}
\begin{aligned}
({h} (0) - \zeta) (R(\zeta,0)u, v) 
&= \langle (H(0) - \zeta I) R(\zeta, 0) u, v \rangle_{L^2 (0,1)} \\
&= \langle u, v \rangle_{L^2 (0,1)},
\end{aligned}
\end{equation*}
and likewise 
\begin{equation*}
\begin{aligned}
({h}(0) - \zeta) (v, R(\zeta,0)^* u) 
&= \langle v, (H(0) - \bar{\zeta} I) R(\zeta, 0)^* u \rangle_{L^2 (0,1)} \\
&= \langle v, u \rangle_{L^2 (0,1)}.
\end{aligned}
\end{equation*}

We compute
\begin{equation*}
\begin{aligned}
\langle w,v\rangle_{L^2(0,1)} &= ({h} (0) - \zeta) (w,R(\zeta, 0)^*v)
=({h} (0)-\zeta) (R(\zeta, s)u - R(\zeta, 0)u, R(\zeta, 0)^*v)\\
&=({h} (0) - \zeta) (R(\zeta, s)u,R(\zeta, 0)^*v)
-({h} (0)-\zeta) (R(\zeta, 0)u,R(\zeta, 0)^*v).
\end{aligned}
\end{equation*}
At this stage, we notice that 
\begin{equation*}
{h} (s) (u,v) = {h} (0) (u,v) + 
s^2 \langle V(sx) u, v \rangle_{L^2 (0,1)} 
- s (\mathcal{P} \gamma_D u, \gamma_D v)_{\mathbb{C}^{2n}}.
\end{equation*}
Using this, we can write 
\begin{equation*}
\begin{aligned}
\langle w,v\rangle_{L^2(0,1)} 
&= ({h} (s) - \zeta) (R(\zeta, s)u,R(\zeta, 0)^*v)
-s^2 \langle V(sx) R(\zeta, s) u, R(\zeta, 0)^* v \rangle_{L^2 (0,1)} \\
&+ s (\mathcal{P} \gamma_D R(\zeta, s) u, \gamma_D R(\zeta, 0)^* v)_{\mathbb{C}^{2n}}
- ({h} (0) - \zeta) (R(\zeta, 0)u,R(\zeta, 0)^*v).
\end{aligned}
\end{equation*}
In this way, we obtain 
\begin{equation*}
\begin{aligned}
\langle w,v\rangle_{L^2(0,1)}
&= \langle (H(s) - \zeta I) R(\zeta, s)u, R(\zeta, 0)^* v\rangle_{L^2 (0,1)} 
- s^2 \langle V(sx) R(\zeta, s) u, R(\zeta, 0)^* v \rangle_{L^2 (0,1)} \\
&+ s (\mathcal{P} \gamma_D R(\zeta, s) u, \gamma_D R(\zeta, 0)^* v)_{\mathbb{C}^{2n}}
- \langle (H (0) - \zeta I) R(\zeta, 0)u, R(\zeta, 0)^*v \rangle_{L^2 (0,1)} \\
&=\langle u, R(\zeta, 0)^* v\rangle_{L^2 (0,1)} 
- s^2 \langle V(sx) R(\zeta, s) u, R(\zeta, 0)^* v \rangle_{L^2 (0,1)} \\
&+ s (\mathcal{P} \gamma_D R(\zeta, s) u, \gamma_D R(\zeta, 0)^* v)_{\mathbb{C}^{2n}}
- \langle u, R(\zeta, 0)^*v \rangle_{L^2 (0,1)} \\
&= - s^2 \langle V(sx) R(\zeta, s) u, R(\zeta, 0)^* v \rangle_{L^2 (0,1)} 
+ s (\mathcal{P} \gamma_D R(\zeta, s) u, \gamma_D R(\zeta, 0)^* v)_{\mathbb{C}^{2n}}. 
\end{aligned}
\end{equation*}

Using (\ref{adjoint_relation}), we find 
\begin{equation*}
\langle w,v\rangle_{L^2(0,1)} 
= - s^2 \langle R(\zeta, 0) V(sx) R(\zeta, s) u, v \rangle_{L^2 (0,1)} 
+ s ((\gamma_D R(\zeta, 0)^*)^* \mathcal{P} \gamma_D R(\zeta, s) u,  v)_{L^2 (0,1)}.
\end{equation*}
Since this is true for all $v \in L^2 (0,1)$, we have 
\begin{equation*}
w = -s^2 R(\zeta, 0) V(sx) R(\zeta, s) u 
+ s (\gamma_D R(\zeta, 0)^*)^* \mathcal{P} \gamma_D R(\zeta, s) u,
\end{equation*}
and recalling the definition of $w$, we arrive at  
\begin{equation} \label{formula}
R(\zeta, s)u = R(\zeta, 0)u - s^2 R(\zeta, 0) V(s x) R(\zeta, s)u 
+ s(\gaD R(\zeta, 0)^*)^* \mathcal{P} \gaD  R(\zeta, s)u.
\end{equation}

Replacing $R(\zeta, s)u$ in the right-hand  side of \eqref{formula} 
again by \eqref{formula} yields
\begin{equation} \label{fformula}
\begin{aligned} 
R(\zeta, s)u &= R(\zeta, 0)u
-s^2 R(\zeta, 0) V(s x) \Big( R(\zeta, 0)u - s^2 R(\zeta, 0)V(s x)R(\zeta, s)u \\
&\quad + s(\gaD R(\zeta, 0)^*)^* \mathcal{P} \gaD  R(\zeta, s)u \Big) \\
&+ s(\gaD R(\zeta, 0)^*)^* \mathcal{P} \gaD \Big( R(\zeta, 0)u - s^2 R(\zeta, 0) V(s x) R(\zeta, s)u \\
&\quad +s (\gaD R(\zeta, 0)^*)^* \mathcal{P} \gaD  R(\zeta, s) u\Big) \\
&= R(\zeta, 0)u + s(\gaD R(\zeta, 0)^*)^* \mathcal{P} \gaD R(\zeta, 0)u
-s^2 R(\zeta, 0) V(0) R(\zeta, 0)u \\
&\quad + s^2(\gaD R(\zeta, 0)^*)^* \mathcal{P} \gaD (\gaD R(\zeta, 0)^*)^* \mathcal{P} \gaD R(\zeta, 0)u+r(s),
\end{aligned}
\end{equation}
where
\begin{equation}
\begin{split}
r(s) = -&s^2 R(\zeta, 0) \big(V(s x)-V(0)\big) R(\zeta, 0)u\\
&+ s^2(\gaD R(\zeta, 0)^*)^* \mathcal{P} \gaD(\gaD R(\zeta, 0)^*)^* \mathcal{P} \gaD \big(R(\zeta, s)u-R(\zeta, 0)u\big) \\
&\qquad- s^3 R(\zeta, 0) V(s x) (\gaD R(\zeta, 0)^*)^* \mathcal{P} \gaD  R(\zeta, s)u \\
&\qquad\qquad-s^3 (\gaD R(\zeta, 0)^*)^* \mathcal{P} \gaD R(\zeta, 0) V(s x) R(\zeta, s)u \\
&\qquad\qquad\qquad+ s^4 R(\zeta, 0) V(s x) R(\zeta, 0) V(s x) R(\zeta, s)u.
\end{split}
\end{equation}
Finally, we remark that $\|w\|_{L^2(0,1)}\to0$ and $\|R(\zeta,s)\|_{\cB(L^2(0,1))}$ is bounded as $s\to0$ by Lemma \ref{rescont} 
and thus, using \eqref{unif1} for $s_0=0$, we conclude that $\|r(s)\|_{L^2(0,1)}=\mathrm{o}(s^{2})$ as $s\to0$, 
uniformly for $\zeta$ in compact subsets of $\bbC \setminus \sigma (H(s))$ and $\|u\|_{L^2(0,1)}\le1$.
\end{proof}

We've already noted that $H(0)$ may have $\lambda = 0$ as an eigenvalue 
(for example, in the Neumann-based case), and our next goal is to 
understand the corresponding family of eigenvalues $\{\lambda_j (s)\}$,
with $\lambda_j (0) = 0$. To begin, we will separate the spectrum of $H(s)$. 
First, we note that $0$ is the only possible nonpositive eigenvalue in 
$\sigma (H(0))$. We would like to appeal to the continuity of eigenvalues
with respect to $s$, but since $H(s)$ is unbounded we must take care with 
our argument. We will proceed by shifting the spectrum so that it lies 
entirely to the left of 0, and then inverting our operator to work with 
a resolvent (which will be bounded). 

We clarify that in contrast with the setting of Lemma \ref{lem:Gamma4},
we are concerned here with eigenvalues of $H(s)$ so that $H(s) u = \lambda u$
(i.e., the $s^2$ scaling from Lemma \ref{lem:Gamma4} does not appear on 
$\lambda$). Nonetheless, a calculation similar to the proof of 
Lemma \ref{lem:Gamma4} shows that any eigenvalue of $H (s)$ must 
satisfy 
\begin{equation*}
\lambda \ge - (\|V\|_{L^{\infty} (0,1)} + c_B \beta (\epsilon)),
\end{equation*}  
for constants $c_B$ and $\beta (\epsilon)$ that arise precisely
as in the proof of Lemma \ref{lem:Gamma4}. It's clear, then, that
there exists a value $\Lambda > 0$ sufficiently large so that 
$-\Lambda/2 < \lambda$ for all $\lambda \in \sigma (H(s))$
and $s \in [0,1]$.
By the spectral mapping theorem, we infer
\beq \lb{spmth}
\sigma \big((-\Lambda-H(s))^{-1}\big)\setminus\{0\}
= \big\{(-\Lambda-\lambda)^{-1}:
\lambda \in \sigma (H(s)) \big\}, \, s\in[0,1].
\enq
In particular, if $0 \in \sigma(H(0))$, 
then $-1/\Lambda \in \sigma \big((-\Lambda-H(0))^{-1}\big)$. 

Now fix a sufficiently small $\varepsilon \in (0,1/(2\Lambda))$ such that 
the disc of radius $2\varepsilon$ centered at the point $-1/\Lambda$  
does not contain any other eigenvalues in $\sigma \big((-\Lambda-H(0))^{-1}\big)$ 
except  $-1/\Lambda$. Using Lemma \ref{rescont} we know that  
$(-\Lambda-H(s))^{-1} \to (-\Lambda-H(0))^{-1}$ in $\cB(L^2(0,1))$ as $s\to0$. 
By the upper semicontinuity of the spectra of bounded operators, see, e.g., 
\cite[Theorem IV.3.1]{Kato}, there exists a $\delta = \delta(\varepsilon)$ 
such that if $s \in [0,\delta]$, then 
\beq\lb{spincl}
\sigma \big((-\Lambda-H(s))^{-1}\big) 
\subset \{\mu: \dist\big(\mu,\sigma \big((-\Lambda-H(0))^{-1}\big)\big)<\varepsilon\big\}.
\enq

In the remaining part of this section we take $s \le \delta$. 
Let $\{\nu_\ell (s)\}_{\ell=1}^{\tilde{n}} \subset \sigma \big((-\Lambda-H(s))^{-1}\big)$ 
denote the eigenvalues of $(-\Lambda-H(s))^{-1}$ which are located inside of the disc of 
radius $\varepsilon$ centered at the point $-1/\Lambda$, and let 
$\lambda_\ell(s)=-\Lambda-1/\nu_\ell(s)$ be the respective eigenvalues of $H(s)$.  
Let $\gamma$ be a small circle centered at zero which encloses the eigenvalues 
$\lambda_\ell(s)$ for all $\ell=1,\dots,\tilde{n}$ and $s\in[0,\delta]$ and separates them 
from the rest of the spectrum of $H(s)$. By choosing $\varepsilon$ sufficiently small, 
we can ensure that $\{\lambda_\ell(s)\}_{l=1}^{\tilde{n}}$ are precisely the 
eigenvalues bifurcating from $\lambda (0) = 0$, and also that $\gamma$ separates 
$0 \in \sigma (H(0))$ from the rest of the spectrum of $H(0)$.  

We denote by $P_0$ the orthogonal Riesz projection for $H(0)$ corresponding to the 
eigenvalue $0 \in \sigma (H(0))$, 
with $\ran(P_0)=\ker(H(0)) = (\ker P_{D_0}) \cap (\ker P_{D_1})$. 
(If $0 \notin \sigma (H(0))$ then $P_0 \equiv 0$.) 
Also, we let $\{P(s)\}_{s \in [0, \delta]}$ denote the family of 
Riesz spectral protections for $H(s)$ corresponding to the eigenvalues 
$\{\lambda_j (s)\}_{j=1}^d \subset \sigma (H(s))$, where $d$ denotes 
the dimension of the subspace $(\ker P_{D_0}) \cap (\ker P_{D_1})$.
That is, 
\beq\lb{dfnRPr}
P_0=\frac{1}{2\pi i}\int_\gamma(\zeta-H(0))^{-1}\,d\zeta,\,
P(s)=\frac{1}{2\pi i}\int_\gamma(\zeta-H(s))^{-1}\,d\zeta,
\enq
where $\gamma$ encloses the set $\{\lambda_j (s)\}_{j=1}^d$.

Our objective is to establish an asymptotic formula for the eigenvalues 
$\lambda_j (s)$ as $s\to0$ similar to \cite[Theorem II.5.11]{Kato}, 
which is valid for families of bounded operators on finite-dimensional 
spaces. We stress that one cannot directly use a related result 
\cite[Theorem VIII.2.9]{Kato} for families of unbounded operators, 
as the $s$-dependence of $H(s)$ in our case is more complicated than 
allowed in the latter theorem. We are thus forced to mimic the main 
strategy of \cite{Kato} in order to extend the relevant results to 
the family $\{H(s)\}_{s\in[0,\delta]}$. 

Keeping in mind that our main goal for $\Gamma_1$ is to count the 
number of negative eigenvalues of the operator $H(s)$ for $s$ near zero,
we next establish the following claim.

\begin{claim} \label{HpvP}
For $s \in [0, \delta]$, the number of negative eigenvalues of $H(s)$
is equivalent to the number of negative eigenvalues of $H(s) P(s)$; that
is, the restriction of $H(s)$ to the finite-dimensional subspace 
$\ran(P(s))$.
\end{claim}

\begin{proof}
By the spectral mapping theorem \eqref{spmth}, $\lambda<0$ is in $\sigma (H(s))$ 
if and only if $(-\Lambda-\lambda)^{-1}<-1/\Lambda$. Thus for $s$ near zero 
the negative eigenvalues of $H(s)$ are in one-to-one correspondence with the 
eigenvalues $\nu_j(s) \in \sigma \big((-\Lambda-H(s))^{-1}\big)$ that satisfy 
the inequality $\nu_j(s)<-1/\Lambda$, and therefore with the negative eigenvalues 
among $\lambda_j(s) \in \sigma \big(H(s)P(s)\big)$ as claimed. 
\end{proof}

Next, we would like to work with a Neumann-type expansion for $R(\zeta, 0)$. 
From \cite[Section III.6.5]{Kato}, we can write
\begin{equation}\lb{decomp}
R(\zeta,0)=(-\zeta)^{-1} P_0 + \sum_{n=0}^\infty \zeta^n S^{n+1},
\end{equation}	
where
\begin{equation} \lb{decomp1}
S=\frac{1}{2\pi i}\int_\gamma\zeta^{-1}R(\zeta,0)\,d\zeta 
\end{equation}
is the {\it reduced resolvent} for the operator $H(0)$ in $L^2(0,1)$ (this uses
equations (III.6.32) and (III.6.33) in \cite{Kato}). Moreover, we have
from \cite{Kato} the useful relation $P_0 S = S P_0=0$. (We'll say much more 
about the nature of the reduced resolvent at the end of this section.)

We introduce the notation 
\begin{equation} \label{D_defined}
D(s) = P(s) - P_0 = - \frac{1}{2 \pi i} \int_{\gamma} R(\zeta, s) - R(\zeta, 0) d\zeta, 
\end{equation}
and it's clear from Lemma \ref{RminusR} that this is $\mathbf{O} (s)$. This implies
that $I - D(s)^2$ is strictly positive for $s$ near 0, and following \cite[Section I.4.6]{Kato}, 
we may introduce mutually inverse operators $U(s)$ and $U(s)^{-1}$ in $\cB(L^2(0,1))$ as follows:
\begin{equation}\label{dfnUUinv}
\begin{split}
U(s)& = (I-D^2(s))^{-1/2}\big((I-P(s))(I-P_0) + P(s)P_0\big),\\
U(s)^{-1}&=(I-D^2(s))^{-1/2}\big((I-P_0)(I-P(s))+P_0 P(s)\big),
\end{split}
\end{equation}
for which 
\begin{equation}\label{up}
U(s)P_0=P(s)U(s)
\end{equation}
(equation (I.4.42) in \cite{Kato}).
We see that $U(s)$ is an isomorphism of the $d$-dimensional subspace 
$\ran({P_0})$ onto the subspace $\ran({P(s)})$. 

We isolate the main technical steps of our perturbation analysis 
in the following lemma, for which the statement and proof have been 
adapted with only minor changes from \cite{CJLS2014}.

\begin{lemma} \lb{lem:simile} 
Let $P_0$ be the Riesz projection for $H(0)$ onto the subspace $\ran({P_0})=\ker(H(0))$ 
and $P(s)$ the respective Riesz projection for $H(s)$ from \eqref{dfnRPr}. 
Let $S$ be the reduced resolvent for $H(0)$ defined in \eqref{decomp1}, 
and let the transformation operators $U(s)$ and $U(s)^{-1}$ be defined in \eqref{dfnUUinv}. Then
\begin{align} \label{pulup}
\begin{split}
P_0 U(s)^{-1} &H(s) P(s) U(s) P_0
= - s (\gaD P_0)^* \mathcal{P} \gaD P_0 + s^2 P_0 V(0) P_0 \\
&\qquad- s^2(\gaD P_0)^* \mathcal{P} \gaD(\gaD S)^* \mathcal{P} \gaD P_0 
+ \mathrm{o}(s^2)
\text{ as $s\to0$}.
\end{split}
\end{align}
\end{lemma}

\begin{proof} We will split the proof into several steps.

{\em Step 1.}\, We first claim the following four asymptotic relations for $\zeta \in \gamma$:
\begin{align}
\label{rp}
R(\zeta, s)P_0 &= (-\zeta)^{-1} P_0 + s (-\zeta)^{-1} (\gaD R(\zeta, 0)^*)^* \mathcal{P} \gaD P_0 + \mathrm{o}(s)_u,\\
P_0 R(\zeta, s)
  &= (-\zeta)^{-1} P_0 + s (-\zeta)^{-1}(\gaD P_0)^* \mathcal{P} \gaD R(\zeta, 0) + \mathrm{o}(s)_u, \label{eq1.63}\\
P_0 R(\zeta, s)P_0 &= (-\zeta)^{-1}P_0 + s (-\zeta)^{-2}(\gaD P_0)^* \mathcal{P} \gaD P_0
- s^2(-\zeta)^{-2} P_0 V(0) P_0 \nonumber \\
&\quad+ s^2(-\zeta)^{-2}(\gaD P_0)^* \mathcal{P} \gaD(\gaD R(\zeta, 0)^*)^* \mathcal{P} \gaD P_0 
+\mathrm{o}(s^{2})_u, \label{prp} \\
(I-P_0) R(\zeta, s)P_0 &= s (-\zeta)^{-1}(I-P_0) (\gaD R(\zeta, 0)^*)^* \mathcal{P} \gaD P_0 
+ \mathrm{o}(s)_u.\label{7.44.2}
\end{align}
Here and below we write $\mathrm{o}(s^\alpha)_u$ to indicate a term which is $\mathrm{o}(s^\alpha)$ 
as $s\to0$ uniformly for $\zeta\in\gamma$.

To prove \eqref{rp} we note that $R(\zeta,0) P_0 = (-\zeta)^{-1} P_0$, by \eqref{decomp} and the 
relation $S P_0 = 0$. Using Lemma \ref{RminusR} with $u = P_0 v$, we see that 
\begin{equation*}
R(\zeta, s) P_0 - (-\zeta)^{-1} P_0 
= s (-\zeta)^{-1} (\gaD R(\zeta, 0)^*)^* \mathcal{P} \gaD P_0 + \mathrm{O}(s^2)_u,
\end{equation*}
which gives (\ref{rp}) with a slightly better error. (Our errors are 
stated generally as $\mathrm{o}(\cdot)_u$ for consistency.)

For (\ref{eq1.63}) we observe 
\begin{equation*}
(\gamma_D (P_0 R(\zeta, 0))^*)^* = P_0 R(\zeta, 0) \gamma_D^* 
= P_0 (\gamma_D R(\zeta,0)^*)^*,
\end{equation*}
and likewise 
\begin{equation*}
(\gamma_D (P_0 R(\zeta, 0))^*)^* = (\gamma_D ((-\bar{\zeta})^{-1} P_0))^* 
= (-\zeta)^{-1} (\gamma_D P_0)^*,
\end{equation*}
so that 
\begin{equation*}
P_0 (\gamma_D R(\zeta,0)^*)^* = (-\zeta)^{-1} (\gamma_D P_0)^*.
\end{equation*}
If we apply $P_0$ on the left to the identity in Lemma \ref{RminusR},
and use this last relation, we arrive at (\ref{eq1.63}).

For (\ref{prp}) we again take $u = P_0 v$ in Lemma \ref{RminusR},
and we apply $P_0$ on the left of the resulting expression. Finally,
(\ref{7.44.2}) is a straightforward consequence of (\ref{rp})
and (\ref{prp}).

{\em Step 2.}\, We claim the following asymptotic relations for the Riesz projections:
\begin{align}\label{ptp}
P(s) P_0 &= P_0 + s(\gaD S)^* \mathcal{P} \gaD P_0 + \mathrm{o}(s)_u,\\
\label{ppt}
P_0 P(s) &= P_0 + s(\gaD P)^* \mathcal{P} \gaD S + \mathrm{o}(s)_u,\\
\label{ppp}
P_0 P(s) P_0 &= P_0 - s^2(\gaD P)^* \mathcal{P} \gaD(\gaD (S^2))^* \mathcal{P} \gaD P_0 + \mathrm{o}(s^{2})_u.
\end{align}
To see (\ref{ptp}), we integrate (\ref{rp}) with $-\frac{1}{2\pi i}\int_\gamma(\cdot)\,d\zeta$. We find
\begin{equation*}
\begin{aligned}
P(s) P_0 &= -\frac{1}{2\pi i} \int_\gamma (-\zeta)^{-1} d\zeta P_0 
-\frac{s}{2\pi i} \int_\gamma (-\zeta)^{-1} (\gamma_D R(\zeta,0)^*)^* d\zeta \mathcal{P} \gamma_D P_0 
+ \mathrm{o}(s)_u \\
&= P_0 - s \Big(\gamma_D \frac{1}{2\pi i} \int_\gamma (-\zeta)^{-1} R(\zeta,0)^* d\zeta \Big)^* \mathcal{P} \gamma_D P_0 
+ \mathrm{o}(s)_u \\
&= P_0 + s (\gamma_D S^*)^* \mathcal{P} \gamma_D P_0 + \mathrm{o}(s)_u,
\end{aligned}
\end{equation*}
from which (\ref{ptp}) follows because $S$ is self-adjoint. Likewise, 
(\ref{ppt}) and (\ref{ppp}) follow respectively by applying 
$-\frac{1}{2\pi i}\int_\gamma(\cdot)\,d\zeta$ to (\ref{eq1.63})
and (\ref{prp}). 

{\em Step 3.}\, We next claim the following asymptotic relations for 
the transformation operators defined in \eqref{dfnUUinv}:
\begin{align}
U(s) &= I+s \Big((\gaD S)^* \mathcal{P} \gaD P_0 
- (\gaD P_0)^* \mathcal{P} \gaD S \Big) + \mathrm{o}(s)_u, \lb{eq1.72}\\
U(s)^{-1} &= I+s \Big((\gaD P_0)^* \mathcal{P} \gaD S
- (\gaD S)^* \mathcal{P} \gaD P_0) \Big) + \mathrm{o}(s)_u, \label{ut-1} \\
P_0 U(s) P_0 &= P_0 - \frac{1}{2} s^2 (\gaD P_0)^* \mathcal{P} \gaD(\gaD (S^2))^* \mathcal{P} \gaD P_0 
+\mathrm{o}(s^2)_u, \label{pup} \\
\label{u-1}
P_0 U(s)^{-1}P_0 &= P_0 - \frac{1}{2}s^2 (\gaD P_0)^* \mathcal{P} \gaD(\gaD (S^2))^* \mathcal{P} \gaD P_0 
+\mathrm{o}(s^2)_u, \\
P_0 U(s)^{-1}(I-P_0) &= s(\gaD P_0)^* \mathcal{P} \gaD S
+\mathrm{o}(s)_u.\label{7.44.1}
\end{align}

Indeed, recalling that $D(s) = P(s) - P_0$ and using \eqref{ptp} and \eqref{ppt} yields
\begin{align}\label{qq}
D^2(s) &= (P(s)-P_0)(P(s)-P_0)=P(s)+P_0-P(s)P_0-P_0P(s)\\
&=(P(s)-P_0)+(P_0-P(s)P_0)+(P_0-P_0 P(s))
=D(s)-s P^{(1)}+\mathrm{o}(s)_u,
\no
\end{align}
where from Step 2 
\begin{equation*}
(P_0-P(s)P_0)+(P_0-P_0 P(s)) 
= -(\gamma_D S)^* \mathcal{P} \gamma_D P_0 - s (\gamma_D P_0)^* \mathcal{P} \gamma_D S 
+\mathrm{o}(s)_u, 
\end{equation*}
and we define 
\begin{equation*}
P^{(1)} = (\gaD S)^* \mathcal{P} \gaD P_0 
+(\gaD P_0)^* \mathcal{P} \gaD S.
\end{equation*} 
Hence,
\begin{align}\label{qqq}
(I-D(s))(D(s)-s P^{(1)}) = s D(s)P^{(1)}+\mathrm{o}(s)_u=\mathrm{o}(s)_u,
\end{align}
and therefore $D(s)=s  P^{(1)}+(I-D(s))^{-1}\mathrm{o}(s)_u$, yielding
\begin{align}\label{qqqq}
D(s)=s  P^{(1)}+\mathrm{o}(s)_u.
\end{align}

Turning now to (\ref{eq1.72}), we have 
\begin{equation*}
\begin{aligned}
U(s) &= (I - D(s)^2)^{-1/2} \Big((I-P(s))(I-P_0) + P(s) P_0 \Big) \\
&= I - P_0 - P(s) + 2 P(s) P_0 + \mathrm{O} (s^2)_u \\
&= I - P_0 - P(s) + 2P_0 + 2s (\gamma_D S)^* \mathcal{P} \gamma_D P_0 + \mathrm{o}(s)_u \\
&= I - D(s) + 2s (\gamma_D S)^* \mathcal{P} \gamma_D P_0 + \mathrm{o}(s)_u \\
&= I - s P^{(1)} + 2s (\gamma_D S)^* \mathcal{P} \gamma_D P_0 + \mathrm{o}(s)_u \\
&= I - s (\gamma_D S)^* \mathcal{P} \gamma_D P_0 - s(\gamma_D P_0)^* \mathcal{P} \gamma_D S
+ 2 s(\gamma_D S)^* \mathcal{P} \gamma_D P_0 + \mathrm{o}(s)_u \\
&= I + s (\gamma_D S)^* \mathcal{P} \gamma_D P_0 - s(\gamma_D P_0)^* \mathcal{P} \gamma_D S
+ \mathrm{o}(s)_u, 
\end{aligned}
\end{equation*}
which is (\ref{eq1.72}).
Likewise, (\ref{ut-1}) is established by a similar calculation, beginning 
with 
\begin{equation*}
U(s)^{-1} = (I - D(s))^{-1/2} \Big((I-P_0)(I-P (s)) + P_0 P (s) \Big).
\end{equation*} 

Formula \eqref{pup} follows from the calculation
\begin{equation*}
\begin{split}
P_0 U(s) P_0 &= P_0 (I-D^2(s))^{-1/2} P(s) P_0
= P_0 P(s) P_0 + \frac{1}{2}P_0 D(s)^2 P_0 +O(s^3)_u \\
&= P_0 - s^2(\gaD P_0)^* \mathcal{P} \gaD(\gaD (S^2))^* \mathcal{P} \gaD P_0 
+ \frac{1}{2} s^2 P_0 (P^{(1)})^2 P_0 + \mathrm{o}(s^2)_u \\
&= P_0 -s^2 (\gaD P_0)^* \mathcal{P} \gaD(\gaD (S^2))^* \mathcal{P} \gaD P_0 \\
&+ \frac{1}{2} s^2 P_0 \Big(((\gaD S)^* \mathcal{P} \gaD P_0 +
(\gamma_D P_0)^* \mathcal{P} \gaD S\Big)^2 P_0 
+\mathrm{o}(s^2)_u,
\end{split}
\end{equation*}
from which we see that  
\begin{equation*}
\begin{split}
P_0 U(s) P_0 &=
P_0 - s^2 (\gaD P_0)^* \mathcal{P} \gaD(\gaD (S^2))^* \mathcal{P} \gaD P_0 \\
& + \frac{1}{2} s^2 P_0 
\Big{\{} (\gaD S)^* \mathcal{P} \gaD P_0 (\gaD S)^* \mathcal{P} \gaD P_0
+ (\gaD S)^* \mathcal{P} \gaD P_0 (\gaD P_0)^* \mathcal{P} \gaD S \\
&+ (\gaD P_0)^* \mathcal{P} \gaD S (\gaD S)^* \mathcal{P} \gaD P_0
+ (\gaD P_0)^* \mathcal{P} \gaD S (\gaD P_0)^* \mathcal{P} \gaD S
\Big{\}} P_0 + \mathrm{o}(s^2)_u.
\end{split}
\end{equation*}
Three terms are eliminated by the relation $S P_0 = 0$ (first, second, fourth), 
and we also have the identity $S (\gamma_D S)^* = (\gamma_D S^2)^*$. 
Combining these observations, we obtain (\ref{pup}).  

A similar argument yields \eqref{u-1}, and \eqref{7.44.1} follows 
using \eqref{ut-1}.

{\em Step 4.}\, We now claim the following asymptotic relation for the resolvent:
\begin{equation}\lb{finstep}
\begin{split}
P_0 U(s)&^{-1} R(\zeta, s) U(s) P_0 = (-\zeta)^{-1}P_0
+ (-\zeta)^{-2} s(\gaD P_0)^* \mathcal{P} \gaD P_0 \\ 
&- (-\zeta)^{-2} s^2 P_0 V(0)P_0 
- (-\zeta)^{-1} s^2 (\gaD P_0)^* \mathcal{P} \gaD(\gaD (S^2))^* \mathcal{P} \gaD P_0 \\
&+ (-\zeta)^{-2} s^2 (\gaD P_0)^* \mathcal{P} \gaD 
\Big{\{} \gaD \Big[R(\zeta, 0)(I+2(-\zeta) S+(-\zeta)^2 S^2)\Big]^* \Big{\}}^* \mathcal{P} \gamma_D P_0 
+ \mathrm{o}(s^2)_u.
\end{split}
\end{equation}
To see this, we begin by writing 
\begin{equation*}
\begin{aligned}
P_0 U(s)^{-1} R(\zeta, s) U(s) P_0 
&= P_0 U(s)^{-1} 
\Big(P_0 R(\zeta, s) P_0 + (I-P_0) R(\zeta, s) P_0 \\
&+ P_0 R(\zeta, s) (I - P_0) + (I-P_0) R(\zeta, s) (I-P_0) 
\Big) U(s) P_0 \\
&=A_1+A_2+A_3+A_4, 
\end{aligned}
\end{equation*}
where we denote
\begin{align*}
\begin{split}
&A_1 = P_0 U(s)^{-1}P_0 R(\zeta, s)P_0 U(s)P_0 ,\\&
\quad A_2 = P_0 U(s)^{-1}(I-P_0)R(\zeta, s)P_0 U(s)P_0,\\
&\qquad A_3 = P_0 U(s)^{-1}P_0 R(\zeta, s)(I-P_0)U(s)P_0,
\\&\quad\qquad
 A_4=P_0 U(s)^{-1}(I-P_0) R(\zeta, s)(I-P_0)U(s)P_0.
\end{split}
\end{align*}

For $A_1$, we use the fact that $P_0$ is a projection, along with 
\eqref{u-1}, \eqref{prp} and \eqref{pup},  to obtain
\begin{align*}
\begin{split}
A_1 &= (P_0 U(s)^{-1} P_0) (P_0 R(\zeta,s) P_0) (P_0 U(s) P_0) \\
&= \Big(P_0 - \frac{1}{2}s^2 (\gaD P_0)^* \mathcal{P} \gaD(\gaD (S^2))^* \mathcal{P}\gaD P_0
+ \mathrm{o}(s^2)_u \Big) \\
&\quad \times \Big((-\zeta)^{-1}P_0 + (-\zeta)^{-2} s (\gaD P_0)^* \mathcal{P} \gaD P_0 - s^2 (-\zeta)^{-2} P_0 V(0)P_0 \\
&+ s^2 (-\zeta)^{-2} (\gaD P_0)^* \mathcal{P} \gaD(\gaD R(\zeta, 0)^*)^* \mathcal{P} \gaD  P_0
+\mathrm{o}(s^{2})_u \Big) \\
&\quad \times \Big(P_0 - \frac{1}{2}s^2 (\gaD P_0)^* \mathcal{P} \gaD(\gaD (S^2))^* \mathcal{P} \gaD P_0
+\mathrm{o}(s^2)_u \Big) \\
&= (-\zeta)^{-1}P_0 + s (-\zeta)^{-2} (\gaD P_0)^* \mathcal{P} \gaD P_0 
- s^2 (-\zeta)^{-2} P_0 V(0)P_0 \\
&\quad + s^2 (-\zeta)^{-2} (\gaD P_0)^* \mathcal{P} \gaD(\gaD R(\zeta, 0)^*)^* \mathcal{P} \gaD P_0 \\
&\quad - s^2 (-\zeta)^{-1} (\gaD P_0)^* \mathcal{P} \gaD(\gaD (S^2))^* \mathcal{P} \gaD P_0
+\mathrm{o}(s^{2})_u.
\end{split}
\end{align*}

Likewise, it follows from \eqref{7.44.1} and \eqref{7.44.2} that 
\begin{align*}
\begin{split}
A_2 &= (P_0 U(s)^{-1} (I-P_0)) ((I-P_0) R(\zeta,s)P_0) (P_0 U(s) P_0) \\ 
&= \Big(s(\gaD P_0)^* \mathcal{P} \gaD S + \mathrm{o}(s)_u\Big) 
\Big( s(-\zeta)^{-1}(I-P_0)(\gaD R(\zeta, 0)^*)^* \mathcal{P} \gaD P_0 
+\mathrm{o}(s)_u \Big) \\
& \quad \times \Big(P_0 - s^2 \frac{1}{2} (\gaD P_0)^* \mathcal{P} \gaD(\gaD (S^2))^* \mathcal{P} \gaD P_0 
+\mathrm{o}(s^2)_u \Big) \\
&= s^2 (-\zeta)^{-1} (\gaD P_0)^* \mathcal{P} \gaD S(\gaD R(\zeta, 0)^*)^* \mathcal{P} \gaD P_0
+\mathrm{o}(s^2)_u.
\end{split}
\end{align*}

For $A_3$, we have 
\begin{align*}
\begin{split}
A_3 &= (P_0 U(s)^{-1} P_0) (P_0 R(\zeta, s)) ((I-P_0) U(s)P_0 ) \\
&= \Big(P_0 - \frac{1}{2}s^2(\gaD P_0)^* \mathcal{P} \gaD(\gaD (S^2))^* \mathcal{P} \gaD P_0 
+\mathrm{o}(s^2)_u \Big) \\
& \quad \times \Big((-\zeta)^{-1} P_0 + s(-\zeta)^{-1} (\gaD  P_0)^* \mathcal{P} \gaD R(\zeta, 0)
+\mathrm{o}(s)_s \Big) \Big(s(\gaD S)^* \mathcal{P} \gaD P_0 + \mathrm{o}(s)_s \Big) \\
&= s^2 (-\zeta)^{-1} (\gaD P_0)^* \mathcal{P} \gaD R(\zeta, 0)(\gaD S)^* \mathcal{P} \gaD P_0
+\mathrm{o}(s^2)_u,
\end{split}
\end{align*}
while for $A_4$ we have 
\begin{align*}
\begin{split}
A_4&= (P_0 U(s)^{-1}(I-P_0)) R(\zeta, s) ((I-P_0) U(s) P_0 \\
&= s^2(\gaD P_0)^* \mathcal{P} \gaD S R(\zeta, 0)(\gaD S)^* \mathcal{P} \gaD P_0
+\mathrm{o}(s^2)_u.
\end{split}
\end{align*}

Collecting all these terms, we obtain \eqref{finstep}:
\begin{align*}
\begin{split}
P_0 U(s)^{-1}R &(\zeta, s)U(s)P_0
= (-\zeta)^{-1}P_0 + s (-\zeta)^{-2} (\gaD P_0)^* \mathcal{P} \gaD P_0 
- s^2 (-\zeta)^{-2} P_0 V(0) P_0 \\
&+ s^2 (-\zeta)^{-2} (\gaD P_0)^* \mathcal{P} \gaD(\gaD R(\zeta, 0)^*)^* \mathcal{P}\gaD P_0 \\
&\quad - s^2 (-\zeta)^{-1} (\gaD  P_0)^* \mathcal{P} \gaD(\gaD (S^2))^* \mathcal{P} \gaD P_0 
+ \mathrm{o}(s^{2}) \\
&\quad \quad + s^2 (-\zeta)^{-1} (\gaD P_0)^* \mathcal{P} \gaD S(\gaD R(\zeta, 0)^*)^* \mathcal{P} \gaD P_0
+ \mathrm{o}(s^2)_u \\
&\quad \quad \quad + s^2 (-\zeta)^{-1} (\gaD P_0)^* \mathcal{P} \gaD R(\zeta, 0)(\gaD S)^* \mathcal{P} \gaD P_0
+ \mathrm{o}(s^2)_u \\
&\quad \quad \quad \quad + s^2 (\gaD P_0)^* \mathcal{P} \gaD SR(\zeta, 0)(\gaD S)^* \mathcal{P} \gaD P_0 
+ \mathrm{o}(s^2)_u,
\end{split}
\end{align*}
from which the claim is immediate.

{\em Step 5.}\, We are ready to finish the proof of the lemma. 
Using the standard relation from  \cite[Equation (III.6.24)]{Kato} we have
\[
H(s)P(s)=-\frac{1}{2\pi i}\int_\gamma \zeta R(\zeta,s)\,d\zeta,
\]
and applying integration $-\frac{1}{2\pi i}\int_\gamma\zeta(\cdot)\,d\zeta$ in \eqref{finstep}, we
find 
\begin{equation*}
\begin{aligned}
-\frac{1}{2\pi i} &\int_\gamma \zeta P_0 U(s)^{-1} R(\zeta,s) U(s) P_0 \,d\zeta \\
&= -s (\gamma_D P_0)^* \mathcal{P} \gamma_D P_0 + s^2 P_0 V(0) P_0 \\
& \quad - s^2 (\gamma_D P_0)^* \mathcal{P} \gamma_D 
\Big(\gamma_D \frac{1}{2\pi i} \int_\gamma \zeta^{-1} R(\zeta,s)\,d\zeta \Big)^*
\mathcal{P} \gamma_D P_0 + \mathrm{o}(s^2) \\
&= - s (\gaD P_0)^* \mathcal{P} \gaD P_0 + s^2 P_0 V(0) P_0 
- s^2(\gaD P_0)^* \mathcal{P} \gaD(\gaD S)^* \mathcal{P} \gaD P_0 
+ \mathrm{o}(s^2).
\end{aligned}
\end{equation*}
\end{proof}

We now complete our perturbation analysis with the following lemma. 

\begin{lemma}
Under the assumptions of Theorem \ref{main}, we have 
\begin{equation*}
\Mor H(s) = \Mor (B) + \Mor (Q (V(0) - (P_{R_0} \Lambda_0 P_{R_0})^2) Q),
\end{equation*}
for $s > 0$ sufficiently small.
\end{lemma}

\begin{proof}
By Claim \ref{HpvP}, it suffices to count the negative eigenvalues 
of the finite-dimensional operator $H(s)P(s)$. By Lemma \ref{lem:simile}, 
it is enough to obtain an asymptotic formula for the eigenvalues of the operator 
$T(s) := P_0 U(s)^{-1}H(s)P(s)U(s)P_0$, where
\[
T(s)=T+s T^{(1)}+s^2T^{(2)}+\mathrm{o}(s^2)
\text{ as $s\to0$}
\]
and we denote
\begin{align}
T&=0,\quad T^{(1)}=-(\gaD P_0)^* \mathcal{P} \gaD P_0, \quad 
T^{(2)}=T^{(2)}_1+T^{(2)}_2, \nonumber\\
T^{(2)}_1&=P_0 V(0)P_0,\quad T^{(2)}_2=-(\gaD P_0)^* \mathcal{P} \gaD
(\gaD S)^* \mathcal{P} \gaD P_0. \lb{dfnT22}
\end{align}
(For this calculation, we're following \cite{Kato}, along with some 
notation from that reference.)
These operators act on the $d$-dimensional space 
$\ran(P_0)=\ker(H(0)) = (\ker P_{D_0}) \cap (\ker P_{D_1})$. 
We will apply a well known finite-dimensional perturbation result 
\cite[Theorem II.5.1]{Kato} to the family $\{T(s)\}$ for $s$ near zero. 
For this we will need some more notations and preliminaries.

Let $\{\lambda^{(1)}_j\}_{j=1}^{\mathfrak{m}(1)}$ denote the 
$\mathfrak{m} (1)$ distinct eigenvalues of the operator $T^{(1)}$, 
let $m_j^{(1)}$ denote their multiplicities, and let $P^{(1)}_j$ 
denote the respective orthogonal Riesz spectral projections. We
define the bilinear form
\begin{equation*}
\mathfrak{b} (p,q) = (B p,q)_{\mathbb{C}^n}, \quad
\forall p, q \in (\ker P_{D_0}) \cap (\ker P_{D_1}),
\end{equation*}
where we recall that we denote by $B$ the operator obtained by 
restricting $(P_{R_0} \Lambda_0 P_{R_0} - P_{R_1} \Lambda_1 P_{R_1})$ to 
the space $(\ker P_{D_0}) \cap (\ker P_{D_1})$.

The quadratic form on $\ran(P)$ associated with $T^{(1)}$ is given by 
\begin{align*}
\mathfrak{t}^{(1)} (p,q) &= \langle T^{(1)}p,q \rangle_{L^2 (0,1)} 
= - \langle (\gaD P_0)^* \mathcal{P} \gaD P_0 p,q \rangle_{L^2 (0,1)} \\
&= - \Big(\mathcal{P} \gaD P_0 p, \gaD P_0 q \Big)_{\mathbb{C}^{2n}}
= -\Big( \begin{pmatrix} - P_{R_0} \Lambda_0 P_{R_0} & 0 \\ 0 & P_{R_1} \Lambda_1 P_{R_1} \end{pmatrix}
\begin{pmatrix} p \\ p \end{pmatrix}, \begin{pmatrix} q \\ q \end{pmatrix} \Big)_{\mathbb{C}^{2n}} \\
&= \Big( (P_{R_0} \Lambda_0 P_{R_0} - P_{R_1} \Lambda_1 P_{R_1}) p, q \Big)_{\mathbb{C}^n}
= \Big(B p, q \Big)_{\mathbb{C}^n}
=: \mathfrak{b}(p,q).
\end{align*}
In particular, we see that the number of negative values in 
$\{\lambda^{(1)}_j\}_{j=1}^{\mathfrak{m}(1)}$, including multiplicities, 
is $n_- (\mathfrak{b})$ (the number of negative values of $B$, 
including multiplicities), and likewise for the number of positive 
and zero values in $\{\lambda^{(1)}_j\}_{j=1}^{\mathfrak{m}(1)}$
with the respective values $n_+(\mathfrak{b})$ and $n_0(\mathfrak{b})$. 

Turning now to $T^{(2)}$, and following \cite[Section II.5]{Kato}, 
we let $\lambda^{(2)}_{j k}$, $j=1,\dots,\mathfrak{m} (1)$, $k=1,\dots,m^{(1)}_j$, 
denote the eigenvalues of the family of operators $P^{(1)}_j T^{(2)} P^{(1)}_j$ 
in $\ran(P^{(1)}_j)$ (recall that in our case the 
unperturbed operator is just $T=0$ and thus its reduced resolvent is zero 
and $P^{(1)}_j\widetilde{T}^{(2)}P^{(1)}_j=P^{(1)}_jT^{(2)}P^{(1)}_j$ using the 
notations from \cite[Section II.5]{Kato}). By \cite[Theorem II.5.11]{Kato} the 
eigenvalues $\lambda_{jk}(s)$ of the operator $T(s)$ are given by the formula
\beq\lb{eq1.84}
\lambda_{jk}(s)=s  \lambda^{(1)}_j+s^2\lambda^{(2)}_{jk}+\mathrm{o}(s^2)
\,\text{ as $s\to0$}, \, j=1,\dots,\mathfrak{m} (1), k=1,\dots,m^{(1)}_j.
\enq

It's clear from (\ref{eq1.84}) that if $\lambda_j^{(1)} \ne 0$ 
the value of $\lambda_{jk}^{(2)}$ will be 
inconsequential for $s$ sufficiently small. In particular, if 
$\lambda_j^{(1)} < 0$ then $T(s)$ (and hence $H(s)$) will have a 
negative eigenvalue, while if $\lambda_j^{(1)} > 0$ then $T(s)$ 
(and hence $H(s)$) will have a positive eigenvalue. Since our convention
takes the Morse index to be a count of negative eigenvalues, we 
conclude that $\Mor (B)$ is precisely a count of the negative 
eigenvalues of $H(s)$ corresponding with $\lambda_j^{(1)} < 0$.

In the event that $\lambda_j^{(1)} = 0$ we need a sign for $\lambda_{jk}^{(2)}$ 
(which will be non-zero by our non-degeneracy assumption). For notational 
convenience, we index the eigenvalues so that $\lambda^{(1)}_1  = 0$, with 
corresponding Riesz projection $P^{(1)}_1$ onto the $m_1^{(1)}$-dimensional
eigenspace $\ker B$. The corresponding values 
$\{\lambda_{1 k}^{(2)}\}_{k=1}^{m_1^{(1)}}$ will be eigenvalues of 
$T^{(2)}$, and in particular will be precisely the $m_1^{(1)}$ eigenvalues 
of $P^{(1)}_1 T^{(2)} P^{(1)}_1$. For $p,q \in (\ker P_{D_0}) \cap (\ker P_{D_1})$,
we define 
\begin{equation} \label{T2}
\mathfrak{t}^{(2)} (p,q) = \langle P^{(1)}_1 T^{(2)} P^{(1)}_1 p, q \rangle_{L^2 (0,1)}
= \langle P^{(1)}_1 T^{(2)}_1 P^{(1)}_1 p, q \rangle_{L^2 (0,1)}
+ \langle P^{(1)}_1 T^{(2)}_2 P^{(1)}_1 p, q \rangle_{L^2 (0,1)}.
\end{equation}
For the first summand on the right-hand side of (\ref{T2}), we have
\begin{equation} \label{T21}
\langle P^{(1)}_1 P_0 V(0) P_0 P^{(1)}_1 p, q \rangle_{L^2 (0,1)}
=
(P^{(1)}_1 P_0 V(0) P_0 P^{(1)}_1 p, q )_{\mathbb{C}^n}
= (P^{(1)}_1 V(0) P^{(1)}_1 p, q )_{\mathbb{C}^n},
\end{equation} 
where in the first equality we've observed that the $L^2 (0,1)$ inner
product is equivalent to the $\mathbb{C}^n$ inner product for constant 
vectors, and in the second we've observed that since $P_1^{(1)}$ projects 
onto a subspace of $\ran P_0$ we have $P_0 P^{(1)}_1 = P_1^{(1)}$
and $P^{(1)}_1 P_0  = P_1^{(1)}$. 

For the second summand on the right-hand side, we have 
\begin{equation} \label{T22}
\begin{aligned}
\langle P^{(1)}_1 T^{(2)}_2 P^{(1)}_1 p, q \rangle_{L^2 (0,1)} &= 
- \langle P^{(1)}_1 (\gaD P_0)^* \mathcal{P} \gaD (\gaD S)^* \mathcal{P} \gaD P_0 P^{(1)}_1 p, q \rangle_{L^2 (0,1)} \\
&= - \Big(\gaD (\gaD S)^* \mathcal{P} \gaD P_0 P^{(1)}_1 p, \mathcal{P} \gaD P_0 P^{(1)}_1 q \Big)_{\mathbb{C}^{2n}}.
\end{aligned}
\end{equation}
We notice that if we denote $P^{(1)}_1 p = p_1^{(1)} \in \ker B$ then 
\begin{equation*}
\mathcal{P} \gaD P_0 P^{(1)}_1 p
= \begin{pmatrix} - P_{R_0} \Lambda_0 P_{R_0} & 0 \\ 0 & P_{R_1} \Lambda_1 P_{R_1} p \end{pmatrix} 
\begin{pmatrix} p_1^{(1)} \\ p_1^{(1)} \end{pmatrix}
= \begin{pmatrix} - P_{R_0} \Lambda_0 P_{R_0} p_1^{(1)} \\ P_{R_1} \Lambda_1 P_{R_1} p_1^{(1)} \end{pmatrix}.
\end{equation*}
Since $p_1^{(1)} \in \ker B$, we have $P_{R_0} \Lambda_0 P_{R_0} p_1^{(1)} = P_{R_1} \Lambda_1 P_{R_1} p_1^{(1)}$,
so that 
\begin{equation} \label{target_form}
\mathcal{P} \gaD P_0 P^{(1)}_1 p
= \begin{pmatrix} - P_{R_0} \Lambda_0 P_{R_0} p_1^{(1)} \\ P_{R_0} \Lambda_0 P_{R_0} p_1^{(1)} \end{pmatrix}.
\end{equation}
Of course the same calculation hold for $q$ as well. Setting 
\begin{equation*}
\psi = (\gaD S)^* \begin{pmatrix} - P_{R_0} \Lambda_0 P_{R_0} p_1^{(1)} \\ P_{R_0} \Lambda_0 P_{R_0} p_1^{(1)} \end{pmatrix},
\end{equation*}
we see that 
\begin{equation} \label{just_about_there}
\begin{aligned}
\langle P^{(1)}_1 T^{(2)}_2 P^{(1)}_1 p, q \rangle_{L^2 (0,1)} 
&= - \Big(\gaD \psi, \begin{pmatrix} - P_{R_0} \Lambda_0 P_{R_0} q_1^{(1)} \\ P_{R_0} \Lambda_0 P_{R_0} q_1^{(1)} \end{pmatrix} \Big)_{\mathbb{C}^{2n}} \\
&= (\psi (0) - \psi (1), P_{R_0} \Lambda_0 P_{R_0} P_1^{(1)} q)_{\mathbb{C}^n}.  
\end{aligned}
\end{equation}

At this point, we need to understand the action of $\gaD (\gaD S)^*$ on vectors in the form 
on the right-hand side of (\ref{target_form}). This problem has been studied in detail 
in \cite{GM2008} for the case of multiple space dimensions, and the current setting is 
much easier (though a bit different). We will organize the main points of our discussion into
a pair of propositions. 

\begin{proposition} \label{S_star} Suppose $v = {v_1 \choose v_2} \in \mathbb{C}^{2n}$,
with $v_1 \in \ran P_{R_0}$ and $v_2 \in \ran P_{R_1}$. Then 
\begin{equation*}
(\gaD S)^* v = \frac{1}{2 \pi i} \int_{\Gamma} \zeta^{-1} w (x;\zeta) d\zeta,
\end{equation*}
where $\Gamma$ is a small enough loop around $\zeta = 0$ so that it encloses
no other eigenvalues of $H(0)$, and for each $\zeta \in \Gamma$, $w$ is the unique solution to 
$- w'' - \zeta w = 0$, with boundary conditions
\begin{alignat}{2} \label{bc_for_w}
P_{D_0} w (0) &= 0; & \qquad P_{D_1} w (1) &= 0; \notag \\
P_{N_0} w'(0) &= 0; & \qquad  P_{N_1} w'(1) &= 0; \\
P_{R_0} w'(0) &= - v_1; & \qquad P_{R_1} w'(1) &= v_2. \notag
\end{alignat} 
\end{proposition}

\begin{proof}
We note at the outset that by the definition of $S$ as the reduced
resolvent for $H(0)$, we have 
\begin{equation} \label{gaS}
\begin{aligned}
(\gaD S)^* &= S \gaD^* = \frac{1}{2 \pi i} \int_{\Gamma} \zeta^{-1} R(\zeta, 0)\gaD^* d\zeta \\
&= \frac{1}{2 \pi i} \int_{\Gamma} \zeta^{-1} (\gaD R(\zeta, 0)^*)^* d\zeta. 
\end{aligned}
\end{equation}

Let $f \in L^2 (0,1)$ and consider the equation $- u'' - \bar{\zeta} u = f$, with 
boundary conditions 
\begin{alignat*}{2}
P_{D_0} u (0) &= 0; & \qquad P_{D_1} u (1) &= 0; \\
P_{N_0} u'(0) &= 0; & \qquad  P_{N_1} u'(1) &= 0; \\
P_{R_0} u'(0) &= 0; & \qquad P_{R_1} u'(1) &= 0,
\end{alignat*} 
which is solved by $u (x) = R(\zeta,0)^* f$. Notice that for any 
$v \in \mathbb{C}^{2n}$ we can compute 
\begin{equation} \label{first_way}
(\gaD R(\zeta,0)^* f, v)_{\mathbb{C}^{2n}} 
= (\gaD u, v)_{\mathbb{C}^{2n}} 
= (u(0), v_1)_{\mathbb{C}^{n}} 
+ (u(1), v_2)_{\mathbb{C}^{n}}. 
\end{equation} 
On the other hand, 
\begin{equation*}
(\gaD R(\zeta,0)^* f, v)_{\mathbb{C}^{2n}} 
= \langle f, (\gaD R(\zeta,0)^*)^* v \rangle_{L^2 (0,1)}. 
\end{equation*}

Motivated by the analysis of \cite{GM2008}, we set 
\begin{equation*}
w := (\gaD R(\zeta,0)^*)^* v,
\end{equation*}
so that 
\begin{equation} \label{second_way}
\begin{aligned}
\langle f, (\gaD R(\zeta,0)^*)^* v \rangle_{L^2 (0,1)} 
&= \langle - u'' - \bar{\zeta} u, w \rangle_{L^2 (0,1)} \\
&= - (u',w)_{\mathbb{C}^n} \Big|_0^1 + (u,w')_{\mathbb{C}^n} \Big|_0^1
- \langle u, w'' \rangle_{L^2 (0,1)} - \bar{\zeta} \langle u, w\rangle_{L^2 (0,1)}.
\end{aligned}
\end{equation}
In order to eliminate the $L^2 (0,1)$ inner products, we take $w$ to solve 
$- w'' - \zeta w = 0$, and in order to make (\ref{first_way}) correspond
with (\ref{second_way}) we choose the boundary conditions 
(\ref{bc_for_w}).

With this choice of $w$, we have 
\begin{equation*}
\begin{aligned}
(u'(1),w(1))_{\mathbb{C}^n} &= (u'(1), P_{D_1} w(1) + P_{N_1} w(1) + P_{R_1} w(1))_{\mathbb{C}^n} \\
&= (u'(1), P_{D_1} w(1))_{\mathbb{C}^n} + (P_{N_1} u'(1), w(1))_{\mathbb{C}^n}
+ (P_{R_1} u'(1),w(1))_{\mathbb{C}^n} = 0,
\end{aligned}
\end{equation*}
and likewise $(u'(0),w(0))_{\mathbb{C}^n} = 0$. Proceeding by an almost identical 
calculation we find $(u(1) ,w'(1))_{\mathbb{C}^n} = (u(1),v_2)_{\mathbb{C}^n}$
and $(u(0) ,w'(0))_{\mathbb{C}^n} = - (u(0),v_1)_{\mathbb{C}^n}$.

Combining with (\ref{gaS}), we see that the proposition follows.
\end{proof}

Recalling (\ref{target_form}) we see that we need to solve for 
$w$ with $v_1 = - P_{R_0} \Lambda_0 P_{R_0} p_1^{(1)}$ and 
$v_2 = P_{R_0} \Lambda_0 P_{R_0} p_1^{(1)}$. We do this with 
the following proposition.

\begin{proposition} \label{solving_for_w} 
If $v_1 = -v_2$ in (\ref{bc_for_w}), with 
$v_1, v_2 \in (\ran P_{R_0}) \cap (\ran P_{R_1})$,
then 
\begin{equation*}
w (x;0) = v_2 x - \frac{1}{2} v_2.
\end{equation*} 
\end{proposition}

\begin{proof}
First, notice that if we set $\check{w} (x;\zeta) = - w (1-x; \zeta)$,
we find that $w$ and $\check{w}$ solve the same equation, so that 
by uniqueness (for $|\zeta| > 0$ sufficiently small) we have 
\begin{equation} \label{check_relation}
w(x;\zeta) = - w (1-x; \zeta).  
\end{equation}

Next, we set $\tilde{w} = w - v_2 x$, so that
\begin{equation*}
-\tilde{w}'' - \zeta \tilde{w} = \zeta v_2 x, 
\end{equation*}
with homogeneous boundary conditions 
\begin{alignat*}{2}
P_{D_0} \tilde{w} (0) &= 0; & \qquad P_{D_1} \tilde{w} (1) &= 0; \\
P_{N_0} \tilde{w}'(0) &= 0; & \qquad  P_{N_1} \tilde{w}'(1) &= 0; \\
P_{R_0} \tilde{w}'(0) &= 0; & \qquad P_{R_1} \tilde{w}'(1) &= 0.
\end{alignat*} 
We see from Lemma \ref{zero_eigenspace} that $\tilde{w} (x;0)$ is a 
constant function $\tilde{w}_c$, with 
$\tilde{w}_c \in \ker H(0) = (\ker P_{D_0}) \cap (\ker P_{D_1})$.
In this way, we see that 
\begin{equation*}
w (x;0) = \tilde{w}_c + v_2 x,
\end{equation*}
and taking $\zeta \to 0$ in (\ref{check_relation}) we see that 
\begin{equation*}
\tilde{w}_c + v_2 x = - (\tilde{w}_c + v_2 (1-x)),
\end{equation*}
from which we find 
\begin{equation*}
\tilde{w}_c = - \frac{1}{2} v_2,
\end{equation*}
giving precisely the claim.
\end{proof}

Combining Proposition \ref{S_star} with Proposition \ref{solving_for_w}
see that 
\begin{equation*}
\begin{aligned}
\psi (0) - \psi (1) &=
\frac{1}{2 \pi i} \int_{\Gamma} \zeta^{-1} \Big(w (0;\zeta) - w(1;\zeta)\Big) d\zeta \\
&= w(0;0) - w(1;0) = 
= - P_{R_0} \Lambda_0 P_{R_0} P_1^{(1)} p.
\end{aligned}
\end{equation*}
Using (\ref{just_about_there}), we compute
\begin{equation*}
\begin{aligned}
\langle P^{(1)}_1 T^{(2)}_2 P^{(1)}_1 p, q \rangle_{L^2 (0,1)}
&= 
- (P_{R_0} \Lambda_0 P_{R_0} P_1^{(1)} p, P_{R_0} \Lambda_0 P_{R_0} P_1^{(1)} q)_{\mathbb{C}^n} \\
&= - (P_1^{(1)} (P_{R_0} \Lambda_0 P_{R_0})^2 P_1^{(1)} p, q)_{\mathbb{C}^n}.
\end{aligned}
\end{equation*}
Combining with (\ref{T21}), we conclude that 
\begin{equation*}
\langle P^{(1)}_1 T^{(2)} P^{(1)}_1 p, q \rangle_{L^2 (0,1)}
= 
(P_1^{(1)} (V(0) - (P_{R_0} \Lambda_0 P_{R_0})^2) P_1^{(1)} p, q)_{\mathbb{C}^n}.
\end{equation*}

\begin{remark} \label{HandH}
We emphasize that in this section, we have been working with eigenvalues
$\lambda (s)$ of $H(s)$, and as discussed in Remark \ref{HversusH} 
these are related to the eigenvalues $\lambda_s$ of $H_s$ by 
$\lambda_s = \lambda(s)/s^2$. 
\end{remark}

In view of expansion (\ref{eq1.84}), we see that for any $\lambda_j^{(1)} < 0$
we will have $\lambda (s) \sim \lambda_j^{(1)} s$, and so we will have a
crossing along $\Gamma_1$ at $\lambda_{s_0} \sim \lambda_j^{(1)}/{s_0}$.
I.e., {\it each negative eigenvalue of $B$ corresponds with a crossing of 
$\Gamma_1$}. In addition, for $\lambda_1^{(1)} = 0$, if $\lambda_{1k}^{(2)} < 0$
then $\lambda (s) \sim \lambda_{1k}^{(2)} s^2$, and so we will have a 
crossing along $\Gamma_1$ at $\lambda_{s_0} \sim \lambda_{1k}^{(2)}$.
I.e., {\it each negative eigenvalue of 
$P_1^{(1)} (V(0) - (P_{R_0} \Lambda_0 P_{R_0})^2) P_1^{(1)}$ corresponds with a crossing of 
$\Gamma_1$}.
We conclude that 
\begin{equation*}
\Mas (\ell, \ell_1; \Gamma_1) = - \Mor (H(s)) 
= - \Mor(B) - \Mor \Big(Q (V(0) - (P_{R_0} \Lambda_0 P_{R_0})^2) Q \Big),
\end{equation*}     
where for notational convenience we've taken $Q = P_1^{(1)}$ in the 
statement of Theorem \ref{main}, and we use that notation here for
clarity.
\end{proof}

\subsection{Monotoncity in $s$.} \label{monotonicity_in_s} In our proof of Lemma \ref{Gamma3Lemma},
we established that the rotation of the eigenvalues of $\tilde{W}_{s, \lambda}$
is monotonic along $S^1$ as $\lambda$ increases or decreases. This is not 
generally the case as $s$ increases or decreases, but we'll see that it does
hold under certain conditions. In
order to see when this is possible, we differeniate $\tilde{W}_{s,\lambda}$ with 
respect to $s$.
\begin{lemma} Under the assumptions of Lemma \ref{continuity_lemma}, we have
\begin{align}
 \frac{\partial}{\partial s} \tilde{W}_{s,\lambda}
= i \tilde{W}_{s,\lambda} \tilde{\Omega} (s,\lambda),
\end{align}
where
\begin{equation*}
 \tilde{\Omega} (s,\lambda) = 2\Big((X(s,\lambda)-iZ(s,\lambda))^{-1} \tilde{\mathfrak{B}} \Big)^*
 [X^t(V-\lambda I)X-Z^tZ]
\Big( (X(s,\lambda)-iZ(s,\lambda))^{-1} \tilde{\mathfrak{B}} \Big),
\end{equation*}
is a self-adjoint matrix.
\end{lemma}

\begin{proof} First, we recall the notation
\begin{equation*}
\tilde{\mathfrak{B}} = (\beta_1^t \beta_1 - \beta_2^t \beta_2) - i 2 \beta_2^t \beta_1.
\end{equation*}
We begin by computing
\begin{align*}
 &\frac{\partial}{\partial s} \tilde{W}_{s,\lambda}
= (X'+iZ')(X-iZ)^{-1} \tilde{\mathfrak{B}} - (X+iZ)(X-iZ)^{-1}(X'-iZ')(X-iZ)^{-1} \tilde{\mathfrak{B}} \\
& = (X'+iZ')(X-iZ)^{-1} \tilde{\mathfrak{B}} 
- \tilde{W}_{s, \lambda} \tilde{\mathfrak{B}}^{-1} (X'-iZ')(X-iZ)^{-1} \tilde{\mathfrak{B}},
\end{align*}
where $\prime$ denotes differentiation with respect to $s$. Upon multiplication of
both sides by $\tilde{W}_{s, \lambda}^*$, we find 
\begin{align*}
 &\tilde{W}_{s,\lambda}^* \frac{\partial}{\partial s} \tilde{W}_{s, \lambda} 
= \tilde{\mathfrak{B}}^* (X^t+iZ^t)^{-1}(X^t-iZ^t)(X'+iZ')(X-iZ)^{-1}\tilde{\mathfrak{B}} \\
& \quad \quad - \tilde{\mathfrak{B}}^{-1} (X'-iZ')(X-iZ)^{-1} \tilde{\mathfrak{B}} \\
 &= \tilde{\mathfrak{B}}^* (X^t+iZ^t)^{-1} 
\Big[(X^t-iZ^t)(X'+iZ')-(X^t+iZ^t)(X'-iZ')\Big] 
(X-iZ)^{-1} \tilde{\mathfrak{B}} \\
 & = \Big((X-iZ)^{-1} \tilde{\mathfrak{B}}\Big)^* [2iX^tZ'-2iZ^tX'] \Big((X-iZ)^{-1} \tilde{\mathfrak{B}}\Big) \\
 & = i \Big((X-iZ)^{-1} \tilde{\mathfrak{B}}\Big)^* [2X^t(V-\lambda I)X-2Z^tZ] \Big((X-iZ)^{-1} \tilde{\mathfrak{B}}\Big)  
= i \tilde{\Omega}.
\end{align*}

We now multiply both sides by $\tilde{W}_{s, \lambda}$ and use the fact that $\tilde{W}_{s, \lambda}$ is 
unitary to see the claim.
\end{proof}

{\it The Dirichlet case at} $x=1$. In the event that the boundary conditions at $x = 1$ are Dirichlet, the frame 
for our target space is ${0 \choose I}$. Even in this special case, we won't generally have monotonicity in 
$s$, but we'll check that we have monotoncity at crossings. 

Fix $\lambda \in [-\lambda_{\infty}, 0]$ and suppose there is a crossing at $s^* \in (s_0,1)$, so 
that $\tilde{W}_{s^*, \lambda}$ has -1 as an eigenvalue (possibly with multiplicity greater than 1). 
Let $V^*$ denote the eigenspace associated with $-1$, so that 
\begin{equation*}
\tilde{W}_{s^*, \lambda} v = - v \quad \forall v \in V^*,
\end{equation*}
and correspondingly (by the definition of $\tilde{W}_{s^*, \lambda}$) we have 
\begin{equation} \label{X_Y_minus1}
(X(s^*, \lambda) - i Z(s^*, \lambda))^{-1} (X(s^*, \lambda) + i Z(s^*, \lambda)) v = -v,
\end{equation}
so that 
\begin{equation*}
(X(s^*, \lambda) + i Z(s^*, \lambda)) v = - (X(s^*, \lambda) - i Z(s^*, \lambda)) v.
\end{equation*}
Rearranging terms, we see that $X (s^*, \lambda) v = 0$, so that $V^*$ corresponds with the 
null space of $X(s^*, \lambda)$. Moreover, if we set $w = (X - iZ)^{-1} v$ and substitute
$v = (X-iZ) w$ into (\ref{X_Y_minus1}), we see that 
\begin{equation*}
(X(s^*, \lambda) + i Z(s^*, \lambda)) w = - (X(s^*, \lambda) - i Z(s^*, \lambda)) w,
\end{equation*} 
where we've recalled that $(X-iZ)^{-1}$ and $(X+iZ)$ commute. We see that $w$ is also 
in $V^*$, so $(X(s^*, \lambda) - i Z(s^*, \lambda))^{-1}$ maps $V^*$ to $V^*$. 

Recall from our proof of Lemma \ref{W} that the rotation of the eigenvalues of $\tilde{W}$ can 
be determined by the motion of the eigenvalues of 
\begin{equation*}
A_{s, \lambda} := i (e^{i \theta} I - \tilde{W}_{s,\lambda})^{-1} (e^{i \theta} I + \tilde{W}_{s,\lambda}),
\end{equation*}  
for which we've seen 
\begin{equation*}
\frac{\partial}{\partial s} A_{s, \lambda} \Big|_{s = s^*} 
= 2 \Big( (e^{i \theta} I - \tilde{W}_{s^*, \lambda})^{-1} \Big)^*
\tilde{\Omega}_{s^*, \lambda} (e^{i \theta} I - \tilde{W}_{s^*, \lambda})^{-1}.
\end{equation*}

According to the Spectral Mapping Theorem, the eigenvalue $-1$ of $\tilde{W}_{s^*, \lambda}$ 
corresponds with the eigenvalue 
\begin{equation*}
a = i (e^{i \theta} + 1)^{-1} (e^{i \theta} I - 1),
\end{equation*} 
and both eigenvalues correspond with the eigenspace $V^*$. Let $P$ denote projection 
onto this space. According, then, to Theorem II.5.4 in \cite{Kato} the motion of 
$a$ as $s$ varies near $s^*$ is determined by the eigenvalues of 
$P A'_{s^*, \lambda} P$, where prime denotes differentiation with respect to $s$.
In order to get a sign for these eigenvalues, we take any vector $v \in \mathbb{C}^n$
and compute 
\begin{equation*}
\begin{aligned}
(P A'_{s^*, \lambda} P v, v)_{\mathbb{C}^n} &= (A'_{s^*, \lambda} P v, P v)_{\mathbb{C}^n} \\
&= 
2 \Big( \Big( (e^{i \theta} I - \tilde{W}_{s^*, \lambda})^{-1} \Big)^*
\tilde{\Omega}_{s^*, \lambda} (e^{i \theta} I - \tilde{W}_{s^*, \lambda})^{-1} Pv, P v\Big)_{\mathbb{C}^n} \\
&= 
2 \Big(\tilde{\Omega}_{s^*, \lambda} (e^{i \theta} I - \tilde{W}_{s^*, \lambda})^{-1} Pv, 
(e^{i \theta} I - \tilde{W}_{s^*, \lambda})^{-1} P v \Big)_{\mathbb{C}^n}.
\end{aligned}
\end{equation*}
Using $\tilde{W}_{s^*, \lambda} Pv = - Pv$, we arrive at 
\begin{equation*}
(P A'_{s^*, \lambda} P v, v)_{\mathbb{C}^n} =
\frac{2}{|e^{i\theta} + 1|^2} (\tilde{\Omega}_{s^*, \lambda} Pv, Pv)_{\mathbb{C}^n}.
\end{equation*}

We see that we need to determine a sign for the matrix $\tilde{\Omega}_{s^*, \lambda}$,
restricted to the space $V^*$. To this end, we compute (with all evaluations 
at $(s^*, \lambda)$)
\begin{equation*}
\begin{aligned}
(\tilde{\Omega} Pv, Pv)_{\mathbb{C}^n} 
&=
\Big( 2\Big((X-iZ)^{-1} \Big)^*
 [X^t(V-\lambda I)X-Z^tZ]
\Big( (X-iZ)^{-1} \Big) Pv, P v \Big)_{\mathbb{C}^n} \\
&=
2 \Big([X^t(V-\lambda I)X-Z^tZ] (X-iZ)^{-1} Pv, 
(X-iZ)^{-1} P v \Big)_{\mathbb{C}^n}.
\end{aligned}
\end{equation*}
where we've observed that for the Dirichlet case $\tilde{\mathfrak{B}} = I$.
Recalling that $(X-iZ)^{-1}$ maps $V^*$ to $V^*$, and that $V^*$ is the 
kernel of $X$,  we see that 
\begin{equation*}
\Big([X^t(V-\lambda I)X] (X-iZ)^{-1} Pv, 
(X-iZ)^{-1} P v \Big)_{\mathbb{C}^n}
= 0,
\end{equation*} 
and so 
\begin{equation*}
(\tilde{\Omega} Pv, Pv)_{\mathbb{C}^n} 
= 
- 2 \Big(Z^tZ (X-iZ)^{-1} Pv, 
(X-iZ)^{-1} P v \Big)_{\mathbb{C}^n} 
\le 0.
\end{equation*}

We conclude that crossings for the Dirichlet case must proceed in the 
clockwise direction as $s$ increases. (We emphasize that we only 
require Dirichlet conditions at $x = 1$.) In particular, the Maslov index
will always be non-increasing as $s$ increases in this case. (See
Figure \ref{Dirichlet_figure}.) Combining this observation with 
our definition of the Maslov index, we see that in the Dirichlet case
we can write 
\begin{equation*}
\Mor (H) = \sum_{s \in [s_0, 1)} \dim \ker (- \frac{d^2}{dx^2} + s^2 V (sx)).
\end{equation*}

\begin{remark} \label{dirichlet_remark} 
The preceding discussion illuminates the manner in which the current analysis 
is a generalization of the Sturm-Liouville oscillation theorem for $n=1$. We 
see that in the case of Dirichlet conditions at $x=1$, the relation of 
negative eigenvalues to zeros of the eigenfunction associated with $\lambda = 0$
is replaced by a relation of negative eigenvalues to the kernel of $X(s,0)$. Precisely,
we have 
\begin{equation*}
\Mor (H) = \sum_{s \in [s_0, 1)} \dim \ker X(s,0).
\end{equation*}
\end{remark}

\section{Applications} \label{applications_section}

In this section we apply our framework to four illustrative examples. All 
calculations were carried out in MATLAB, and the figures were created 
in MATLAB.

We note at the outset that these calculations have been carried out 
to highlight certain observations in our analysis, and that in practice
Theorem \ref{main} only requires a calculation of the Principal Maslov 
Index (along with some matrix eigenvalues). Such a calculation is quite
straightforward, and for convenient reference, we summarize it here. 

\medskip
{\it Calculation of the Principal Maslov Index.} We construct a frame
$\mathbf{X} = {X \choose Z}$ by solving the ODE system (\ref{first_order})
with initial values ${X \choose Z} = {\alpha_2^t \choose - \alpha_1^t}$. 
We then compute the spectral flow of $\tilde{W}_{s, \lambda}$ through 
the point $(-1,0)$; that is, we count the number of eigenvalues, including 
multiplicities, crossing $(-1,0)$ in the counterclockwise direction, and 
subtract the number crossing $(-1,0)$ in the clockwise direction.

\medskip
\noindent
{\bf Example 1 (Dirichlet Case).} We consider (\ref{eq:hill}) with 
\begin{equation*}
V(x) = 
\begin{pmatrix}
-22 & 10 \sin x \\
x & -20
\end{pmatrix},
\end{equation*}
and Dirichlet boundary conditions specified by $\alpha_1, \beta_1 = I$, 
$\alpha_2, \beta_2 = 0$. In this case, there can be no crossings along 
the bottom shelf, and indeed the only allowable behavior is for the 
eigenvalue curves to enter the box through $\Gamma_2$ and move upward
until exiting through $\Gamma_3$. See Figure \ref{Dirichlet_figure}. 
The Principal Maslov Index in this case is $-2$, and according to 
Theorem \ref{main} this means the Morse index is $2$, consistent with
our figure.  

\begin{figure}[ht] 
\begin{center}\includegraphics[%
  width=10cm,
  height=8cm]{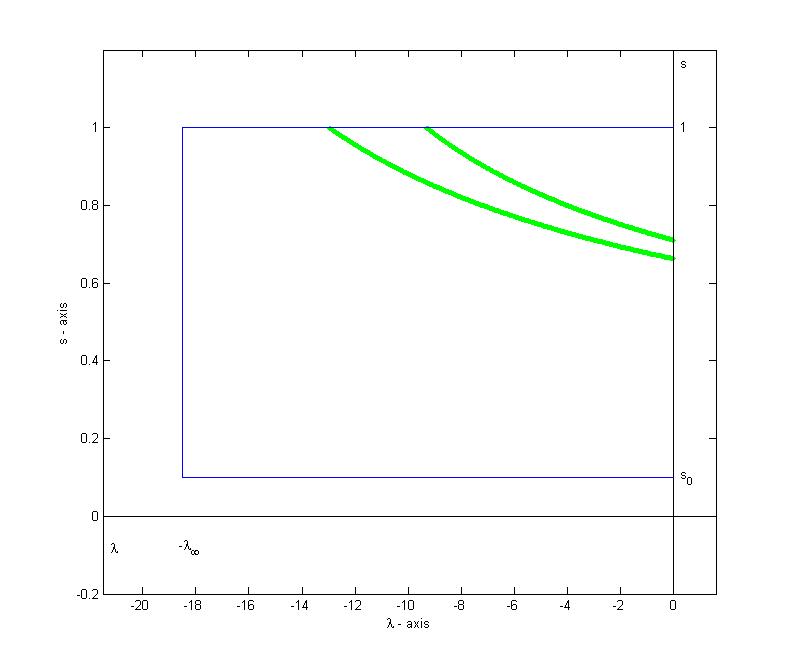}\end{center}
\caption{Eigenvalue curves for Example 1: Dirichlet case. \label{Dirichlet_figure}}
\end{figure}

\medskip
\noindent
{\bf Example 2 (Neumann Case).} We consider (\ref{eq:hill}) with 
\begin{equation*}
V(x) = 
\begin{pmatrix}
-.13 - \frac{.7 \cos(6 \pi x)}{2+\cos(6 \pi x)} & 0 \\
- \frac{\cos(\pi x)}{2+\cos(4 \pi x)} & 1
\end{pmatrix},
\end{equation*}
and Neumann boundary conditions specified by $\alpha_1, \beta_1 = 0$, 
$\alpha_2, \beta_2 = I$. In this case, we see the emergence of 
an eigenvalue from the bottom shelf (corresponding with the second 
order term in our perturbation series), and we notice a very distinct
loss of the monotonicity in $s$ associated with the Dirichlet case.
See Figure \ref{Neumann_figure}. 
The Principal Maslov Index in this case is $0$, and according to 
Theorem \ref{main} the Morse index of $H$ is the Morse index of 
$V(0)$ (because $B = 0$ and $Q = I$). The eigenvalues of $V(0)$
are $-.3633$ and $1$, so that $\Mor (V(0)) = 1$, and indeed we 
see that the eigenvalue emerges from $s = 0$ at $-.3633$.

\begin{figure}[ht] 
\begin{center}\includegraphics[%
  width=10cm,
  height=8cm]{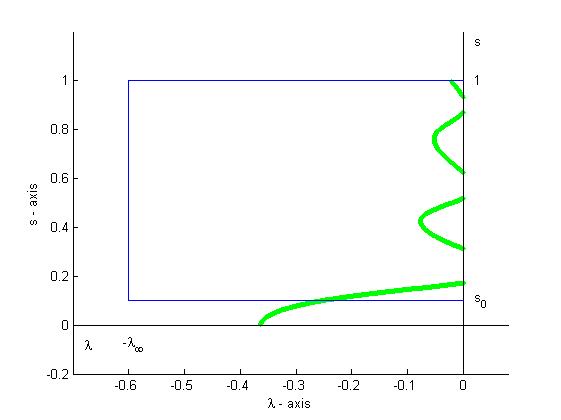}\end{center}
\caption{Eigenvalue curves for Example 2: Neumann case. \label{Neumann_figure}}
\end{figure}

\medskip
\noindent
{\bf Example 3 (Neumann-based Case, I: First Order Perturbation Terms).} 
We consider (\ref{eq:hill}) with 
\begin{equation*}
V(x) = 
\begin{pmatrix}
-13 + 12 x^2 & -7 \cos x \\
- x & -9
\end{pmatrix},
\end{equation*}
and Neumann-based boundary conditions specified by 
$\alpha_1 = \frac{1}{\sqrt{2}} I$, $\beta_1 = 0$, 
$\alpha_2 = \frac{1}{\sqrt{2}} I$, and $\beta_2 = I$. In this case, 
we see an eigenvalue curve entering through $\Gamma_2$, and 
also two curves entering through $\Gamma_1$ (corresponding with the first 
order term in our perturbation series). The Principal Maslov Index 
in this case is $-1$, and according to Theorem \ref{main} the 
contribution from the bottom shelf to the Morse index of $H$ 
will be the Morse index of $B = - \alpha_2^{-1} \alpha_2 = - I$,
which is clearly 2. We conclude that $\Mor (H) = 3$, as indicated
by Figure \ref{Neumann_based1}.

\begin{figure}[ht] 
\begin{center}\includegraphics[%
  width=10cm,
  height=8cm]{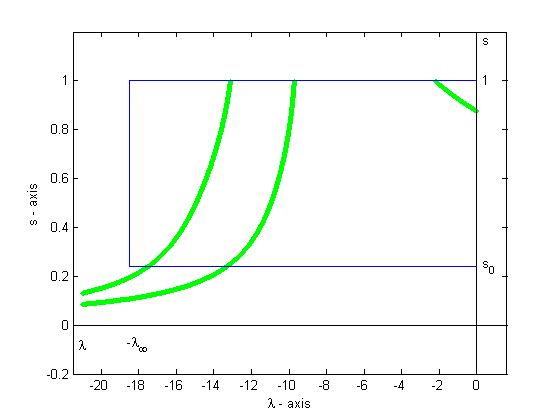}\end{center}
\caption{Eigenvalue curves for Example 3: First Order Perturbation Terms. \label{Neumann_based1}}
\end{figure}

\medskip
\noindent
{\bf Example 4 (Neumann-based Case, II: Second Order Perturbation Terms).} 
We consider (\ref{eq:hill}) with 
\begin{equation*}
V(x) = 
\begin{pmatrix}
-10 - 5 x^2 & - 3 x \\
- 9 \sin x & - 5 - 7 x^2
\end{pmatrix},
\end{equation*}
and Neumann-based boundary conditions specified by 
$\alpha_1 = \frac{1}{\sqrt{2}} I$, $\beta_1 = \frac{1}{\sqrt{2}} I$, 
$\alpha_2 = \frac{1}{\sqrt{2}} I$, and $\beta_2 = \frac{1}{\sqrt{2}} I$. 
In this case, we see an eigenvalue curve entering through $\Gamma_2$, and
two eigenvalue curves entering through $\Gamma_1$ (corresponding with the 
second order term in our perturbation series). The Principal Maslov Index 
in this case is $-1$, and according to Theorem \ref{main} the 
contribution from the bottom shelf to the Morse index of $H$ 
will be the Morse index of $V(0) - (\alpha_2^{-1} \alpha_1)^2$ 
(because $B = 0$ and $Q = I$). The eigenvalues of 
$V(0) - (\alpha_2^{-1} \alpha_1)^2$ are $-11$ and $-6$. We see that the 
Morse index of this matrix is $2$, and indeed that the eigenvalues 
that come in through the bottom shelf originate when $s = 0$ at 
$\lambda = -11$ and $\lambda = -6$. We conclude that the Morse index
of $H$ is 3 in this case, as indicated in Figure \ref{Neumann_based2}.   

\begin{figure}[ht] 
\begin{center}\includegraphics[%
  width=10cm,
  height=8cm]{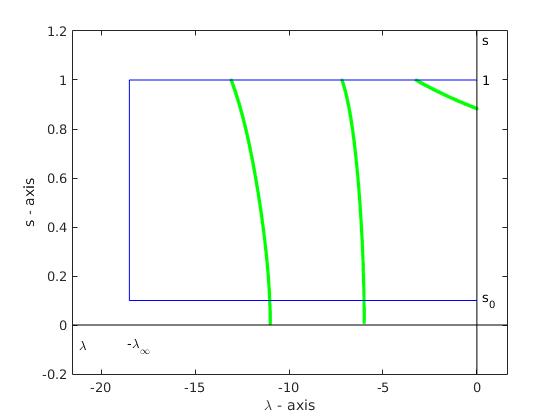}\end{center}
\caption{Eigenvalue curves for Example 4: Second Order Perturbation Terms. \label{Neumann_based2}}
\end{figure}

\bigskip
{\it Acknowledgements.}  The authors are indebted to Gregory Berkolaiko 
for directing them to the elegant formulation of self-adjoint 
boundary conditions in \cite{BK}.

\end{document}